\documentclass[12pt]{amsart}
\usepackage{amsmath, amsfonts, amsbsy, amsthm, amscd, amssymb}
\usepackage{latexsym, ulem}
\usepackage[dvipdfmx]{hyperref}
\usepackage{graphics,graphicx,color}

\numberwithin{equation}{section}

\theoremstyle{plain}
\newtheorem{Theorem}{Theorem}[section]
\newtheorem{Proposition}[Theorem]{Proposition}
\newtheorem{Lemma}[Theorem]{Lemma}
\newtheorem{Corollary}[Theorem]{Corollary}

\newtheorem{Main}{Theorem}

\newtheorem*{Cor}{Corollary}
\newtheorem*{Setting}{Setting}

\theoremstyle{definition}
\newtheorem{Definition}[Theorem]{Definition}
\newtheorem{Remark}[Theorem]{Remark}
\newtheorem{Example}[Theorem]{Example}
\theoremstyle{remark}

\newcommand{\R}{{\mathbb R}}

\newcommand{\Z}{{\mathbb Z}}
\newcommand{\N}{{\mathbb N}}

\newcommand{\nrprj}{\mathrm{proj}}
\newcommand{\tcprj}{\log}
\newcommand{\ene}[1]{E_{#1}}
\newcommand{\inner}[2]{\left\langle #1,\, #2 \right\rangle}
\newcommand{\diam}[1]{\mathrm{diam}(#1)}
\newcommand{\rad}[2]{\mathrm{rad}_{#1}({#2})}
\newcommand{\isom}[1]{\mathrm{Isom}({#1})}
\newcommand{\cat}[1]{\mathrm{CAT}(#1)}
\newcommand{\uulim}[1][]{\lambda{\text{-}}\!{\lim}^{(#1)}}
\newcommand{\ulim}{\lambda{\text{-}}\!\lim}

\newcommand{\p}[1]{p^{(#1)}}
\newcommand{\q}[1]{q^{(#1)}}
\newcommand{\xxi}[1]{\xi^{(#1)}}
\newcommand{\ray}[1]{[#1)}
\newcommand{\cc}[2]{c_{#1}^{#2}}
\newcommand{\tri}[3]{\triangle(#1,#2,#3)}
\newcommand{\bary}[1]{\mathrm{bar}(#1)}
\newcommand{\trans}[2]{\tau_{#1}^{#2}}
\newcommand{\ch}[1]{\mathrm{CH}(#1)}
\newcommand{\sect}[3][]{\mathrm{Sect}_{#1}(#2,#3)}
\newcommand{\cone}[2][]{C_{#1}^{(#2)}}
\newcommand{\zure}{\theta_0}
\newcommand{\supp}{\mathrm{supp}\,}

\makeatletter
\newsavebox\myboxA
\newsavebox\myboxB
\newlength\mylenA

\newcommand*\widebar[2][0.75]{%
    \sbox{\myboxA}{$\m@th#2$}%
    \setbox\myboxB\null
    \ht\myboxB=\ht\myboxA%
    \dp\myboxB=\dp\myboxA%
    \wd\myboxB=#1\wd\myboxA
    \sbox\myboxB{$\m@th\overline{\copy\myboxB}$}
    \setlength\mylenA{\the\wd\myboxA}
    \addtolength\mylenA{-\the\wd\myboxB}%
    \ifdim\wd\myboxB<\wd\myboxA%
       \rlap{\hskip 0.5\mylenA\usebox\myboxB}{\usebox\myboxA}%
    \else
        \hskip -0.5\mylenA\rlap{\usebox\myboxA}{\hskip 0.5\mylenA\usebox\myboxB}%
    \fi}
\makeatother

\begin{document}

\title{Isometric group actions with vanishing rate of escape on
$\cat{0}$ spaces}
\author{Hiroyasu Izeki}
\address{Department of Mathematics, Keio Univeristy,
Kohoku-ku, Yokohama 223-8522, Japan}
\email{izeki@math.keio.ac.jp}
\thanks{This work was supported by JSPS Grants-in-Aid for Scientific
Research Grant Number JP20H01802.}
\date{\today}

\begin{abstract}
 Let $Y=(Y,d)$ be a $\cat{0}$ space which is either proper or of finite
 telescopic dimension, and $\Gamma$ a countable group equipped
 with a symmetric and nondegenerate probability measure $\mu$. 
 Suppose that $\Gamma$ acts on $Y$ via a homomorphism 
 $\rho\colon \Gamma \rightarrow \isom{Y}$, where $\isom{Y}$ denotes the
 isometry group of $Y$, 
 and that the action given by $\rho$ has finite second moment
 with respect to $\mu$.
 We show that if $\rho(\Gamma)$ does not fix a point in the boundary at
 infinity $\partial Y$ of $Y$ and the rate of escape 
 $l_{\rho}(\Gamma)=l_{\rho}(\Gamma,\mu)$ associated to an action given
 by $\rho$ vanishes, then there exists a flat subspace in $Y$ that is
 left invariant under the action of $\rho(\Gamma)$. 
 Note that if the rate of escape does not vanish, then we know that
 there exists an equivariant map from the Poisson boundary of
 $(\Gamma,\mu)$ into the boundary at infinity of $Y$ by a result of
 A.~Karlsson and G.~Margulis. 
 The key ingredient of the proof is $\mu$-harmonic functions on $\Gamma$
 and $\mu$-harmonic maps from $\Gamma$ into $Y$. 
 We prove a result similar to the above for an isometric action of
 $\Gamma$ on a locally finite-dimensional $\cat{0}$ space. 
\end{abstract}

\maketitle

\section{Introduction}
\label{sec:introduction}

Let $\Gamma$ be a countable group equipped with a probability
measure $\mu$. 
We regard $\mu$ as the law of a random walk; for each $\gamma\in\Gamma$,
we regard $\mu(\gamma)$ as the transition probability of a move from the
identity element $e$ to $\gamma$. 
Then the $n$-fold convolution $\mu^n$ of $\mu$ tells us the distribution
of the position of the random walk, generated by $\mu$ and starting from
$e$, at time $n$. 

Suppose that $\Gamma$ acts on a metric space $Y=(Y,d_Y)$ via a
homomorphism $\rho\colon \Gamma \rightarrow \isom{Y}$, where $\isom{Y}$
denotes the isometry group of $Y$. 
Take any point $p\in Y$ and consider an orbit map 
$f_p\colon \gamma \mapsto \rho(\gamma)p$. 
Then we can transplant the random walk generated by $\mu$ to 
the orbit $f_p(\Gamma)=\rho(\Gamma)p\subset Y$. 
The behavior of this transplanted random walk depends on how
$\Gamma$ acts on $Y$.  When $Y$ is a nonpositively curved
metric space, one can ask the following convergence question: 
for which action do almost all sample paths of the transplanted random
walk converge to the ``boundary at infinity'' $\partial Y$ of $Y$?
If almost all sample paths converge to the boundary at infinity,
then this means that any orbit map induces a canonical equivariant map from
the Poisson boundary $\partial (\Gamma,\mu)$, a probability space that
describes the distribution of the position of a random walk on $\Gamma$
generated by $\mu$ at ``time infinity'', into $\partial Y$.  
Such a boundary map is known to be very useful tool in rigidity theory 
as we can see, for example, in a pioneering work due to H.~Furstenberg
discussed in \cite{furstenberg}, and a superrigidity result due to
G.~Margulis \cite{margulis1}, \cite{margulis2}. 

There have been many works related to the convergence question above.
Lie groups and their subgroups acting on associated symmetric spaces are
classical and well-understood ones (see \cite{benoist-quint}). 
V.~Kaimanovich and H.~Masur \cite{kaimanovich-masur} treated mapping class
groups acting on Teichm\"uller spaces and showed that, under certain
conditions, the Poisson boundary of a mapping class group can be
realized on the Thurston boundary of corresponding Teichm\"uller space. 
Also a theorem due to Kaimanovich \cite{kaimanovich} says that it is the
case for hyperbolic-type groups acting on themselves, and that the
Poisson boundary of these groups
can be realized on the boundary at infinity of themselves.
Then, for example, F.~Gautero and F.~Math\'eus \cite{gautero-matheus}
treated certain groups acting on $\R$-trees, and J.~Maher and G.~Tiozzo
\cite{maher-tiozzo} treated groups acting on separable Gromov hyperbolic
spaces. Also groups acting on $\cat{0}$ cubical complexes were
considered by T.~Fern\'os, J.~L\'ecureux and F.~Math\'eus
\cite{fernos-lecureux-matheus}. 
Recently Chois \cite{chois} treated groups containing elements with
bounded geodesic image property acting on various nonpositively curved
spaces. The reader can find more works on the question in the references
of papers cited above.

On the other hand, there has been a known criterion for the
convergence to the boundary for groups acting on general nonpositively
curved metric spaces due to A.~Karlsson and G.~Margulis
\cite{karlsson-margulis}.
Consider the {\it rate of escape}
$l_{\rho}(\Gamma)=l_{\rho}(\Gamma,\mu)$ of an action given by $\rho$
which turns out to be equal to 
\begin{equation*}
 l_{\rho}(\Gamma) = \lim_{n\to \infty} 
  \frac{\int_{\Gamma} d_Y(p,\rho(\gamma)p) d\mu^n(\gamma)}{n}
\end{equation*}
(see \S~\ref{sec:rate_of_escape_and_hmap}). 
This limit is known to exist, and since the numerator of the right-hand
side is the average of the distance between the starting point and the
position at time $n$ of the transplanted random walk, it measures the
speed of the transplanted random walk's escape to the boundary at
infinity $\partial Y$ of $Y$.
Indeed, for a nonpositively curved metric space $Y$, 
A.~Karlsson and G.~Margulis proved in \cite{karlsson-margulis} that if
$l_{\rho}(\Gamma)>0$, then 
almost all sample paths of the transplanted random walk
leave every bounded subset fast enough and go almost straight toward
$\partial Y$; 
almost all sample paths converge to points in $\partial Y$.
However, it is usually difficult to check the positivity of
$l_{\rho}(\Gamma)$. For example, in
\cite{fernos-lecureux-matheus} and \cite{maher-tiozzo}, the positivity
of $l_{\rho}(\Gamma)$ was obtained as a consequence of the convergence
to the boundary.

If $l_{\rho}(\Gamma)=0$, then the transplanted random
walk is rather wandering around in $Y$.  
Thus one might think that information obtained from such a random
walk should contain lots of noise, and that it is hard to
understand.  
Indeed, it seems that there has been no known efficient method to handle 
this case in full generality.
However, under mild assumptions on $Y$ and $\rho$, our approach using 
harmonic functions and harmonic maps associated to $\mu$ and $\rho$
enables us to specify such an action as stated in
Theorem~\ref{thm:main-1} below; an action given by $\rho$ with
$l_{\rho}(\Gamma)=0$ turns out to be the simplest one in the sense that
it is reduced to that on a Euclidean space. 

\begin{Main}
\label{thm:main-1}
 Let $Y$ be a complete $\cat{0}$ space which is either proper or
 of finite telescopic dimension, 
 and $\Gamma$ a countable group equipped with a symmetric and 
 nondegenerate probablity measure $\mu$.
 Let $\Gamma$ act on $Y$ via a homomorphism 
 $\rho\colon \Gamma \rightarrow \isom{Y}$, and 
 suppose that $\rho(\Gamma)$ does not fix a point in 
 $\partial Y$ and that the action given by $\rho$ has finite
 second moment with respect to $\mu$.  
 If  $l_{\rho}(\Gamma,\mu)=0$, then there exists a flat subspace $F$ in
 $Y$, which is a convex subset of $Y$ isometric to $\R^n$ $(n\geq 0)$,
 left invariant under the action of $\rho(\Gamma)$. 
\end{Main}

The case $n=0$ in the theorem above should be understood as that
$\rho(\Gamma)$ fixes a point in $Y$. 
Let us briefly explain the assumptions in Theorem~\ref{thm:main-1}. 
(See \S~\ref{sec:harmonic_maps} for precise definitions and more
details.)
To obtain sufficient information of actions of $\Gamma$ from the
random walk generated by $\mu$, we assume that 
$\mu$ is symmetric and nondegenerate.
Namely, $\mu$ satisfies $\mu(\gamma^{-1})=\mu(\gamma)$, and the support
$\supp{\mu}$ of $\mu$ generates $\Gamma$. 
We also assume that the action given by $\rho$ has finite second
moment with respect to $\mu$:
\begin{equation*}
 \int_{\Gamma} d_S(p,\rho(\gamma)p)^2 d\mu(\gamma)< \infty
\end{equation*}
holds for $p \in Y$; 
it is easy to see that, by the triangle inequality
and the Schwarz inequality, this does not depend on the choice of 
$p \in Y$.  If $\Gamma$ is finitely generated and $\mu$ has finite
second moment with respect to some (hence any) word metric on $\Gamma$,
then any $\rho \colon \Gamma \rightarrow \isom{Y}$ gives the action with
finite second moment (see \S~\ref{sec:harmonic_maps}). 
A metric space under consideration is a nonpositively curved metric
space called $\cat{0}$ space. 
An important feature of $\cat{0}$ spaces on which our approach relies
is the convexity of distance functions. 
We also assume that a $\cat{0}$ space $Y$ is either proper or of finite
telescopic dimension in order to provide $Y$ with a certain compactness
property. 
Here a $\cat{0}$ space $Y$ is called {\it proper} if every closed metric
ball $B(p,r)$, $p\in Y$, $r>0$, is compact, while a $\cat{0}$ space $Y$
is said to have {\it telescopic dimension} at most $n$
if every asymptotic cone of $Y$ has geometric dimension at most $n$; 
the notion of geometric dimension and telescopic dimension was 
introduced by B.~Kleiner in \cite{kleiner} and by
P.-E.~Caprace and A.~Lytchak in \cite{caprace1}, respectively. 
(See \S~\ref{sec:finite_dim_properties} for more details.)
The assumption that $\rho(\Gamma)$ does not fix a point in $\partial Y$
guarantees the existance of a $\rho$-equiavariant $\mu$-harmonic map 
$f \colon \Gamma \rightarrow Y$. 
Here, a $\rho$-equivariant map $f\colon \Gamma \rightarrow Y$ is called 
$\mu$-{\it harmonic} if $f$ minimizes the energy $\ene{\mu}$ defined by
\begin{equation*}
 \ene{\mu}(f)=
 \frac{1}{2} \int_{\Gamma} d_Y(f(e),f(\gamma))^2 d\mu(\gamma)
\end{equation*}
among all $\rho$-equivariant maps. 
As equivariant maps from the Poisson boundary, equivariant harmonic
maps are also known to be useful tools in rigidity theory. 
For example, K.~Corlette \cite{corlette}, 
N.~Mok, Y.-T.~Siu and S.-K.~Yeung \cite{mok-siu-yeung}, 
J.~Jost and S.-T.~Yau \cite{jost-yau1}, \cite{jost-yau2}, and 
M.~Gromov and R.~Schoen \cite{gromov-schoen} used the existence
of nonconstant equivariant harmonic maps from Riemannian symmetric
spaces in order to prove superrigidity results for lattices in Lie
groups, where the local geometric structure of the domains of harmonic
maps has key information to determine the images of these harmonic
maps and deduce superrigidity results. 
Also discrete equivariant harmonic maps from the vertex set of a
simplicial complex or a finitely generated group itself were 
treated by,
for example, M.~Gromov \cite{gromov}, M.-T.~Wang \cite{wang}, and 
H.~Kondo, S.~Nayatani and the author \cite{izeki-nayatani},
\cite{izeki-kondo-nayatani1}, \cite{izeki-kondo-nayatani2},
\cite{izeki}, where various fixed-point theorems for isometric actions
of finitely generated groups on $\cat{0}$ spaces were obtained from the
existence of constant equivariant harmonic maps. 
While in the present paper, we prove Theorem~\ref{thm:main-1} by using
the existence of a nonconstant equivariant harmonic map defined on a
countable group. 
Typically, the image of our harmonic map is nowhere dense in the
target space or even in the convex hull of the image; the space we need
to specify its structure is far from being covered by the image. 
In order to prove our theorem, however, we need to fill this huge
complement. 
We achieve this by using the convexity of distance
functions of the target space effectively. 

 The outline of the proof of Theorem~\ref{thm:main-1} is as follows.  
 Since $Y$ is either proper or of finite telescopic dimension, under the
 assumption that $\rho(\Gamma)$ does not fix a point in $\partial Y$,
 one sees that there exists a $\rho$-equivariant $\mu$-harmonic map 
 $f\colon\Gamma \rightarrow Y$ (see \S~\ref{sec:harmonic_maps} and
 Proposition~\ref{thm:existence_finite_teledim}). 
 A $\mu$-harmonic map has a property that the pull-back of a convex
 function $u$ on the target space $Y$ becomes a $\mu$-subharmomnic
 function on $\Gamma$, that is, $-\Delta f^*u(\gamma)\geq 0$ holds for
 any $\gamma \in \Gamma$.
 Here, one may regard the value 
 $|-\Delta f^*u(\gamma)|$ as a quantity that measures the strength of
 the convexity of a function $u$ around $f(\gamma)$. 
 We note here that a distance function $p\mapsto d(p,p_0)$ is a convex
 function on $Y$, since $Y$ is a $\cat{0}$ space, 
 and that the strength of the convexity of a distance function should be
 related to the strength of the nonpositivity (negativity) of the
 ``curvature'' of $Y$. 
 Now, under the assumption that $l_{\rho}(\Gamma)=0$, we can 
 show that there exists a sequence $\{\gamma_j\}\subset \Gamma$ such that 
 the pull-back of $u_j\colon p \mapsto d(p, f(\gamma_j))$ 
 approaches a $\mu$-harmonic function on $\Gamma$. 
 By our assumption that $Y$ is either proper or of finite telescopic
 dimension, after a suitable normalization, we obtain a limit function
 $u$ of $\{u_j\}$ which is either a distance function from a point in
 $Y$ or a Busemann function associated to a point in $\partial Y$
 (see \S~\ref{sec:rate_of_escape_and_hmap} and
 \S~\ref{sec:global_harmonicity}). 
 Then the pull-back $f^*u$ of the limit function $u$ is $\mu$-harmonic,
 that is, $-\Delta f^*u(\gamma)=0$ holds for any $\gamma \in \Gamma$,
 and hence $u$ has the weakest convexity on the convex hull
 $\ch{f(\Gamma)}$ of $f(\Gamma)$. 
 This suggests that the nonpositivity of the curvature of
 $\ch{f(\Gamma)}$ is the weakest one, namely, $\ch{f(\Gamma)}$ must be
 a flat subspace in $Y$. 

 Combining Theorem~\ref{thm:main-1} with a theorem due to A.~Karlsson
 and G.~Margulis mentioned above, we immediately obtain the following
 theorem, which is a refinement of a part of a theorem due to U.~Bader,
 B.~Duchesne, and J.~L\'ecureux 
 \cite[Theorem 1.1]{bader-duchesne-lecureux}. 

\begin{Main}
\label{thm:main-2}
 Let $Y$ be a complete $\cat{0}$ space which is either proper or
 of finite telescopic dimension, and $\Gamma$ a countable group
 with a symmetric and nondegenerate probability measure $\mu$.
 Let $\rho\colon \Gamma \rightarrow \isom{Y}$ be a homomorphism and 
 suppose that $\rho(\Gamma)$ does not fix a point in 
 $\partial Y$ and that the action given by $\rho$ has finite
 second moment with respect to $\mu$.
 Then either of the following is true. 

  $(1)$ An orbit map induces a canonical $\rho$-equivariant map from the
 Poisson boundary $\partial (\Gamma,\mu)$ of $\Gamma$ into the boundary
 $\partial Y$ of $Y$. 

 $(2)$ There exists a flat subsapce of $Y$ left invariant under the
 action of $\rho(\Gamma)$. 
\end{Main}

 For example, if $\Gamma$, equipped with the discrete topology,
 has Kazhdan's property $(T)$, then 
 $\Gamma$ is finitely generated and any
 isometric action of $\Gamma$ on a Hilbert space (a Euclidean space) has
 a fixed point (see for example \cite{bekka-delaharpe-valette}). 
 Therefore we obtain the following. 

\begin{Cor} Let $\Gamma$ be a discrete group with Kazhdan's
 property $(T)$, and $\mu$ a symmetric and nondegenerate probability
 measure $\mu$ on $\Gamma$ with finite second moment. 
 Let $Y$ be a complete $\cat{0}$ space which is either proper or
 of finite telescopic dimension, and 
 $\rho\colon \Gamma \rightarrow \isom{Y}$ a homomorphism. 
 Suppose that $\rho(\Gamma)$ does not fix a point in $Y\cup \partial Y$.
 Then any orbit map induces a canonical $\rho$-equivariant map from the
 Poisson boundary $\partial (\Gamma,\mu)$ of $\Gamma$ into the boundary
 $\partial Y$ of $Y$. 
\end{Cor}

Since a group $\Gamma$ with Kazhdan's property $(T)$ is known to
have an ``entropy gap'', the assertion of the corollary above says that 
there is certain restrictions on isometric actions of $\Gamma$ on
$\cat{0}$ spaces as it has been suggested in many works from various
viewpoints.

If $Y$ is locally finite-dimensional $\cat{0}$ space, we can show a
slightly weaker result. 
We assume that $\rho(\Gamma)$ is 
reductive in the sense of Jost
(see \S~\ref{sec:harmonic_maps} for the definition), which is stronger
than the assumption that there exists a point in $\partial Y$ fixed by
$\rho(\Gamma)$; we need this assumption to deduce the existence of a
$\rho$-equivariant $\mu$-harmonic map from $\Gamma$ into $Y$. 

\begin{Main}
\label{thm:main-3}
 Let $Y$ be a complete, locally finite-dimensional $\cat{0}$ space,  
 and $\Gamma$ a countable group equipped with a symmetric and 
 nondegenerate probablity measure $\mu$. 
 Let $\Gamma$ act on $Y$ via a homomorphism 
 $\rho\colon \Gamma \rightarrow \isom{Y}$, and 
 suppose that $\rho(\Gamma)$ is 
 reductive in the sense of Jost and that the action given by
 $\rho$ has finite second moment with respect to $\mu$.
 If  $l_{\rho}(\Gamma,\mu)=0$,  then there exists a flat subspace
 of $Y$ left invariant under the action of $\rho(\Gamma)$. 
\end{Main}

 It is clear that an analogue of Theorem~\ref{thm:main-2} holds for a 
 locally finite-dimensional $\cat{0}$ space $Y$. 

\begin{Main}
\label{thm:main-4}
 Let $Y$ be a complete, locally finite-dimensional $\cat{0}$ space,
 and $\Gamma$ a countable group
 with a symmetric and nondegenerate probability measure $\mu$.
 Let $\rho\colon \Gamma \rightarrow \isom{Y}$ be a homomorphism and 
 suppose that $\rho(\Gamma)$ is 
 reductive in the sense of Jost and that the action given by
 $\rho$ has finite second moment with respect to $\mu$.
 Then either of the following is true. 

  $(1)$ An orbit map induces a canonical $\rho$-equivariant map from the
 Poisson boundary $\partial (\Gamma,\mu)$ of $\Gamma$ into the boundary
 $\partial Y$ of $Y$. 

 $(2)$ There exists a flat subsapce of $Y$ left invariant under the
 action of $\rho(\Gamma)$. 
\end{Main}

 We also point out that, as a byproduct of our result, we obtain
 a certain gap theorem on the energy growth of equivariant harmonic
 maps. 
 For a $\rho$-equivariant map $f\colon \Gamma \rightarrow Y$, 
 $n$-step energy $\ene{\mu^n}(f)$ is defined by
\begin{equation*}
 \ene{\mu^n}(f) = 
 \frac{1}{2} \int_{\Gamma} d_Y(f(e),f(\gamma))^2 d\mu^n(\gamma). 
\end{equation*}
 It is shown in \cite{izeki-kondo-nayatani1} that, 
 for any $\rho$-equivariant map into a $\cat{0}$ space $Y$,  
 $\ene{\mu^n}(f)\leq Cn^2\ene{\mu}(f)$ holds for some $C>0$, and 
 that if $f$ is a $\rho$-equivariant $\mu$-harmonic map into a Hilbert
 space (a Euclidean space), then 
\begin{equation*}
 \ene{\mu^n}(f) = n \ene{\mu}(f)
\end{equation*}
 holds. 
 On the other, if $l_{\rho}(\Gamma)>0$, then the Schwarz inequality
 tells us that 
\begin{equation*}
 \liminf_{n\to \infty} \frac{\ene{\mu^n}(f)}{n^2} > 0.
\end{equation*}
 Therefore we obtain: 

\begin{Main}
\label{thm:main-5}
 Let $Y$ be a $\cat{0}$ space which is either proper, of
 finite telescopic dimension, or of locally finite dimension.
 And let $\Gamma$ be a countable group with a symmetric and
 nondegenerate probability measure $\mu$.
 Suppose that the action given by a homomorphism 
 $\rho\colon \Gamma \rightarrow \isom{Y}$ has finite second moment with
 respect to $\mu$ and that
 there exists a nonconstant $\rho$-equivariant
 $\mu$-harmonic map $f$. 
 Then either $\ene{\mu^n}(f)=n\ene{\mu}(f)$ or 
 $\ene{\mu^n}(f)\asymp n^2$ holds. 
\end{Main}

We note here that it was pointed out and proved in \cite{gromov} and
\cite{izeki-kondo-nayatani2} that if there exist $\varepsilon>0$ and a
natural number $n$ such that every $\rho$-equivariant map $f$ satisfies
$\ene{\mu^n}(f)\leq (n-\varepsilon)\ene{\mu}(f)$, then there exists a
constant $\rho$-equivariant harmonic map, and hence $\rho(\Gamma)$ fixes
a point in $Y$, which is used in \cite{izeki-kondo-nayatani2} and
\cite{izeki} to prove a fixed-point property for random groups. 

This paper is organized as follows. 
In \S~\ref{sec:harmonic_maps}, we recall fundamental definitions and
facts on $\cat{0}$ spaces, and some facts on equivariant harmonic maps. 
We first finish the proof of Theorem~\ref{thm:main-1} for proper
$\cat{0}$ spaces in \S~\ref{sec:rate_of_escape_and_hmap},
\S~\ref{sec:convexity_of_distance_and_Busemann_fct}, and
\S~\ref{sec:proper_case}. 
In \S~\ref{sec:rate_of_escape_and_hmap}, under the assumption
$l_{\rho}(\Gamma)=0$ together with the existence of an equivariant
$\mu$-harmonic map $f$, we show that the pull-back of a distance
function or a Busemannn function by $f$ becomes a (locally)
$\mu$-harmonic function; they are obtained as a limit of distance
functions from a sequence of points. 
The properness of a $\cat{0}$ space guarantees that the limit
function can be regarded as either a distance function or a Busemann
function. 
Then, in \S~\ref{sec:convexity_of_distance_and_Busemann_fct}, we prove two
inequalities reflecting the convexity of metric balls and horoballs
respectively; the equality case plays a key role in the recognition of a
flat subspace in Theorem~\ref{thm:main-1}. 
In \S~\ref{sec:proper_case}, using these results, we prove
Theorem~\ref{thm:main-1} for proper $\cat{0}$ spaces. 

The proofs of Theorem~\ref{thm:main-1} for nonproper $\cat{0}$ spaces
and that of Theorem~\ref{thm:main-3} are given in
\S~\ref{sec:when-y-not} and \S~\ref{sec:key-lemma-proof}. 
In order to apply our arguments to nonproper $\cat{0}$ spaces, we need a
sort of compactness theorem (Theorem~\ref{thm:compactness}) for
$\cat{0}$ spaces with certain dimension bounds. 
In \S~\ref{sec:when-y-not}, we present a quick review of the notion of
ultralimit and consider ultralimits of $\cat{0}$ spaces with certain
dimension bounds.  
Assuming the compactness theorem mentioned above, 
we finish the proof of Theorem~\ref{thm:main-1}. 
The proof of Theorem~\ref{thm:main-3} is a straightforward variant of
that of Theorem~\ref{thm:main-1} 
(see \S~\ref{sec:proof_of_thm_main-3}). 
Then \S~\ref{sec:key-lemma-proof} is devoted to the proof of
Theorem~\ref{thm:compactness}. 
In the course of the proof of Theorem~\ref{thm:main-1}, we need to
replace a probability measure $\mu$ on $\Gamma$ with one whose support
covers all non-trivial elements of $\Gamma$. 
It is achieved by taking a convex combination of the $n$-fold
convolution $\mu^n$'s of $\mu$. 
In \S~\ref{sec:appendix}, we show that this replacement does not affect
our assumption $l_{\rho}(\Gamma)=0$ in Theorem~\ref{thm:main-1}. 

The author would like to thank Uri Bader, Atsushi Katsuda, Shin
Nayatani, Takefumi Kondo, Ryokichi Tanaka, Masato Mimura, and Tomohiro
Fukaya for their comments and interests in this work. 

\section{Equivariant harmonic maps into $\cat{0}$ spaces}
\label{sec:harmonic_maps}

%
%
In this section, we give some fundamental definitions and
facts on $\cat{0}$ spaces and equivariant harmonic maps from 
countable groups into $\cat{0}$ spaces. 

Let $Y$ be a metric space with metric $d_Y$, and $I\subset \R$ an
interval. 
A map $c\colon I \rightarrow Y$ is called a {\it geodesic} if $c$ is
an isometric embedding, namely, $c$ satisfies $d_Y(c(t),c(t'))=|t-t'|$
for any $t,t' \in I$. 
If $I$ is a bounded closed interval, we often call $c$ a 
{\it geodesic segment}, while if $I=\ray{0,\infty}$, we call $c$ a 
{\it geodesic ray}.  
We call a geodesic defined on whole $\R$ a {\it geodesic line}. 
If every pair of points in $Y$ is joined by a geodesic segment, then $Y$
is called a {\it geodesic space}. 
In what follows, all metric spaces under consideration are assumed to be
complete as metric spaces. 

A geodesic space $Y$ is called $\cat{0}$ space if any geodesic triangle
is thinner than that in the Euclidean plane in the following sense: 
Let $p_1,p_2,p_3 \in Y$, and denote the image of geodesic joining $p_i$
and $p_j$ by $[p_i,p_j]$. 
We call $[p_1,p_2]\cup [p_2,p_3]\cup [p_3,p_1]$ a {\it geodesic
triangle}, and denote by $\tri{p_1}{p_2}{p_3}$. 
Note that we can take 
a triangle $\tri{\overline{p_1}}{\overline{p_2}}{\overline{p_3}}$ in
$\R^2$ so that $d_Y(p_i, p_j)=d_{\R^2}(\overline{p_i},\overline{p_j})$,
since the side lengths of $\tri{p_1}{p_2}{p_3}$ satisfy the triangle
inequality. 
We call $\tri{\overline{p_1}}{\overline{p_2}}{\overline{p_3}}$ a 
{\it comparison triangle} for $\tri{p_1}{p_2}{p_3}$.  
And we call a point 
$\overline{q}\in [\overline{p_i},\overline{p_j}]$
a {\it comparison point} for $q\in [p_i,p_j]$ if 
$d_Y(p_i,q) = d_{\R^2}(\overline{p_i}, \overline{q})$.
Then we say that $\tri{p_1}{p_2}{p_3}$ satisfies the 
{\it $\cat{0}$ inequality}
if $d_Y(q_1,q_2) \leq d_{\R^2}(\overline{q_1}, \overline{q_2})$ holds
for any pair of points $q_1, q_2$ on the sides of $\tri{p_1}{p_2}{p_3}$
and their comparison points $\overline{q_1}, \overline{q_2}$. 
If every geodesic triangle in $Y$ satisfies the $\cat{0}$ inequality, 
then $Y$ is called a {\it $\cat{0}$ space}. 
A complete, simply
connected Riemannian manifold of nonpositive sectional curvature is a
typical example of $\cat{0}$ space.  Also trees and Hilbert spaces are
$\cat{0}$ spaces. 
If one replaces the Euclidean plane, the present location of the
comparison triangle, with the $2$-dimensional space form of curvature
$\kappa$, we obtain the notion of $\cat{\kappa}$ space, which can be
regarded as a metric space whose curvature is bounded from above by
$\kappa$. 
For a detailed exposition on $\cat{\kappa}$ spaces, we refer the reader
to \cite{bridson-haefliger}. 

%
%
If $Y$ is a $\cat{0}$ space, then it is clear that any pair of distinct
points is joined by a unique geodesic, and in particular, a $\cat{0}$
space turns out to be contractible. 
Furthermore, it is not hard to see that the distance function on a
$\cat{0}$ space becomes a convex function, an important feature of a
$\cat{0}$ space: 
For any geodesic $c_i\colon [0,1]\rightarrow Y$, $i=1,2$, parametrized
proportional to arc length (possibly a constant map), a function 
$t \mapsto d(c_1(t),c_2(t))$ is a convex function, that is, for any 
$t \in [0,1]$, 
 \begin{equation*}
  d_Y(c_1(t), c_2((t)) \leq 
   (1-t)d_Y(c_1(0),c_2(0))  + td_Y(c_1(1),c_2(1))
 \end{equation*}
holds (see \cite[p.~176, 2.2 Proposition]{bridson-haefliger}). 
For a $\cat{0}$ space $Y$ the following is a part of the list of its
properties more or less directly derived from the convexity
above and the definition of $\cat{0}$ spaces. 
We will frequently use them in what follows.  
\begin{itemize}
 \item Recall that $C \subset Y$ is called a {\it convex} subset if any
       pair of points in $C$ can be joined by a geodesic lying in $C$.  
       For a closed convex subset $C$ of $Y$ and $p \in Y$, there exists
       a unique point $q \in C$ such that $d_Y(p,C)=d_Y(p,q)$ holds. 
       Thus we can define the {\it nearest point projection} $\nrprj_C$
       that sends $p \in Y$ to this $q \in C$; $\nrprj_C$ turns out to
       be a distance nonincreasing map. 
       (See (\cite[p.~176, 2.4 Proposition]{bridson-haefliger}) for the
       detail.)
 \item Let $C$ be a closed convex subset of $Y$. 
        If $\nrprj_C$ restricted to $C'$ becomes an isometry onto its
       image, then the convex hull of $\nrprj_C(C')\cup C'$ is isometric
       to a product $C' \times [0,l]$. 
       (See \cite[p.~182, 2.12 Exercise (3)]{bridson-haefliger}.)
 \item If two geodesic lines $c_i \colon \R \rightarrow Y$, $i=1,2$,
       are parallel to each other, that is, $c_1(\R)$ stays in a bounded
       distance from $c_2(\R)$ and vice versa, then the convex hull of
       $c_1(\R)\cup c_2(\R)$ is isometric to $\R \times [0,l]$. 
       (See \cite[p.~182, 2.13 The Flat Strip Theorem]{bridson-haefliger}.)
 \item Let $c\colon \R \rightarrow Y$ be a geodesic line. 
       The union of the images of geodesic lines parallel to $c$ forms a
       closed convex subset $Y_c$ of $Y$, and $Y_c$ is isometric to a
       product metric space $\R \times Y_0$ for some convex subset 
       $Y_0 \subset Y_c$. 
       (See \cite[p.~183, 2.14 A Product Decomposition
       Theorem]{bridson-haefliger}. )
\end{itemize}
From the viewpoint of the convexity of the distance function, 
the Euclidean space $\R^n$, whose nonpositivity of the curvature is the
least among the $\cat{0}$ spaces, should be characterized as a $\cat{0}$
space with the ``weakest'' convexity. 
In order to prove the existence of a flat subspace in
Theorem~\ref{thm:main-1}, we need to develope a quantitative
approach to the recognition of the subset with the weakest convexity. 
It is plausible that, for example, the existence of solutions to a sort
of ``equation'' leads us to the goal;
indeed, it is achieved by considering harmonic maps and harmonic
functions as we will see. 
In order to realize this, a tool like Riemannian metric provides us with
a great help; it enables us to produce various quantitative expressions
of the convexity (see, for example, 
Propositions~\ref{prop:convexity_of_distance_sphere} and
\ref{prop:convexity_of_horosphere}, and remarks following them).  
An ``inner product'' defined below on the tangent cone 
plays such a role in this paper. 

%
%
\begin{Definition}
\label{tangent cone}
Let $Y$ be a $\cat{0}$ space.  

$(1)$  Let $c$ and $c'$ be two non-trivial geodesics in $Y$ starting 
from $p \in Y$. The {\it angle} $\angle_p(c,c')$ between $c$ and $c'$ 
is defined by 
\begin{equation*}
 \angle_p(c,c')= \lim_{t,t' \rightarrow 0}
 \angle_{\overline{p}}(\overline{c(t)},\overline{c'(t')}),
\end{equation*}
 where
 $\angle_{\overline{p}}(\overline{c(t)},\overline{c'(t')})$ 
 denotes the angle between the sides
 $[\overline{p},\overline{c(t)}]$ and $[\overline{p},\overline{c'(t')}]$
 of the comparison triangle
 $\tri{\overline{p}}{\overline{c(t)}}{\overline{c'(t')}}\subset\R^2$.
 For $p,q \in Y$, denote by $\cc{p}{q}$ a geodesic starting from $p$ and
 teminating at $q$. 
 We often use an expression $\angle_p(q,q')$ for the angle 
 $\angle_p(\cc{p}{q},\cc{p}{q'})$ between $\cc{p}{q}$ and $\cc{p}{q'}$. 

$(2)$  Let $p \in Y$.
 For a pair of non-trivial geodesics $c$, $c'$ starting from $p$, we 
 define a relation $\sim$ by $c\sim c'$ if and only if 
 $\angle_p(c,c')=0$. 
 Then $\sim$ becomes an equivalence relation on the set of
 all non-trivial geodesics starting from $p$ denoted by 
 $\widetilde{(S_pY)^{\circ}}$. 
 Then the angle $\angle_p$ induces a metric on the quotient 
 $(S_pY)^{\circ} = \widetilde{(S_pY)^{\circ}}/\sim$, 
 which we denote by the same symbol $\angle_p$.  The completion 
 $(S_pY,\angle_p)$ of the metric space $((S_pY)^{\circ}, \angle_p)$ is
 called the {\it space of directions}  at $p$.

$(3)$ Let $TC_pY$ be the cone over $S_pY$, namely,
\begin{equation*}
 TC_pY = (S_pY \times \R_{\geq 0}) / (S_pY \times \{0\}). 
\end{equation*}
 Let $W,W' \in TC_pY$. We may write $W=(V,t)$ and $W'=(V',t')$, where
 $V,V' \in S_pY$ and $t,t' \in \R_{\geq 0}$.   Then
\begin{equation*}
 d_{TC_pY}(W, W')= t^2 + {t'}^2 - 2tt'\cos \angle_p(V,V')
\end{equation*}
 defines a metric on $TC_pY$. The metric space $(TC_pY, d_{TC_pY})$ is
 known to be a $\cat{0}$ space and is called the {\it tangent cone} at
 $p$.
 We define an ``inner product''  on  $TC_pY$ by
\begin{equation*}
 \inner{W}{W'} = tt'\cos \angle_p(V,V').
\end{equation*}
 We often denote the length $t$ of $W$ by $|W|$; thus we have
 $|W|=\sqrt{\langle W,W \rangle}=d_{TC_pY}(O_p,W)$, where $O_p$ denotes
 the origin of $TC_pY$.   
 In what follows, we identify $S_pY$ with the set of elements
 with the length equals one in $TC_pY$: 
 $S_pY = \{W \in TC_pY \mid |W|=1\}$. 

$(4)$ Define a map $\tcprj_p \colon Y \longrightarrow TC_pY$ by
 $\tcprj_p(q)=([c], d_Y(p,q))$, where $c$ is a geodesic starting from
 $p$ and terminating at $q$, and $[c]\in S_pY$ is the equivalence class
 of $c$. Then $\tcprj_p$ is distance nonincreasing.  
\end{Definition}

\begin{Remark}
 For a $\cat{0}$ space $Y$, the space of direction $S_pY$ is known to be
 a $\cat{1}$ space (\cite[p.~191, 3.19 Theorem]{bridson-haefliger}). 
 Note that, by definition, any isometry $\Phi\colon Y \rightarrow Y$
 induces an isometry $\Phi_* \colon S_pY \rightarrow S_{\Phi(p)}Y$ and 
 $\Phi_*\colon TC_pY \rightarrow TC_{\Phi(p)}Y$. 
\end{Remark}

%
%
While the tangent cone describes a local structure of a $\cat{0}$
space $Y$, the {\it boundary at infinity} defined below tells us an
asymptotic nature of $Y$; for example, 
the boundary $\partial Y$ of $\cat{0}$ space $Y$
equipped with the angular metric (see below for the definition)
describes the distribution of the flat subspaces in $Y$. 

\begin{Definition}
\label{defn:boundary}
 Let $Y$ be a $\cat{0}$ space.

$(1)$ Two geodesic rays $c_i\colon \ray{0,\infty}\rightarrow Y$,
 $i=1,2$, are {\it asymptotic} to each other if there exists
 $M\in\ray{0,\infty}$ such that $d(c_1(t),c_2(t))\leq M$ for all 
 $t \in \ray{0,\infty}$. 
 Being asymptotic is easily seen to be an equivalence relation on the set
 of geodesic rays. The {\it boundary} $\partial Y$ of $Y$ is defined to
 be the set of equivalence classes (asymptotic classes) of geodesic
 rays. 
 The topology on $Y \cup \partial Y$ naturally induced from 
 the uniform convergence of geodesics and geodesic rays on each bounded
 interval is called the {\it cone topology} of $Y \cup \partial Y$. 

$(2)$ Let $\xi_i \in \partial Y$.  Denote by $\cc{p}{\xi_i}$ a geodesic
 ray starting from $p$ and representing $\xi_i$ (we will say
 ``terminating at $\xi_i$'' and express as
 $\cc{p}{\xi_i}(\infty)=\xi_i$).  
 Let us denote the angle between
 $\cc{p}{\xi_1}$ and $\cc{p}{\xi_2}$ by $\angle_p(\xi_1,\xi_2)$. 
 Then
\begin{equation*}
 d_{\angle}(\xi_1,\xi_2)=\sup_{p \in Y} \angle_p(\xi_1,\xi_2)
\end{equation*}
 gives a metric on $\partial Y$ called the {\it angular metric}. 

$(3)$  Let $c\colon \ray{0,\infty}\rightarrow Y$ be a geodesic ray. 
 The {\it Busemann function} $b_c$ associated to $c$ is defined by
\begin{equation*}
 b_c (x) =\lim_{t \to \infty} d(c(t),x)-t
\end{equation*}
 The level set $S_c(r)=\{p \in Y \mid b_{c}(p)=r\}$ is called a
 {\it horosphere} centered at $\xi=c(\infty)$ (see
 Remark~\ref{rem:busemann_fct} below). 
 The sublevel set $B_c(r)=\{p \in Y \mid b_{c}(p)\leq r\}$ is
 called a {\it horoball} centered at $\xi=c(\infty)$. 
\end{Definition}

\begin{Remark}
 Since any isometry preserves the relation asymptotic, the isometry
 group $\isom{Y}$ of $Y$ acts on $\partial Y$.  
 The action is by homeomorphisms with respect to the cone topology. 
 On the other hand, it is clear, by definition, the action is by
 isometry with respect to the angular metric on $\partial Y$. 
\end{Remark}

\begin{Remark}
 \label{rem:proper_implies_cpt}
 A metric space $Y=(Y,d_Y)$ is called {\it proper} if every
 closed metric ball is compact. 
 If a $\cat{0}$ space $Y$ is proper, then one sees that, by using 
 Arzel\`a-Ascoli theorem, 
 $Y\cup \partial Y$ equipped with the cone topology is compact. 
\end{Remark}

\begin{Remark}
 For a $\cat{0}$ space $Y$, the boundary $(\partial Y, d_{\angle})$ with
 the angular metric is a $\cat{1}$ spaces
 (see \cite[p.~285, 9.13 Theorem]{bridson-haefliger}). 
 The topology on $\partial Y$ given by the angular metric is stronger
 than the cone topology restricted to $\partial Y$ in general. 
 For example, for a hyperbolic $n$-space $\mathbb{H}^n$,  
 $\partial \mathbb{H}^n$ with the cone topology is homeomorphic to
 $S^{n-1}$,  while $(\partial \mathbb{H}^n,d_{\angle})$ is easily seen
 to be a discrete metric space, since every pair of points in
 $\partial\mathbb{H}^n$ can be joined by a geodesic line. 
\end{Remark}

\begin{Remark}
 \label{rem:sector}
 Let $\xi, \xi' \in \partial Y$ with $d_{\angle}(\xi,\xi')<\pi$. 
 It is known that if there exists $p_0$ such that 
 $\angle_{p_0}(\xi,\xi')=d_{\angle}(\xi,\xi')$, then the convex hull of 
 $\cc{p_0}{\xi}(\ray{0,\infty})\cup \cc{p_0}{\xi'}(\ray{0,\infty})$ is
 isometric to a flat sector, which is a sector in the Euclidean plane
 bounded by a pair of half lines that meet at their endpoints with an
 angle $\angle(\xi,\xi')$ 
 (\cite[p.~283, 9.9 Corollary]{bridson-haefliger}).
 On the other hand, 
 by \cite[p.~281, 9.8 Proposition (2)]{bridson-haefliger}, one sees that
 if the convex hull of $\ray{p,\xi}\cup \ray{p,\xi'}$ is isometric to a
 flat sector, then $d_{\angle}(\xi,\xi')=\angle_p(\xi,\xi')$ holds. 
\end{Remark}

\begin{Remark}
\label{rem:busemann_fct}
 Let $\xi \in \partial Y$ and $\{p_n\}\subset Y$ any sequence that
 converges to $\xi$ in the cone topology. Fix $p_0 \in Y$. 
 Consider a function $p\mapsto b_{\xi}(p,p_0)$ defined by
\begin{equation}
\label{eq:def_of_busemann}
  b_{\xi}(p,p_0)= \lim_n d(p_n,p)-d(p_n,p_0). 
\end{equation}
 The limit in the right-hand side exists and does not depend on the
 choice of a sequence $\{p_n\}$ converging to $\xi$ 
 (see \cite[p.269, 8.19 Corollary]{bridson-haefliger}, see also
 Lemma~\ref{lem:ultralimit_of_buseman_fct}). 
 Therefore we see that $b_{\xi}(p)=b_c(p)$ for a geodesic ray $c$ with 
 $c(0)=p_0$ and $c(\infty)=\xi$;
 we treat (\ref{eq:def_of_busemann}) as the definition of the Busemann
 function, and mainly use the symbol $b_{\xi}(\cdot,p_0)$ in what
 follows. 
 Note that, by definition and the description above,  
 $\lim_{t\to \infty}b_{\xi}(c(t),p_0)\to -\infty$ for any geodesic ray
 terminating at $\xi$. 
 Note also that the definition (\ref{eq:def_of_busemann}) above shows that
 the Busemann function has an additivity:
 \begin{equation}
  \label{eq:additivity}
  b_{\xi}(p,p_0)=b_{\xi}(p,p_0')+ b_{\xi}(p_0',p_0)
 \end{equation}
 for any $p,p_0,p_0' \in Y$. 
 In particular, the change of the basepoint $p_0$ leaves the Busemann
 function invariant up to a constant. 
 Thus we often denote the horosphere and the horoball centered at $\xi$
 by $H_{\xi}(r)$ and $B_{\xi}(r)$ without specifying the basepoint. 
 Since a Busemann function $b_{\xi}(\cdot,p_0)$ is the limit of a
 sequence of convex functions,  $b_{\xi}(\cdot,p_0)$ is also a convex
 function. 
 As a consequnece, a horoball $B_{\xi}(r)$, as well as a metric ball,
 becomes a convex subset of $Y$. 
\end{Remark}

%
%
Let $\Gamma$ be a countable group, and $\mu$ a probability
measure on $\Gamma$.
We assume that the support of $\mu$ generates $\Gamma$, 
and that $\mu$ is symmetric, that is, 
$\mu$ satisfies $\mu(\gamma)=\mu(\gamma^{-1})$ for any 
$\gamma \in \Gamma$.
Suppose that $\Gamma$ acts on a metric space $Y$ isometrically via a
homomorphism $\rho\colon \Gamma \rightarrow \isom{Y}$, where $\isom{Y}$
is the isometry group of $Y$. 
We assume that the action given by $\rho$ has 
{\it finite second moment} with respect to $\mu$, namely, for $p\in Y$, 
\begin{equation*}
 \int_{\Gamma} d_Y(p, \rho(\gamma)p)^2 d\mu(\gamma) < \infty
\end{equation*}
holds. 
Since the action given by $\rho$ is by isometry, using the triangle
inequality, we see that
\begin{equation*}
\begin{split}
  d_Y(p', \rho(\gamma)p')^2 
 & \leq \left(d_Y(p',p) + d_Y(p,\rho(\gamma)p) + d_Y(\rho(\gamma)p,
       \rho(\gamma)p')\right)^2 \\
 & = \left(2d_Y(p',p) + d_Y(p,\rho(\gamma)p) \right)^2 \\
 & = 4d_Y(p',p)^2 + 4d_Y(p',p)d_Y(p,\rho(\gamma)p) + 
      d_Y(p,\rho(\gamma)p)^2. 
\end{split}
\end{equation*}
Applying the Schwarz inequality to the second term in the last
expression, we see that it is also integrable with respect to $\mu$; the
integrability does not depend on the choice of $p \in Y$. 
A map $f \colon \Gamma \rightarrow Y$ is said to be 
$\rho$-{\it equivariant} if $f(\gamma \gamma')=\rho(\gamma)f(\gamma')$
holds for any $\gamma, \gamma' \in \Gamma$. 
The $\mu$-energy $\ene{\mu}(f)$ of a $\rho$-equivariant map $f$ defined
as
\begin{equation*}
 \ene{\mu}(f) = \frac{1}{2} \int_{\Gamma} d_Y(f(e), f(\gamma))^2
d\mu(\gamma)
\end{equation*}
is well-defined by the assumption that the action given by $\rho$ has
finite second moment with respect to $\mu$.

Now suppose that $\Gamma$ is finitely generated.
Then we have a finite generating set $S$ of $\Gamma$.  
We assume that $S$ is symmetric; $s \in S$ implies that $s^{-1}\in S$. 
The {\it word length} $\mathrm{length}_S(\gamma)$ of $\gamma \in \Gamma$
is defined as the length of the shortest expression of $\gamma$ by
generators belonging to $S$:
\begin{equation*}
 \mathrm{length}_S(\gamma)= \min \{ k \in \N \mid \gamma=s_1 \dots s_k,\
  s_i \in S\}. 
\end{equation*}
Then $d_S(\gamma, \gamma')=\mathrm{length}_S(\gamma^{-1}\gamma')$
becomes a metric on $\Gamma$, and called the {\it word metric} on
$\Gamma$ with respect to a finite generating set $S$. 
Note that, by definition, $d_S(\cdot,\cdot)$ is a left-invariant metric: 
$d_S(\gamma \gamma', \gamma \gamma'')=d_S(\gamma',\gamma'')$.  
For finite generating sets $S$ and $S'$, although $d_S$ and $d_{S'}$ may
be different, it is easy to see that they are Lipschitz equivalent. 
We say that $\mu$ has finite second moment, if $\mu$ satisfies 
\begin{equation*}
 \int_{\Gamma} d_S(e,\gamma)^2 d\mu(\gamma) < \infty, 
\end{equation*}
where $e \in \Gamma$ denotes the identity element in $\Gamma$, and 
$d_S$ is a word metric on $\Gamma$ defined above. 
Note that the property of having finite second monent does not depend on
the choice of a finite generating set $S$ ($S$ may not be contained in
$\supp\mu$).  
Taking $u(\cdot)$ to be $d(f(e),\cdot)$ in the following lemma,  
we see that if $\Gamma$ is finitely generated and 
$\mu$ has finite second moment, then any isometric action of $\Gamma$ on
a metric space has finite second moment.

\begin{Lemma}
 \label{lem:integral}
 Let $\Gamma$ be a finitely generated group with a probability measure
 $\mu$ with finite second moment. Suppose $\Gamma$ acts on a metric
 space $Y$ via a homomorphism $\rho\colon \Gamma\rightarrow \isom{Y}$. 
 Suppose further that a function $u\colon Y \rightarrow \R$ satisfies 
 $|u(p)| \leq C_1 d(p_0,p)^2 + C_2$ for some $C_1,C_2>0$ and 
 $p_0 \in Y$. Then, for any $\rho$-equivariant map 
 $f\colon \Gamma \rightarrow Y$, the pull-back $f^*u$ of $u$ by $f$
 belongs to $L^1(\Gamma,\mu)$.  
 In particular, $f^*u$ is a $\mu$-integrable function on $\Gamma$. 
\end{Lemma}

\begin{proof}
Fix a finite generating set $S$ of $\Gamma$. 
Then, by setting $C=\max\{d_Y(f(e), f(s))\mid s \in S \}$, 
for $\gamma = s_1 \dots s_k$ with $s_i \in S$ 
($i=1, \dots, k$), we have
\begin{equation*}
\begin{split}
  & d_Y(f(e),f(s_1 \dots s_k)) \\
 \leq & d_Y(f(e), f(s_1)) + d_Y(f(s_1), f(s_1 s_2))
 + \dots + 
   d_Y(f(s_1 \dots s_{k-1}), f(s_1 \dots s_k)) \\
 = & d_Y(f(e), f(s_1)) + d_Y(f(e), f(s_2))
 + \dots + d_Y(f(e), f(s_k)) 
 \leq  Ck, 
\end{split}
\end{equation*}
because the action given by $\rho$ is isometric and $f$ is
$\rho$-equivariant. 
This implies that $d_Y(f(e),f(\gamma)) \leq C d_S(e,\gamma)$.
Therefore $f^*u$ satisfies
\begin{equation*}
 \begin{split}
  & |f^*u(\gamma)| = |u(f(\gamma))| \\
 \leq & C_1 \left(d(p_0, f(e)) + d(f(e),f(\gamma)) \right)^2 + C_2 \\
 = & C_1 \left(d(f(e),f(\gamma))^2 + 
   2 d(p_0,f(e))d(f(e),f(\gamma)) + d(p_0,f(e))^2\right) + C_2 \\
 \leq &  C_1'd_S(e,\gamma)^2 + C_1''d_S(e,\gamma) + C_2''
\end{split}
\end{equation*}
 for some $C_1', C_1'', C_2''>0$. 
 Thus, by applying the Schwarz inequality to the second term in
 the last expression and by our assumption of $\mu$ having
 finite second moment, we see that 
\begin{equation*}
 \int_{\Gamma} |f^*u(\gamma)| d\mu(\gamma) < \infty. 
\end{equation*}
 This completes the proof. 
\end{proof}

%
%
A $\rho$-equivariant map $f$ is called a $\mu$-{\it harmonic map} if $f$
minimizes the $\mu$-energy $\ene{\mu}(f)$: 
$\ene{\mu}(g)\geq \ene{\mu}(f)$ holds for any $\rho$-equivariant map 
$g \colon \Gamma \rightarrow Y$. 

Note that a $\rho$-equivariant map is completely determined by the image
of $e$; the set of $\rho$-equivariant map can be identified with $Y$. 
We denote by $f_p$ a $\rho$-equivariant map defined by 
$f_p(\gamma)=\rho(\gamma)p$. 
Recall that a function $u \colon Y \rightarrow \R$ is said to be
{\it convex} if, for any geodesic $c\colon [0,l]\rightarrow Y$, 
$t \mapsto u(c(t))$ is a convex function on $[0,l]$. 
Take any geodesic $c \colon [0,1]\rightarrow Y$ parametrized
proportional to arc length. 
Since the distance funtion of a $\cat{0}$ space $Y$ is nonnegative and
convex, its square $t \mapsto d_Y(c(t),\rho(\gamma)c(t))^2$ is also
convex, and hence
\begin{equation*}
\begin{split}
   \ene{\mu}(f_{c(t)}) 
 & =\frac{1}{2} \int_{\Gamma} d_Y(c(t),\rho(\gamma)c(t))^2 d\mu(\gamma) \\
 & \leq \frac{1}{2}\int_{\Gamma} (1-t)d_Y(c(0),\rho(\gamma)c(0))^2
  + t d_Y(c(1),\rho(\gamma)c(1))^2 d\mu(\gamma) \\
 & = (1-t)\ene{\mu}(f_{c(0)}) + t \ene{\mu}(f_{c(1)}), 
\end{split}
\end{equation*}
which means that $p\mapsto \ene{\mu}(f_p)$ is a convex function on
$Y$.  Then we have the following. 

\begin{Proposition}
 \label{prop:first_variation}
 Let $Y$ be a $\cat{0}$ space and 
 $\rho\colon \Gamma\rightarrow \isom{Y}$ a homomorphism. 
 Let $f\colon \Gamma \rightarrow Y$ be a $\rho$-equivariant map. 
 Then $f$ is $\mu$-harmonic if and only if $f$ satisfies the following
 inequality$\,:$ For any $V \in S_{f(e)}Y$, 
\begin{equation}
 \label{eq:first_variation}
 \int_{\Gamma} \inner{V}{\tcprj_{f(e)}(f(\gamma))} d\mu(\gamma)\leq 0. 
\end{equation}
\end{Proposition}

We use the following lemmas in the proof of 
Proposition~\ref{prop:first_variation}.

\begin{Lemma}{{\rm (\cite[p.~185, 3.6 Corollary]{bridson-haefliger})}}
\label{lem:angle}
 Let $Y$ be a $\cat{0}$ space and $p \in Y$. Take two non-trivial
 geodesics $c$, $c'$ starting
 from $p$, and fix a point $q$ on $c'$.  Then we have
\begin{equation*}
  \cos \angle_p (c,c') = \lim_{t \rightarrow 0} 
 \frac{d_Y(p,q) - d_Y(c(t),q)}{t}.
\end{equation*}
\end{Lemma}

\begin{Lemma}{{\rm (\cite[Lemma 1.4]{izeki-nayatani})}}
\label{lem:angle2}
 Let $Y$ be a $\cat{0}$ space and $p,q \in Y$.  Let $c(t)$
 $($resp.~$c'(t)$\,$)$ be a non-trivial geodesic starting from $p$
 $($resp.~$q$\,$)$. Then we have
 \begin{equation*}
  \cos \angle_p(c,\cc{p}{q}) = \lim_{t\rightarrow 0} 
       \frac{ d_Y(p,c'(t))-d_Y(c(t),c'(t))}{t}, 
 \end{equation*}
 where $\cc{p}{q}$ denotes the geodesic starting from $p$ and
 terminating at $q$. 
\end{Lemma}

%
%
\noindent
{\it Proof of Proposition \ref{prop:first_variation}}.
 Although the proof goes along the same line as that of 
 \cite[Proposition 2.5]{izeki-nayatani}, where the weights
 corresoponding to our $\mu$ has finite support, we give a proof here
 for the sake of completeness. 
 
 Suppose that $f$ is $\mu$-harmonic, that is, $f$ minimizes $E_{\mu}$. 
 Let $c\colon [0,\varepsilon]\rightarrow Y$ be a geodesic representing
 $V \in S_{f(e)}Y$, and $f_t\colon \Gamma \rightarrow Y$ a
 $\rho$-equivarinat map defined by 
 $f_t(\gamma)= \rho(\gamma) c(t)$. 
 By Lemmas~\ref{lem:angle} and \ref{lem:angle2} above, we see that 
\begin{equation*}
\begin{split}
  & \lim_{t\to +0} \frac{1}{t} 
 \left( d_Y(f_t(e),f_t(\gamma))- d_Y(f(e),f(\gamma)) \right) \\
 = & \lim_{t\to +0} \frac{1}{t}\left[\big(d_Y(c(t),\rho(\gamma)c(t))
       -d_Y(c(0),\rho(\gamma)c(t))\big) \right. \\
  & \phantom{++}
    + \left. \big(d_Y(c(0),\rho(\gamma)c(t)) 
      -d_Y(c(0),\rho(\gamma)c(0)) \big)\right] \\
 =&  -\cos \angle_{f(e)}(V, \tcprj_{f(e)}(f(\gamma)))
   -\cos \angle_{f(\gamma)}(\tcprj_{f(\gamma)}(f(e)),\rho(\gamma)V)
\end{split}
\end{equation*}
holds. Noting that 
$\angle_{f(\gamma)}(\tcprj_{f(\gamma)}(f(e)),\rho(\gamma)V)=
 \angle_{f(e)}(\tcprj_{f(e)}(f(\gamma^{-1})),V)$, 
we get
\begin{equation*}
\begin{split}
 & \lim_{t\to +0} \frac{1}{t} \left(
 d_Y(f_t(e),f_t(\gamma))^2-d_Y(f(e),f(\gamma))^2 \right) \\
 = & -2 d_Y(f(e),f(\gamma)) \left(
   \cos \angle_{f(e)}(V, \tcprj_{f(e)}(f(\gamma)))
   + \cos\angle_{f(e)}(\tcprj_{f(e)}(f(\gamma^{-1})),V)
  \right) \\
 = & -2\left(\inner{V}{\tcprj_{f(e)}(f(\gamma))} 
    + \inner{V}{\tcprj_{f(e)}(f(\gamma^{-1})}\right). 
\end{split}
\end{equation*}
 Since
\begin{equation*}
\begin{split}
   |d_Y(f_t(e),f_t(\gamma)) - d_Y(f(e),f(\gamma))|
 & \leq  d_Y(f(e),f_t(e)) + d_Y(f(\gamma),f_t(\gamma)) \\
 & =  2d_Y(f(e),f_t(e))=2t, 
\end{split}
\end{equation*}
 and$\gamma\mapsto d_Y(f(e),f(\gamma))$ is $\mu$-integrable by the
 Schwarz inequality, we see that
\begin{equation*}
\begin{split}
  & \int_{\Gamma}\frac{|d_Y(f_t(e),f_t(\gamma))^2
         -d_Y(f(e),f(\gamma))^2|}{t}   d\mu(\gamma) \\
 \leq & \int_{\Gamma} \frac{1}{t }2d_Y(f(e),f_t(e))
 \left(2d_Y(f(e),f(\gamma)) + d_Y(f(e),f_t(e)) 
       + d_Y(f(\gamma),f_t(\gamma))\right) d\mu(\gamma) \\
 = & 4 \int_{\Gamma} \left(d_Y(f(e),f(\gamma)) + t\right)
 d\mu(\gamma) < \infty. 
\end{split}
\end{equation*}
 Since $f$ is $\mu$-harmonic, by Lebesgue's dominated convergence
 theorem and $\mu(\gamma)=\mu(\gamma^{-1})$, we obtain
\begin{equation}
 \label{eq:differential}
\begin{split}
  0 & \leq \lim_{t\to +0} \frac{\ene{\mu}(f_t)-\ene{\mu}(f)}{t} \\
    &= \frac{1}{2}\int_{\Gamma} \lim_{t\to +0}
   \frac{d_Y(f_t(e),f_t(\gamma))^2-d_Y(f(e),f(\gamma))^2}{t} d\mu(\gamma) \\
    & = \int_{\Gamma} -2\inner{V}{\tcprj_{f(e)}(f(\gamma))} d\mu(\gamma).
\end{split}
\end{equation}

 Now suppose $f$ satisfies (\ref{eq:first_variation}).  
 Suppose further that there exists a $\rho$-equivariant map $g$ with
 $\ene{\mu}(g)<\ene{\mu}(f)$. 
 Let $d=d_Y(f(e),g(e))$ and take a geodesic 
 $c\colon [0,d]\rightarrow Y$ joining $f(e)$ to $g(e)$. 
 Then, since $t \mapsto E_{\mu}(f_{c(t)})$ is a convex function, 
 together with the above computation, we see that 
\begin{equation*}
\begin{split}
  \int_{\Gamma} 
   -2\inner{V_c}{\tcprj_{f(e)}(f(\gamma))} d\mu(\gamma)
 & = \lim_{t\to +0}\frac{\ene{\mu}(f_{c(dt)})-\ene{\mu}(f)}{dt} \\
 & \leq \lim_{t\to +0} \frac{\left((1-t)\ene{\mu}(f)+t
 \ene{\mu}(g)\right)-\ene{\mu}(f)}{dt}\\ 
 & =\frac{\ene{\mu}(g)-\ene{\mu}(f)}{d} < 0,
\end{split}
\end{equation*}
 where $V_c \in S_{f(e)}Y$ is the equivalence class of $c$. 
 A contradiction. 
 This completes the proof. 
 \qed 

\begin{Remark}
\label{rem:if_eudlidean}
 If $Y$ is isometric to a Riemannian manifold on a neighborhood of
 $f(e)$, then the tangent cone at $f(e)$ is isometric to $\R^n$, 
 and $f_t$ can be extended uniquely and smoothly to $t<0$ with small 
 $|t|$. 
 Then one sees that if $f$ is $\mu$-harmonic, 
 \begin{equation*}
  \lim_{t\to 0} \frac{\ene{\mu}(f_t)-\ene{\mu}(f)}{t}=0
 \end{equation*}
 must holds, and we have an equality in (\ref{eq:differential}). 
 However we cannot replace the inequality with equality if there is a
 singularity in $TC_{f(e)}Y$. 
\end{Remark}

Proposition~\ref{prop:barycenter} below asserts that a
$\mu$-harmonic map $f$ is characterized as a map that locates $f(e)$ at
the {\it barycenter} of $f_* \mu$ defined as follows, where $f_*\mu$ is
the push-forward of a measure $\mu$ by $f$. 

\begin{Definition}
 \label{defn:barycenter}
 Let $Y$ be a $\cat{0}$ space and $\nu$ a probability measure on $Y$
 with finite second moment, that is, 
\begin{equation*}
 \int_Y d_Y(p_0, p)^2 d\nu(p) < \infty
\end{equation*}
 for some (hence any) $p_0 \in Y$.  The {\it barycenter} $\bary{\mu}$ of
 $\nu$ is defined to be a unique point that minimizes
\begin{equation*}
 q \mapsto \int_Y d_Y(q, p)^2 d\nu(p). 
\end{equation*}
\end{Definition}

 The existence of the barycenter is a consequence of the convexity of a
 $\cat{0}$ space
 (see, for example,  \cite[Lemma 2.5.1]{korevaar-schoen1} for the
 proof). 
 Since $t \mapsto d_Y(c(t),p)$ is a convex function (regard $p$ as the
 image of a geodesic with zero velocity),  we see that
 $q \mapsto \int_Y d_Y(p,q)^2d\nu(p)$ is a convex function as in the
 case of $p \mapsto \ene{\mu}(f_p)$. 
 A slight simplification of the proof of
 Proposition~\ref{prop:first_variation} leads us to the following
 proposition;  since the support of $\nu$ is fixed, in contrast to
 $f_t(\Gamma)$ in Proposition~\ref{prop:first_variation}, we only need 
 Lemma~\ref{lem:angle}. 

\begin{Proposition}
 \label{prop:barycenter}
 Let $Y$ be a $\cat{0}$ space and $\nu$ a probability measrure with
 finite second moment. Then $q$ is the barycenter of $\nu$ 
 if and only if $q$ satisfies
\begin{equation*}
 \int_{Y}\inner{V}{\tcprj_q(p)} d\nu(p) \leq 0
\end{equation*}
 for any $V \in S_qY$. 
\end{Proposition}

\begin{Remark}
\label{rem:barycenter}
 Let $f$ be a $\rho$-equivariant $\mu$-harmonic map. 
 Then, since $f$ is $\rho$-equivariant, we see that $f(\gamma)$
 is also the barycenter of $(f\circ \gamma)_* \mu$. 
\end{Remark}

%
%
Let $\varphi \colon \Gamma \rightarrow \R$.  
Suppose that 
$\int_{\Gamma} \varphi(\gamma \gamma') d\mu(\gamma')<\infty$ 
holds for any $\gamma \in \Gamma$. 
The {\it $\mu$-Laplacian} $\Delta=\Delta_{\mu}$ with respect to
$\mu$ is an operator defined by 
\begin{equation*}
 \Delta \varphi(\gamma) 
 = \varphi(\gamma) - \int_{\Gamma}
    \varphi(\gamma \gamma') d\mu(\gamma'). 
\end{equation*}
A function $\varphi \colon \Gamma \rightarrow \R$ is called a 
$\mu$-{\it subharmonic function} if $\Delta \varphi(\gamma)\leq 0$ holds
for any $\gamma \in \Gamma$.  
If $\varphi$ satisfies $\Delta \varphi(\gamma)=0$ for any
$\gamma\in\Gamma$, then $\varphi$ is called a $\mu$-{\it harmonic
function}. 

The following fact plays an important role in the proof of 
Theorem~\ref{thm:main-1}.

\begin{Proposition}
 \label{prop:pullback-is-subharmonic}
 Let $Y$ be a $\cat{0}$ space and 
 $\rho\colon \Gamma\rightarrow \isom{Y}$ a homomorphism. 
 Let $f\colon \Gamma\rightarrow Y$ be a $\rho$-equivariant
 $\mu$-harmonic map, and $u \colon Y \rightarrow \R$ a lower
 semicontinuous convex function on $Y$ satisfying 
 $|u(p)| \leq C_1 d_Y(p_0,p)^2+C_2$ for some $p_0 \in Y$
 and $C_1,C_2>0$. 
 Then the pull-back $f^*u$ of $u$ defined by $f^*u(\gamma)=u(f(\gamma))$
 is a $\mu$-subharmonic function on $\Gamma$. 
\end{Proposition}

\begin{proof}
 By Jensen's inequality for $\cat{0}$ spaces 
 (\cite[p.~242, Proposition 12.3]{eells-fuglede}), for a lower
 semicontinuous convex function $u$
 and a probability measure $\nu$ on $Y$, we have
\begin{equation*}
 u(\bary{\nu}) \leq \int_Y u(p) d\nu(p). 
\end{equation*}
 Note that, since $\rho(\gamma)$ is an isometry, 
 $\bary{\rho(\gamma)_* \nu}=\rho(\gamma)\bary{\nu}$ holds by the
 definition of the barycenter. 
 Therefore, by Propositions~\ref{prop:first_variation} and
 \ref{prop:barycenter},  we see that a $\mu$-harmonic map $f$ satisfies 
 $f(\gamma)=\bary{(\rho(\gamma)\circ f)_*\mu}$ for any 
 $\gamma \in \Gamma$. 
 Thus, substituting $(\rho(\gamma)\circ f)_*\mu=f_*(\gamma_* \mu)$ to
 $\nu$, we get 
\begin{equation*}
 u(f(\gamma)) \leq 
 \int_{\Gamma} u(p) d (f_*(\gamma_*\mu))(p)=
 \int_{\Gamma} f^*u(\gamma') d\gamma_*\mu(\gamma') 
 = \int_{\Gamma} f^*u(\gamma\gamma'') d\mu(\gamma'')
\end{equation*}
 for any $\gamma\in \Gamma$. 
 A computation similar to that in the proof of Lemma~\ref{lem:integral},
 using the triangle inequality and the Schwarz inequality, 
 we see that $\gamma_*\mu$ also has finite second moment. 
 Thus, by our assumption on $u$ and Lemma~\ref{lem:integral}, 
 integrals in the expression above make sense. 
 Thus we conclude that $f^*u$ is a $\mu$-subharmonic function. 
\end{proof}

\begin{Remark}
 According to Ishihara \cite[Theorem 3.4]{ishihara}, a smooth map
 between Riemannian manifolds is harmonic if and only if the pull-back
 of any (local) convex function on the target manifold is a (local)
 subharmonic function. 
 A similar characterization is proved for a harmonic map from 
 a $1$-dimensional Riemannian polyhedron into a nonpositively curved
 metric space \cite[Proposition 12.4]{eells-fuglede}. 
 Although our situation is much simpler than that in
 \cite{eells-fuglede} as the proof above shows, we do not know if the
 converse of Proposition~\ref{prop:pullback-is-subharmonic} is also true
 or not. 
\end{Remark} 

\begin{Remark}
 It should be easy to imagine that, for a convex function $u$, 
 the quantity $|u(\bary{\nu})-\int_Y u(p) d\nu(p)|$ measures the
 strength of the convexity of $u$. 
 Therefore, if we have a $\mu$-harmonic map 
 $f\colon \Gamma \rightarrow Y$ and a convex function 
 $u\colon Y\rightarrow \R$ that is $\mu$-harmonic, then
 we may expect that $u$ has the ``weakest possible convexity''
 on the convex hull of $f(\Gamma)$. 
\end{Remark}

\begin{Remark}
 In terms of probablity theory, 
 $\mu$-subharmonic functions correspond to submartingales associated to
 a random walk generated by $\mu$.  See \cite{sturm}
 for more information on martingale theory and harmonic maps. 
\end{Remark}

We list some known existence results of $\mu$-harmonic maps. 
(See \cite{gromov}, \cite{korevaar-schoen2}, 
\cite{jost}, and \cite{labourie} for related results.)
A homomorphism $\rho\colon \Gamma\rightarrow \isom{Y}$ is called 
{\it reductive in the sense of Jost \cite{jost}} if there exists
a closed, convex $\rho(\Gamma)$-invariant subset $C$ of $Y$ satsifying
either of the following:
\begin{itemize}
 \item $C$ is isometric to a finite-dimensional Euclidean space, or
 \item for any unbounded sequence $\{p_n\}\subset C$, there exists
       $\gamma \in \Gamma$ such that $\{d(p_n, \rho(\gamma)p_n)\}$ is
       unbounded. 
\end{itemize}
\noindent
It is well-known that if $\rho(\Gamma)$ is 
reductive in the sense of Jost, then there exists a
$\rho$-equivariant $\mu$-harmonic map even when $Y$ is not proper (see
\cite{jost}, \cite{labourie}, and Corollary~\ref{cor:existence}). 
Note that if $\rho(\Gamma)$ is not
reductive in the sense of Jost, then there exists an unbounded
sequence $\{p_n\} \subset Y$ such that, for each $\gamma \in \Gamma$, 
$d(p_n,\rho(\gamma)p_n)$ is bounded. 
If $Y$ is proper, by passing to a subsequence of $\{p_n\}$, we can find
a limit point $\xi \in \partial Y$, and see that $\xi$ is fixed by
$\rho(\Gamma)$; thus if $Y$ is proper and $\rho(\Gamma)$ does not fix a
point in $\partial Y$, then there exists a $\rho$-equivariant
$\mu$-harmonic map. 
We will show that this is also true for nonproper $\cat{0}$ spaces with
certain dimension bound
(Proposition~\ref{thm:existence_finite_teledim}).

\section{The rate of escape and equivariant harmonic maps}
\label{sec:rate_of_escape_and_hmap}

%
%
Let $\Gamma$ be a countable group with a symmetric probability
measure $\mu$ whose support generates $\Gamma$. 
This measure $\mu$ can be regarded as a transition probability of a move
starting from the identity element $e$ in $\Gamma$, and hence generates
a random walk on $\Gamma$. 
We define $(\Omega, \mathbb{P})$ to be 
the set of increments of this random walk:
\begin{equation*}
 (\Omega, \mathbb{P})=(\Gamma,\mu)^{\N}=
 (\Gamma, \mu) \times (\Gamma, \mu) \times \dots. 
\end{equation*}
Then, for each $n \in \N$ and a walk $\omega \in \Omega$, 
the position of $\omega$, starting from $e$, at time $n$ is given by a
map 
\begin{equation*}
 \gamma_n \colon \Omega \rightarrow \Gamma; \ 
 \omega \mapsto \omega_1 \dots \omega_n, 
\end{equation*}
where $\omega=(\omega_1, \omega_2, \dots, \omega_n, \dots) \in \Omega$.
Let $\mu^n$ be a probability measure on $\Gamma$ obtained as the
push-forward of $\mathbb{P}$ by a map $\gamma_n$, that is, 
$\mu^n ={\gamma_n}_* \mathbb{P}$. 
Then we see that $\mu^n$ describes the distribution of the position of a
random walk generated by $\mu$ at time $n$,
and, for $\mu^n$-integrable function
$\varphi\colon \Gamma \rightarrow \R$, we have
\begin{equation*}
\begin{split}
  \int_{\Gamma} \varphi(\gamma) d\mu^n(\gamma)
 & = \int_{\Omega} \varphi(\gamma_n(\omega)) d\mathbb{P}(\omega) \\
 & = \int_{\Gamma}\dots \left(\int_{\Gamma}\left(
  \int_{\Gamma} \varphi (\omega_1 \dots \omega_{n-1}\omega_n) 
  d\mu(\omega_n)\right) 
d\mu(\omega_{n-1})\right) \dots d\mu(\omega_1);
\end{split}
\end{equation*}
$\mu^n$ is the $n$-fold convolution of $\mu$. 
Suppose that $\Gamma$ acts on a metric space $Y=(Y,d)$ via a
homomorphism $\rho\colon \Gamma \rightarrow \isom{Y}$ and that the
action given by $\rho$ has finite second moment with respect to $\mu$. 
Then the action also has finite second moment with respect to
$\mu^n$, according to the following computation and induction:
for $p \in Y$, 
\begin{equation*}
\begin{split}
   &\int_{\Gamma} d(p,\rho(\gamma)p)^2 d\mu^n(\gamma)  \\
 = & \int_{\Gamma}\left(
   \int_{\Gamma} d(p,\rho(\gamma'\omega_n)p)^2 d\mu(\omega_n)\right)
   d\mu^{n-1}(\gamma') \\
 \leq & \int_{\Gamma}\left( \int_{\Gamma} 
    \left(d(p,\rho(\gamma')p)+
     d_S(\rho(\gamma')p,\rho(\gamma'\omega_n)p)\right)^2
 d\mu(\omega_n)\right)
   d\mu^{n-1}(\gamma') \\
  = & \int_{\Gamma}\left( \int_{\Gamma} 
     d(p,\rho(\gamma')p)^2+2d(p,\rho(\gamma')p)d(p,\rho(\omega_n)p)
      + d(p,\rho(\omega_n)p)^2
  d\mu(\omega_n)\right)
    d\mu^{n-1}(\gamma') \\
 \leq &  \left(
  1+ 2 \left(\int_{\Gamma} 1+d(p,\rho(\omega_n)p)^2\
 d\mu(\omega_n)\right)\right) 
 \int_{\Gamma}1+ d(p,\rho(\gamma')p)^2\ d\mu^{n-1}(\gamma') \\
 & \phantom{aaa} + \int_{\Gamma} d(p,\rho(\omega_n)p)^2 d\mu(\omega_n), 
\end{split}
\end{equation*}
 where we use the invariance of $d$ under $\rho(\gamma')$ to deduce the
 second equality, and 
 $d(p,q)\leq 1+d(p,q)^2$ to deduce the last inequality.

Now we introduce the quantity called the rate of escape and see
what happens when it vanishes.
It is well-known that, for almost all $\omega \in \Omega$, 
\begin{equation*}
 \lim_n \frac{d(p, \rho(\gamma_n(\omega))p)}{n}
\end{equation*}
exists and the function 
\begin{equation*}
 \omega \mapsto \lim_n \frac{d(p, \rho(\gamma_n(\omega))p)}{n}
\end{equation*}
converges to a constant function in $L^1(\Omega, \mathbb{P})$
(see, for example, \cite[\S 8, A]{woess}). 
The value of this constant function is called the {\it rate of escape}
(or the {\it drift}) of $\rho$ and denoted by 
$l_{\rho}(\Gamma)=l_{\rho}(\Gamma,\mu)$.  
It is easy to check, by using the triangle inequality,
$l_{\rho}(\Gamma)$ does not depend on the choice of $p \in Y$. 

%
%
\begin{Proposition}
\label{prop:mu-harmonicity}
 Let $Y=(Y,d)$ be a proper $\cat{0}$ space, and $\Gamma$ a 
 countable group equipped with a symmetric probability measure $\mu$
 whose support generates $\Gamma$.
 Suppose that $\Gamma$ acts on $Y$ via a homomorphism 
 $\rho\colon \Gamma \rightarrow \isom{Y}$ and that
 there exists a $\rho$-equivariant $\mu$-harmonic map
 $f\colon \Gamma \rightarrow Y$. 
 $($Note that the existence of a $\rho$-equivariant $\mu$-harmonic
 map means that the action given by $\rho$ has finite second moment.$)$
 If $l_{\rho}(\Gamma)=0$, then, for almost all $\omega \in \Omega$,
 there exists 
 $\xi(\omega) \in Y\cup \partial Y$ such that a function 
 $\varphi_{\omega} \colon \Gamma \rightarrow \R$ defined by 
 \begin{equation*}
  \varphi_{\omega}(\gamma) =
   \begin{cases}
    d(\xi(\omega), f(\gamma))-d(\xi(\omega), f(e)) 
    & \text{if } \xi(\omega) \in    Y \\
    b_{\xi(\omega)}(f(\gamma), f(e)) 
    & \text{if } \xi(\omega) \in \partial Y 
   \end{cases}
 \end{equation*}
 satisfies $\Delta \varphi_{\omega}(e)=0$. 
\end{Proposition}

\begin{Remark}
 If one assumes that $\mu(\gamma)\not= 0$ for any 
 $\gamma \in \Gamma\setminus\{e\}$, then one can show that there exists
 $\omega \in \Omega$ such that $\varphi_{\omega}$ in the proposition 
 is $\mu$-harmonic, that is, $\Delta \varphi_{\omega}(\gamma)=0$ for any 
 $\gamma \in \Gamma$. 
 This is true even if $Y$ is not proper, although $\xi(\omega)$ may not
 be a point in $Y\cup \partial Y$ 
 (Proposition~\ref{prop:global-mu-harmonicity}). 
 We will prove this in \S~\ref{sec:global_harmonicity}. 
\end{Remark}

%
%
In order to prove Proposition~\ref{prop:mu-harmonicity},  
let us consider a function $L^n(f)$ of a $\rho$-equivariant map 
$f \colon \Gamma \rightarrow Y$ defined as
\begin{equation*}
 L^n(f) = \int_{\Gamma} d(f(e), f(\gamma)) d\mu^n(\gamma)
  = \int_{\Gamma} d(f(e),\rho(\gamma_n(\omega))f(e))
  d\mathbb{P}(\omega). 
\end{equation*}
Take $\varphi$ to be a function
$\gamma \mapsto d(f(e), f(\gamma))$ and set
\begin{equation*}
 \alpha_n = \int_{\Gamma} -\Delta \varphi(\gamma) d\mu^n(\gamma)
 = \int_{\Omega} -\Delta \varphi(\gamma_n(\omega)) d\mathbb{P}(\omega). 
\end{equation*} 
Then there is the following relation between this $\alpha_n$ and 
the rate of escape $l_{\rho}(\Gamma)$: 

\begin{Lemma}
\label{lem:alpha_n}
If $\liminf_n \alpha_n >0$, then $l_{\rho}(\Gamma)>0$. 
\end{Lemma}

\begin{proof}
Notice that 
\begin{equation*}
\begin{split}
 & L^n(f)+ \alpha_n = 
 L^n(f) - \int_{\Gamma} \Delta \varphi(\gamma)
 d\mu^n(\gamma) \\
 = & \int_{\Gamma} d(f(e), f(\gamma)) d\mu^n(\gamma)
 - \int_{\Gamma} \left(d(f(e),f(\gamma)) - 
   \int_{\Gamma} d(f(e), f(\gamma \gamma')) d\mu(\gamma') \right) 
    d\mu^n(\gamma)\\
 = & \int_{\Gamma} \left(\int_{\Gamma} d(f(e), f(\gamma \gamma'))
 d\mu(\gamma')\right) d\mu^n(\gamma) \\
 = & L^{n+1}(f). 
\end{split}
\end{equation*}
Now suppose that $\liminf_n \alpha_n = c>0$. 
Then there exists $n_0 \in \N$ such that
\begin{equation*}
 \inf_{n \geq n_0} \alpha_n \geq c/2. 
\end{equation*}
Therefore, by taking $c'=\min\{ L^{n_0}(f)/n_0, c/2 \}$, we see that
\begin{equation*}
 L^n (f) \geq L^{n_0}(f) + \frac{c(n-n_0)}{2} \geq c'n
\end{equation*}
for any $n\geq n_0$. 
Since 
\begin{equation*}
 \omega \mapsto \frac{d(f(e), f(\gamma_n(\omega)))}{n}
\end{equation*}
converges to $l_{\rho}(\Gamma)$ in $L^1(\Omega, \mathbb{P})$, 
we have 
\begin{equation*}
 l_{\rho}(\Gamma) 
 = \lim_{n\to \infty} \int_{\Omega}\frac{d(f(e),f(\gamma_n(\omega)))}{n}
 \ d\mathbb{P}(\omega)
 = \lim_{n\to \infty} \frac{L^n(f)}{n} \geq c'>0. 
\end{equation*}
This completes the proof. 
\end{proof}

\medskip

\noindent
{\it Proof of Proposition~\ref{prop:mu-harmonicity}}. 
Suppose that there exists a
$\rho$-equivariant $\mu$-harmonic map $f \colon \Gamma \rightarrow Y$. 
Since $f$ is harmonic and $y \mapsto d(f(e),y)$ is a continuous convex
function, by Proposition~\ref{prop:pullback-is-subharmonic}, we see that 
$\varphi(\gamma)=d(f(e),f(\gamma))$ becomes a 
$\mu$-subharmonic function, that is, 
\begin{equation*}
 \Delta \varphi(\gamma) 
 = \varphi(\gamma) - \int_{\Gamma}
 \varphi(\gamma \gamma') d\mu(\gamma') \leq 0
\end{equation*}
holds for any $\gamma \in \Gamma$.  

Since we are assuming $l_{\rho}(\Gamma)=0$, we see that 
$\liminf_n \alpha_n\leq 0$ by Lemma~\ref{lem:alpha_n}.  
On the other hand, since $\varphi$ is subharmonic, we have
$-\Delta \varphi(\gamma)\geq 0$ for any $\gamma \in \Gamma$, 
and hence we conclude that
\begin{equation*}
 \liminf_{n} \int_{\Omega} -\Delta \varphi(\gamma_n(\omega))
 d\mathbb{P}(\omega) = 0. 
\end{equation*}
Therefore, by taking a subsequence $\{n_j\}\subset \N$, we have 
\begin{equation*}
 \lim_{j} \int_{\Omega} -\Delta \varphi(\gamma_{n_j}(\omega))
 d\mathbb{P}(\omega) = 0.
\end{equation*}
Since $-\Delta \varphi(\gamma) \geq 0$ for any $\gamma \in \Gamma$, this means
that 
$\{-\Delta \varphi(\gamma_{n_j}(\cdot))\colon \Omega \rightarrow \R\}_{j\in \N}$
converges to $0$ in $L^1(\Omega,\mathbb{P})$. 
Thus, by taking a subsequence if necessary, we may assume that 
$-\Delta \varphi(\gamma_{n_j}(\cdot)) \to 0$ almost everywhere on 
$(\Omega,\mathbb{P})$. 
Take $\omega \in \Omega$ so that 
$-\Delta\varphi(\gamma_{n_j}(\omega))\to 0$. 
Letting
\begin{equation*}
 \varphi_{\omega, n_j}(\gamma)
 = d(f(\gamma_{n_j}(\omega)^{-1}), f(\gamma)) - 
 d(f(\gamma_{n_j}(\omega)^{-1}), f(e)), 
\end{equation*}
we have
\begin{equation*}
 \Delta \varphi_{\omega, n_j} (\gamma)
 = d(f(\gamma_{n_j}(\omega)^{-1}), f(\gamma))
 - \int_{\Gamma}
 d(f(\gamma_{n_j}(\omega)^{-1}), f(\gamma \gamma')) d\mu(\gamma'). 
\end{equation*}
In particular, we see that 
\begin{equation*}
\begin{split}
 \Delta \varphi_{\omega, n_j} (e)
 & = d(f(\gamma_{n_j}(\omega)^{{-1}}), f(e))
 - \int_{\Gamma}
 d(f(\gamma_{n_j}(\omega)^{-1}), f(\gamma')) d\mu(\gamma') \\
 & = d(f(e), f(\gamma_{n_j}(\omega))) - 
 \int_{\Gamma}d(f(e), f(\gamma_{n_j}(\omega))\gamma')d \mu(\gamma') \\
 & = \Delta \varphi(\gamma_{n_j}(\omega)). 
\end{split}
\end{equation*}
Passing to a subsequence if necessary, since $Y$ is proper, we may
assume that the sequence $\{\gamma_{n_j}(\omega)^{-1}\}$ converges to a
point $\xi(\omega) \in Y\cup \partial Y$.
If  $\xi(\omega) \in Y$, we have
 \begin{equation*}
 \varphi_{\omega}(\gamma) := \lim_j \varphi_{\omega, n_j}(\gamma) 
 = d(\xi(\omega), f(\gamma)) - d(\xi(\omega), f(e)), 
 \end{equation*}
while in the case $\xi(\omega) \in \partial Y$ we get
 \begin{equation*}
 \varphi_{\omega}(\gamma):= \lim_j \varphi_{\omega, n_j}(\gamma) 
 = b_{\xi(\omega)}(f(\gamma), f(e)), 
 \end{equation*}
where $b_{\xi(\omega)}(\cdot, \cdot)$ denotes the 
Busemann function associated to $\xi(\omega)$ 
(see Remark~\ref{rem:busemann_fct}).  
Note that
\begin{equation*}
 |\varphi_{\omega,n_j}(\gamma)| = 
 |d(f(\gamma_{n_j}(\omega)^{-1}),f(\gamma)) - 
  d(f(\gamma_{n_j}(\omega)^{-1}),f(e))| \leq d(f(e),f(\gamma))
\end{equation*}
and that 
the action given by $\rho$ has finite second moment with respect
to $\mu$. 
Then we see that, by Lebesgue's dominated convergence theorem, 
\begin{equation*}
 \lim_j \int_{\Gamma} \varphi_{\omega,n_j}(\gamma) d\mu(\gamma)
 = \int_{\Gamma} \varphi_{\omega}(\gamma) d\mu(\gamma), 
\end{equation*}
and hence we get $\Delta \varphi_{\omega} (e)=0$.
This completes the proof of Proposition~\ref{prop:mu-harmonicity}. 
\qed

\section{Some consequences of the convexity of distance functions}
\label{sec:convexity_of_distance_and_Busemann_fct}

In this section, we show two propositions which reflect the
convexity of distance functions on $\cat{0}$ spaces. 
These propositions, together with Proposition~\ref{prop:mu-harmonicity}
above,  play a key role in the proof of Theorem~\ref{thm:main-1}. 

%
%
\begin{Proposition}
\label{prop:convexity_of_distance_sphere}
 Let $Y$ be a $\cat{0}$ space,  and $p, q \in Y$. 
 Let $c\colon [0,\infty)\longrightarrow Y$ be a geodesic ray starting
 from $p$. 
 Take $r_0>0$ and let $p_0=c(r_0)$. 
 Denote by $V_c$ the element in $S_pY$ corresponding to the geodesic
 segment $t \mapsto c(t)$.  
 Then we have
\begin{equation*}
   \inner{V_c}{\tcprj_p(q)}  \geq d(p_0,p)-d(p_0,q).
\end{equation*}
\end{Proposition}

\begin{proof}
 Suppose that $\angle (\tcprj_p(p_0),\tcprj_p(q))=0$. 
 Then $\tcprj_p(q)$ lies in a geodesic ray $\{tV_c\mid t\geq 0\}$ in
 $TC_pY$ and 
 \begin{equation*}
   \inner{V_c}{\tcprj_p(q)} = d_{T_p}(O_p,\tcprj_p(q))
   = d(p,q) \geq d(p_0,p)-d(p_0,q)
 \end{equation*}
 by the triangle inequality.  

 Suppose $\angle (\tcprj_p(p_0),\tcprj_p(q))\not=0$ and 
 $\angle (\tcprj_p(p_0),\tcprj_p(q))< \pi$. 
 Set  $W_q = \tcprj_p(q)/\Vert \tcprj_p(q) \Vert \in S_pY$. 
 Since $0< \angle (V_c, W_q) =\angle (\tcprj_p(p_0),\tcprj_p(q)) < \pi$, 
 there is a non-trivial geodesic segment $\bar c$ in $S_pY$ starting
 from $V_c$ and terminating at $W_q$.
 By attaching a segment of length $\pi-\angle (V_c, W_q)$ to the
 endpoint $W_q$, we obtain a $\cat{1}$ space $S_p$, containing
 $S_pY$ as a convex subset in which $\bar c$ can be extended up to
 length $\pi$
 (\cite[p.~347, 11.1 Basic Glueing Theorem]{bridson-haefliger}). 
  Denote by $-V_c \in S_p$ the endpoint of this extended
 geodesic $\bar c$.  Then we have
 \begin{equation*}
  \angle(V_c, W_q) + \angle(W_q,-V_c) = \pi, 
 \end{equation*}
 in particular, 
 \begin{equation}
 \label{eq:minus_vector}
  -\inner{-V_c}{\tcprj_p(q)} 
  = \inner{V_c}{\tcprj_p(q)}. 
 \end{equation}
 Let $T_p$ be the Euclidean cone over $S_p$.  
 Then $T_p$ contains $TC_pY$ as a convex subset. 
 Note that $\{tV_c\mid t\geq 0\} \cup \{t(-V_c) \mid t\geq 0\}$
 forms a geodesic line in $T_p$ since $\angle(V_c,-V_c)=\pi$. 
 Let $\tilde c$ be a geodesic defined by
\begin{equation*}
 \tilde c (t)=
 \begin{cases}
  tV_c & t\geq 0, \\
  \vert t \vert (-V_c) & t \leq 0. 
 \end{cases}
\end{equation*}
 Then $\tilde c$ satisfies
 $\tilde c(\R) = \{tV_c\mid t\geq 0\} \cup \{t(-V_c) \mid t\geq 0\}$, 
 $\tilde c(0)=O_p$ and $\inner{V_c}{\tilde c(t)} = t$ for
 $t \in \R$. 
 Set $d_o=d_{T_p}(\tcprj_p(p_0),O_p)$ and 
 $d_q=d_{T_p}(\tcprj_p(p_0),\tcprj_p(q))$, and let
 $q'=\tilde c(d_o-d_q)$. 
 Then $q'$ is a point on $\tilde c(\R)$ satisfying
 \begin{equation*}
  \inner{V_c}{q'} = d_o-d_q \geq d(p_0,p)-d(p_0,q), 
 \end{equation*}
\begin{minipage}[b]{8cm}
 where the last inequality holds because
 $d_q \leq d(p_0,q)$ while $d_o=d(p_0,p)$. 
 Let $\bar q$ be the image of $\tcprj_p(q)$ by the nearest point
 projection onto the image $\tilde c(\R)$ of $\tilde c$  in
 $T_p$. Since the nearest point projection does not increase the
 distance,
 $d_{T_p}(\tcprj_p(p_0),\bar q) \leq d_q$,
 and hence
 \begin{equation*}
  \bar q \in \tilde 
 c\left(\left[d_o-d_q,\  d_o+d_q\right]\right). 
 \end{equation*}
  Therefore
 \begin{equation*}
  \inner{V_c}{\bar q} \geq d_o-d_q
  = \inner{V_c}{q'}\geq d(p_0,p)-d(p_0,q). 
 \end{equation*}
\end{minipage} 
\begin{minipage}[b]{6cm}
 \hspace{2.5cm} 
 \includegraphics[scale=0.9]{fig.1}
\end{minipage}
 
\vskip0.3cm
\noindent
 By the definition of the distance of a Euclidean cone, we have
 
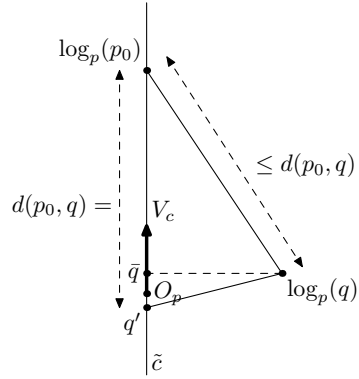
\begin{equation}
 \label{eq:distance_from_tilde_c}
\begin{split}
   & d_{T_p}(\tilde c(t), \tcprj_p(q))^2 
   \\ & =
   \begin{cases}
	t^2 + |\tcprj_p(q)|^2 -2t |\tcprj_p(q)| \cos \angle(V_c,\tcprj_p(q))
	& t \geq 0, \\
	t^2 + |\tcprj_p(q)|^2 -2|t| |\tcprj_p(q)| \cos \angle(-V_c,\tcprj_p(q))
	& t \leq 0. 
\end{cases}     
\end{split}
\end{equation}
 Take $t_0$ to be $t_0V_c=\tilde c(t_0)=\bar q$, and note that the
 left-hand side of (\ref{eq:distance_from_tilde_c}) is minimized at
 $t=t_0$. 
 If $t_0\not= 0$, then the differentiation of 
 $d_{T_p}(\tilde c(t), \tcprj_p(q))^2$ at $t=t_0$ exists and it 
 must be equal to $0$, that is, 
 \begin{equation*}
 \begin{cases}
  t_0 - |\tcprj_p(q)| \cos \angle(V_c,\tcprj_p(q)) = 0 & t_0 > 0, \\
  t_0 + |\tcprj_p(q)| \cos \angle(-V_c, \tcprj_p(q)) = 0 & t_0 < 0. \\
 \end{cases}\end{equation*}
Therefore 
 \begin{equation*}
  \begin{cases}
   \inner{V_c}{\tcprj_p(q)} =
   |\tcprj_p(q)|\cos \angle (V_c,\tcprj_p(q)) =t_0 = \inner{V_c}{\bar q}. 
    & t_0 >0, \\
   \inner{-V_c}{\tcprj_p(q)} =
   |\tcprj_p(q)|\cos \angle (-V_c,\tcprj_p(q)) =-t_0 = 
    -\inner{V_c}{\bar q}. & t_0 <0. \\
  \end{cases}     
 \end{equation*}
 Since $q'$ lies in $\tilde c(\R)$, we have 
 $\inner{-V_c}{q'} = -\inner{V_c}{q'}$. 
 Together with (\ref{eq:minus_vector}), we get 
 \begin{equation*}
  \begin{cases}
   \inner{V_c}{\tcprj_p(q)} = \inner{V_c}{\bar q}
    \geq d(p_0,p)-d(p_0,q) & t_0 >0 \\
   \inner{V_c}{\tcprj_p(q)} = -\inner{-V_c}{\tcprj_p(q)}
   = \inner{V_c}{\bar q}
    \geq d(p_0,p)-d(p_0,q)
   & t_0 <0. 
  \end{cases}
 \end{equation*}
 Thus we obtain the desired inequality in the case $t_0 \not= 0$. 

 If  $t_0=0$, then, from (\ref{eq:distance_from_tilde_c}),  
 the right (resp.~the left) differentiation of 
 $d_{T_p}(\tilde c(t), \tcprj_p(q))^2$ at $t=0$ exists and it is 
 nonnegative (resp.~nonpositive).  Noting (\ref{eq:minus_vector}) again,
 we have
 \begin{equation*}
 \begin{cases}
  - |\tcprj_p(q)| \cos \angle(V_c,\tcprj_p(q))
  = -\inner{V_c}{\tcprj_p(q)} \geq 0  \\
  - |\tcprj_p(q)| \cos \angle(-V_c, \tcprj_p(q))
  = -\inner{-V_c}{\tcprj_p(q)}
  = \inner{V_c}{\tcprj_p(q)} \geq 0,  \\
 \end{cases}\end{equation*}
 which implies $\inner{V_c}{\tcprj_p(q)} = 0$. 
 On the other hand, since the nearest point projection does not increase
 the distance, $t_0=0$ means that $d_q \geq d_o$;
 therefore 
 \begin{equation*}
  \inner{V_c}{\tcprj_p(q)}
  \geq d_o - d_q 
  \geq d(p_0,p)-d(p_0,q)
 \end{equation*}
 again. 

 Now suppose $\angle(V_c, \tcprj_p(q)) = \pi$.  In this case, 
 $V_c$ and $\tcprj_p(q)$ can be joined by a geodesic segment in $T_p$
 passing through $O_p$, hence
 \begin{equation*}
\begin{split}
 \inner{V_c}{\tcprj_p(q)}
  & = d_{T_p}(O_p, \tcprj_p(q)) \cos \angle (V_c, \tcprj_p(q)) \\
  & = -d_{T_p}(O_p, \tcprj_p(q))\\ 
  & = d_o - d_q. 
\end{split} 
\end{equation*}
 Since $d_o=d(p_0,p)$ and $d_q \leq d(p_0,q)$, we get 
\begin{equation*}
 \inner{V_c}{\tcprj_p(q)} \geq d(p_0,p) -d(p_0,q). 
\end{equation*}
This completes the proof. 
\end{proof}

\begin{Remark}
\label{rem:pi_p(q)_and_bar_q}
 Note that, according to the proof above, we have
 \begin{equation*}
  \inner{V_c}{\tcprj_p(q)} = \inner{V_c}{\bar q}
  \geq \inner{V_c}{q'}
 \end{equation*}
 in any case. 
 This can be regarded as an expression of 
 the convexity of a metric ball centered at $\tcprj_p(p_0)$ (or $p_0$). 
 And if the equality in the inequality above holds, 
 then it should be understood that the points in the statement are
 arranged in a convex set with the ``weakest'' convexity, which is 
 an interval. 
 See the next remark.
\end{Remark}

\begin{Remark}
\label{rem:equality_for_distance_ball}
 In the proof above, when $\angle (V_c, \tcprj_p(q))$ is equal to either
 $0$ or $\pi$, the geodesic
 segment joining $V_c$ and $\tcprj_p(q)$ extends or passes through the
 origin $O_p$, and hence $\tcprj_p(q)=\bar q$. 
 If $\angle (V_c, \tcprj_p(q))\not=0, \pi$, then the equality in the
 proposition holds if and only if $\bar q=q'$. 
 In this case, since 
 $d(\tcprj_p(p_0), \tcprj_p(q))=d(\tcprj_p(p_0),q')=d(\tcprj_p(p_0),\bar q)$, 
 the nearest point projection from $[\tcprj_p(p_0),\tcprj_p(q)]$ to 
 $\tilde c(\R)$ turns out to be an isometry. 
 By \cite[p.~182, 2.12(3)]{bridson-haefliger}, 
 the distance between $\tilde c(\R)$ and each point of 
 $[\tcprj_p(p_0),\tcprj_p(q)]$ must be constant. 
 Both sets share the same point $\tcprj_p(p_0)$, the distance between these
 set must be $0$.  Therefore $\tcprj_p(q)=\bar q$. 
\end{Remark}

%
%
Letting $p_0\to \xi=c(\infty)$ in 
Proposition~\ref{prop:convexity_of_distance_sphere}, we obtain the
following proposition, which can be viewed as a sort of
quantitative expression of the convexity of the horoball.

\begin{Proposition}
\label{prop:convexity_of_horosphere} 
 Let $Y$ be a $\cat{0}$ space, and $p, q \in Y$. 
 Let $c\colon [0,\infty)\longrightarrow Y$ be a geodesic ray starting
 from $p$ and terminating at $\xi \in \partial Y$. 
 Denoting the element corresponding to $t \mapsto c(t)$ in $S_pY$ by 
 $V_c$, we have
\begin{equation*}
   \inner{V_c}{\tcprj_p(q)}
   \geq b_{\xi}(p,q). 
\end{equation*}
 If the equality holds in the inequality above, then the convex hull of 
 $[p, \xi)\cup [q, \xi)\cup [p, q]$ is 
 isometric to a half-infinite parallelogram in the Euclidean plane.
\end{Proposition}

\begin{proof}
 Let $p_0 = c(t_0)$ and recall 
 Proposition~\ref{prop:convexity_of_distance_sphere}. Then we have
 \begin{equation*}
  \inner{V_c}{\tcprj_p(q)} \geq d(c(t_0),p)-d(c(t_0),q)
 \end{equation*}
 for any $t_0 \in \R$. Letting $t_0 \to \infty$ yields the desired
 inequality. 

 Take $S_p$ and $T_p$ as in the proof of
 Proposition~\ref{prop:convexity_of_distance_sphere}. 
 Note that 
 $O_p$ and $q'$ in $\tilde c(\R)$ are located on the same side of 
 $\tilde c(\R)$ with respect to $\tcprj_p(p_0) = \tilde c(t_0)$ for
 sufficiently large $t_0$.  Recalling the definition of $q'$ and
 Remark~\ref{rem:pi_p(q)_and_bar_q}, although $q'$ may depend on $t_0$, we have
 \begin{equation*}
  \begin{split}
   b_{\tilde c(\infty)}(O_p,\tcprj_p(q))
   & = \lim_{t_0\to \infty}d_{T_p}(\tilde c(t_0),O_p)
   - d_{T_p}(\tilde c(t_0), \tcprj_p(q)) \\
   & = \lim_{t_0\to \infty}d_{T_p}(\tilde c(t_0),O_p)
   - d_{T_p}(\tilde c(t_0), q') \\
   & = \lim_{t_0\to \infty}\inner{V_c}{q'} \\
   & \leq \inner{V_c}{\tcprj_p(q)}.
  \end{split} 
 \end{equation*}
 On the other hand, since $\tcprj_p$ does not increase the distance while
 $d(c(t_0),p)=d_{TC_pY}(\tilde c(t_0),O_p)$ and 
 $\tcprj_p(c(t_0))=\tilde c(t_0)$ for $t_0 \geq 0$, we have
 \begin{equation*}
\begin{split}
   b_{c(\infty)}(p,q) & = \lim_{t_0 \to \infty}
   d(c(t_0),p) -d(c(t_0),q) \\
   & \leq \lim_{t_0 \to \infty}
    d_{TC_pY}(\tilde c(t_0), O_p) - 
    d_{TC_pY}(\tilde c(t_0), \tcprj_p(q)) \\
   & = b_{\tilde c(\infty)}(O_p, \tcprj_p(q)).
\end{split} 
\end{equation*}
 Therefore the equality 
 $\inner{V_c}{\tcprj_p(q)} =b_{c(\infty)}(p,q)$ means that
 \begin{equation*}
  b_{\tilde c(\infty)}(O_p,\tcprj_p(q))=b_{c(\infty)}(p,q)
 \end{equation*}
 which implies that
 \begin{equation*}
  \lim_{r_0 \to \infty}
   \left(d(c(t_0),q) -d_{TC_pY}(\tilde c(t_0), \tcprj_p(q)) \right)
   = b_{\tilde c(\infty)}(O_p,\tcprj_p(q))-b_{c(\infty)}(p,q)
   = 0. 
 \end{equation*}
 Then, by Lemma~\ref{angle_nondecreasing} below, we see that 
 $d(c(t_0),q)=d_{TC_pY}(\tilde c(t_0), \tcprj_p(q))$ for any 
 $t_0 \in [0,\infty)$.  Note that, by the definition of the distance on
 $TC_pY$,  any triangle in $TC_pY$ having a vertex $O_p$ is isometric to
 a triangle in $\R^2$. Thus
 $\tri{\tilde c(t_0)}{O_p}{\tcprj_p(q)}$ can be regarded as the
 comparison triangle for $\tri{c(t_0)}{p}{q}$ $t_0 \in [0,\infty)$. 
 Since $\angle c(t_0)pq = \angle c(t_0) O_p \tcprj_p(q)$ by definition, 
 \cite[p.~180, 2.9 Proposition]{bridson-haefliger} tells us that  
 $\tri{c(t_0)}{p}{q}$ is a flat triangle for any $t_0 \in [0,\infty)$. 
 Therefore the convex hull of $[p, \xi)\cup [q, \xi) \cup [p, q]$ is
 isometric to a region obtained as the limit of a increasing sequence of
 flat triangles isometric to $\tri{c(t_0)}{p}{q}$, $t_0 \in [0,\infty)$, 
 in $\R^2$, which gives us a half-infinite parallelogram. 
 This completes the proof. 
\end{proof}

\begin{Lemma}
\label{angle_nondecreasing}
 Let $c\colon [0,\infty) \rightarrow Y$ be a geodesic ray and $q \in Y$. 
 Set $\tilde c (t):=\tcprj_{c(0)}(c(t))$. If 
\begin{equation*}
\lim_{t\to \infty}
d(c(t),q)-d_{TC_{c(0)}Y}(\tilde c(t),\tcprj_{c(0)}(q))=0
\end{equation*}
 holds, then, for any $t \in [0,\infty)$,
 \begin{equation*}
 d(c(t),q)=d_{TC_{c(0)}Y}(\tilde c(t),\tcprj_{c(0)}(q))
 \end{equation*}
 holds. 
\end{Lemma}

\begin{proof}
 Let $\tri{\overline{c(0)}}{\overline{c(t)}}{\overline{q}}$ be a
 comparison triangle of $\tri{c(0)}{c(t)}{q}$, and denote the
 angle $\angle_{\overline{c(0)}} (\overline{c(t)},\overline{q})$
 by $\theta(t)$.  Let
 $\tilde \theta = \angle_{O_{c(0)}}(\tilde c(t),\tcprj_{c(0)}(q))$. 
 Note that $\tilde \theta$ is a constant independent of $t$, and 
 $\theta(t)\to \tilde \theta$ as $t \to 0$ by the definition of the
 angle, and that 
 $\theta(t)$ is nondecreasing as $t$ getting larger
 by the definition of $\cat{0}$ space 
 (or see, for example, \cite[p.~161, 1.7(3)]{bridson-haefliger}).  
 Set 
\begin{equation*}
 d(t) = d(c(t),q), \quad 
 \tilde d(t) =d_{TC_{c(0)}Y}(\tilde c(t),\tcprj_{c(0)}(q)), \quad 
 P= d(c(0),q), 
\end{equation*}
 then by the cosine rule, 
 we see that 
\begin{equation*}
 \cos \theta(t) = \frac{t^2 + P^2 - d(t)^2}{2tP}, \quad
 \cos \tilde \theta = \frac{t^2 + P^2 - \tilde d(t)^2}{2tP}. 
\end{equation*}
 Since  $\theta(t)$ is nondecreasing, if $\theta(t_0) > \tilde \theta$
 for some $t_0$, then by putting $C=\theta(t_0)-\tilde \theta$ we have
 $\theta(t)-\tilde \theta\geq C$ for any $t\geq t_0$. 
 Therefore, there exists $C'>0$ such that 
 \begin{equation*}
   \cos \tilde \theta - \cos \theta(t)
  = \frac{d(t)^2-\tilde d(t)^2}{2tP} \geq C'  
 \end{equation*}
 for any $t \geq t_0$. 
 Since
 \begin{equation*}
  d(t)=\sqrt{t^2 + P^2 -2tP\cos \theta(t)}, \quad
   \tilde d(t)= \sqrt{t^2+P^2 - 2tP\cos \tilde \theta}, 
 \end{equation*}
 we get  
 \begin{equation*}
  \frac{d(t)^2-\tilde d(t)^2}{2tP}
   = \frac{d(t)-\tilde d(t)}{2P}
   \left(\sqrt{1 + \frac{P^2}{t^2} -2\frac{P\cos \theta(t)}{t}}
    + \sqrt{1 + \frac{P^2}{t^2} -2\frac{P\cos \tilde \theta}{t}}
   \right). 
 \end{equation*}
 On the other hand, for $t\geq t_0$, we have
 \begin{equation*}
  \sqrt{1 + \frac{P^2}{t^2} -2\frac{P\cos \theta(t)}{t}},
   \sqrt{1 + \frac{P^2}{t^2} -2\frac{P\cos \tilde \theta}{t}}
   \leq \sqrt{\left(1+\frac{P}{t} \right)^2}
   = 1 + \frac{P}{t} \leq 1 + \frac{P}{t_0}, 
 \end{equation*}
 hence 
 \begin{equation*}
  C' \leq \frac{d(t)^2-\tilde d(t)^2}{2tP}
   \leq 
  2\left(1 + \frac{P}{t_0}\right) 
  \frac{d(t)-\tilde d(t)}{2P}. 
 \end{equation*}
 Thus, for $t\geq t_0$, 
 \begin{equation*}
   \frac{t_0C'P}{(t_0 + P)} \leq
  d(t)-\tilde d(t) . 
 \end{equation*}
 Thus $\lim_{t\to \infty} d(t)-\tilde d(t)=0$ means that $C'=0$, that
 is, $\theta(t)=\tilde \theta$ for any $t \geq 0$.   
 Therefore we get $d(t)=\tilde d(t)$ for any $t \geq 0$. 
\end{proof}

\section{Proof of Theorem \ref{thm:main-1}: when $Y$ is proper}
\label{sec:proper_case}

In this section, we give the proof of Theorem~\ref{thm:main-1} in the
case that a $\cat{0}$ space $Y$ is proper. 
Suppose that a countable group $\Gamma$ with a random walk
$\mu$ acts on $Y$ via $\rho \colon \Gamma \rightarrow \isom{Y}$ 
with finite second moment, 
and we are assuming that $\rho(\Gamma)$ does not fix a point in
$\partial Y$. 
Then as we mentioned in \S~\ref{sec:harmonic_maps}, there exists
$\rho$-equivariant $\mu$-harmonic map $f$.  
If $l_{\rho}(\Gamma)=0$, then by Proposition~\ref{prop:mu-harmonicity}, 
for almost all $\omega \in \Omega$, there exists
$\xi(\omega) \in Y \cup \partial Y$ such that 
\begin{equation*}
\varphi_{\omega}(\gamma) =
\begin{cases}
d(\xi(\omega), f(\gamma))-d(\xi(\omega), f(e)) & \xi(\omega) \in
Y \\
b_{\xi(\omega)}(f(\gamma), f(e)) & \xi(\omega) \in \partial Y 
\end{cases}
\end{equation*}
satisfies $\Delta \varphi_{\omega}(e)=0$. 

%
%
First we note that we may assume, by replacing $\mu$ with a convex
combination of $\mu^n$, that the support $\supp\mu$ of
a probability measure $\mu$ contains $\Gamma \setminus \{e\}$. 
Indeed, since we are assuming that the support of our original $\mu$
generates $\Gamma$, we see that, for any $\gamma$, there exists $n$ such
that $\mu^n(\gamma)\not=0$.  
Thus we may assume that $\sum_{n=1}^{\infty }a_n \mu^n(\gamma)\not=0$
for any $\gamma \in \Gamma \setminus \{e\}$, where $a_n\geq 0$ and
$\sum_{n=1}^{\infty}a_n=1$. 
It is easy to see that, by taking $a_n$ suitably,
the action given by $\rho$ has finite second moment with respect to 
$\sum_{n=1}^{\infty}a_n \mu^n$, since it is true for each
$\mu^n$ as we have seen in \S~\ref{sec:rate_of_escape_and_hmap}.
Furthermore, according to \cite[Theorem 4.5]{forghani}, 
$l_{\rho}(\Gamma,\sum_{n=1}^{\infty}a_n\mu^n)=
 \left(\sum_{n=1}^{\infty}n a_n \right) l_{\rho}(\Gamma,\mu)$
holds, and hence we may assume 
$l_{\rho}(\Gamma,\sum_{i=n}^{\infty}a_n\mu^n)=0$, which is our
assumption on $\mu$ in Theorems \ref{thm:main-1} and \ref{thm:main-3}. 
(Although the outline of the proof of \cite[Theorem 4.5]{forghani} is
suggested there, for the sake of completeness, we will give a proof of a
weaker result in \S~\ref{sec:appendix}, which is sufficient for our
purpose.)
Note also that, since we are assuming that $\rho(\Gamma)$
fixes no point in $\partial Y$ and $Y$ is proper, 
the existence of $\mu$-harmonic map is not 
affected by this change of a probability measure $\mu$.

The following lemma summarizes the core argument used in the
proof of Theorem~\ref{thm:main-1} and Theorem~\ref{thm:main-3}.

\begin{Lemma}
 \label{lem:core}
 Let $Y$ be a $\cat{0}$ space $($we do not assume $Y$ to be any one of  proper,
 of finite telescopic dimension, or of locally finite dimension here$)$,  
 and $\Gamma$ a countable group with
 symmetric and nondegenerate probability measure $\mu$. 
 Further assume that $\supp \mu \supset \Gamma \setminus \{e\}$. 
 Suppose that $\Gamma$ acts on $Y$ via a homomorphism 
 $\rho\colon \Gamma \rightarrow \isom{Y}$ and that there exists a
 $\rho$-equivariant $\mu$-harmonic map $f\colon \Gamma \rightarrow Y$.

 $(1)$ Let $\xi \in Y$ and  
 $\varphi\colon \gamma \mapsto d(\xi,f(\gamma))-d(\xi,f(e))$. 
 If $\Delta \varphi(e)=0$ holds, then there exists a point in $Y$ fixed
 by $\rho(\Gamma)$.

 $(2)$ Let $\xi \in \partial Y$ and 
 $\varphi\colon \gamma \mapsto b_{\xi}(f(\gamma),f(e))$. 
 If $\Delta \varphi(e)=0$ holds, then, for any
 $\gamma \in \Gamma\setminus \{e\}$, the convex hull of 
 $\ray{f(e),\xi}\cup [f(e),f(\gamma)]\cup \ray{f(\gamma),\xi}$
 is isometric to a half-infinite parallelogram in $\R^2$. 

 $(3)$ Suppose that the assumption of $(2)$ is satisfied and that there
 exists a sequence $\{\gamma_j\}\subset \Gamma$ such that 
 $\{f(\gamma_j)\}$ converges to $\xi$ in $(2)$ in the cone topology. 
 If further $\{\rho(\gamma_j^{-1})\xi\}$ and $\{f(\gamma_j^{-1})\}$
 converge to $\xi^+, \xi^{-}\in \partial Y$ in the cone topology
 respectively, 
 then, for each $\gamma \in \Gamma$, there exists a geodesic 
 $c_{\gamma}\colon \R \rightarrow Y$ such that $c_{\gamma}(0)=f(\gamma)$
 and $c_{\gamma}(\pm \infty)=\xi^{\pm}$.

 $(4)$ Suppose that 
 $l_{\rho}(\Gamma)=0$. 
 Under the assumptions of $(2)$ and $(3)$, 
 the convex hull $\ch{f(\Gamma)}$ splits as $H\times Z$, where
 $H$ is a complete Hilbert space and $Z$ does not admit any splitting
 $\R \times Z'$.  
 Furthermore, $\rho$ splits along this splitting
 as $\rho=\rho_1 \times \rho_2$, where 
 $\rho_1 \colon \Gamma \rightarrow \isom{H}$ and 
 $\rho_2 \colon \Gamma \rightarrow \isom{Z}$. 
 There also exists $\rho_2$-equivariant $\mu$-harmonic map 
 $f_2\colon \Gamma \rightarrow Z$, and $l_{\rho_2}(\Gamma)=0$
 holds.
\end{Lemma}

\begin{proof}
$(1)$ We show that $f$ is a constant map under the assumption.
Then we see that $f(\Gamma)=\rho(\Gamma)f(e)$ is a point, and this point
 is fixed by $\rho(\Gamma)$.

If $\\xi=f(e)$, then
$\varphi(\gamma)=d(f(\gamma),f(e))$, 
and $\Delta \varphi(e)=0$ means that 
\begin{equation*}
 \Delta \varphi(e)= 
 \varphi(e)- \int_{\Gamma}\varphi(\gamma)d\mu(\gamma)
 = - \int d(f(\gamma), f(e)) d\mu(\gamma)= 0, 
\end{equation*}
 namely $f(\Gamma)=f(e)$; $f$ is a constant map.

Now assume $\xi\not= f(e)$. 
Take any $\gamma \in \Gamma\setminus \{e\}$, and note that
$\mu(\gamma)\not=0$ by our assumption. 
Let $p=f(e)$, $q=f(\gamma)$ and $p_0=\xi$ in 
Proposition~\ref{prop:convexity_of_distance_sphere}, 
and denote by $V_c \in S_{f(e)}Y$ the element corresponding to the
 geodesic $\cc{f(e)}{\xi}$ joining $f(e)$ and $\xi$. 
Then we get
\begin{equation*}
 \inner{V_c}{\tcprj_{f(e)}(f(\gamma))}
\geq d(\xi, f(e)) - d(\xi, f(\gamma))= -\varphi(\gamma). 
\end{equation*}
By Proposition~\ref{prop:first_variation} and 
$\varphi(e)=0$, we see that 
\begin{equation*}
\begin{split}
0 & \geq \int_{\Gamma} \inner{V_c}{\tcprj_{f(e)}(f(\gamma))}
d\mu (\gamma) 
\geq \int_{\Gamma} -\varphi(\gamma) d\mu (\gamma) \\
& = \left(
   \varphi(e) - \int_{\Gamma} \varphi(\gamma) d\mu(\gamma)
  \right) \\
& = \Delta (\varphi(e))=0.
\end{split}
\end{equation*}
Therefore we get
\begin{equation*}
\int_{\Gamma} \inner{V_c}{\tcprj_{f(e)}(f(\gamma))}
d\mu(\gamma)=0. 
\end{equation*}
Moreover the equality in
Propostion~\ref{prop:convexity_of_distance_sphere} must hold 
for each $\gamma \in \Gamma\setminus \{e\}$:
 \begin{equation}
 \label{eq:on_the_same_side}
  \inner{V_c}{\tcprj_{f(e)}(f(\gamma))} = d(\xi,f(e))-d(\xi,f(\gamma)). 
 \end{equation}
By Remark~\ref{rem:equality_for_distance_ball} this implies that
$\tcprj_{f(e)}(f(\gamma)) \in \tilde c(\R)$ for every 
$\gamma \in \Gamma\setminus \{e\}$, where $\tilde c$ is a geodesic in
$T_{f(e)}$ passing through $O_{f(e)}$ and $\tcprj_{f(e)}(\xi)$.
Hence we have
\begin{equation*}
\begin{split}
  d(\xi,f(\gamma)) &=
  d(\xi,f(e))-\inner{V_c}{\tcprj_{f(e)}(f(\gamma))} \\
 & = d_{T_{f(e)}}(\tcprj_{f(e)}(\xi), \tcprj_{f(e)}(f(\gamma))). 
\end{split}
\end{equation*}
Therefore 
$\tri{O_{f(e)}}{\tcprj_{f(e)}(f(\gamma))}{\tcprj_{f(e)}(\xi)}$, 
which becomes a comparison triangle of $\tri{f(e)}{f(\gamma)}{\xi}$, 
degenerates to a segment and contained in $\tilde c(\R)$. 
This means that $f(\gamma)$ is contained in a geodesic segment passing through 
$f(e)$ and $\xi$ for every $\gamma \in \Gamma\setminus \{e\}$. 
Note that, although such a geodesic on $Y$ may not be unique, all such
geodesics share a geodesic segment $[\xi, f(e)]$ in common. 

Set 
\begin{equation*}
 a(\gamma):=  \inner{V_c}{\tcprj_{f(e)}(f(\gamma))}
  = d(\xi,f(e))-d(\xi,f(\gamma)), 
\end{equation*}
and $b:=d(\xi,f(e))$.
Then we have
\begin{equation*}
a(\gamma)=
\begin{cases}
d(f(e),f(\gamma))  &
 \inner{V_c}{\tcprj_{f(e)}(f(\gamma))} \geq 0 \\
-d(f(e),f(\gamma))  &
 \inner{V_c}{\tcprj_{f(e)}(f(\gamma))} \leq 0, 
\end{cases}
\end{equation*}
and, by definition, we get $b \geq a(\gamma)$ 
for all $\gamma \in \Gamma \setminus \{e\}$.
Suppose that, for $\gamma \in \Gamma$, $f(\gamma)$ lies on a
geodesic segment starting from $f(e)$ and passing through $\xi$. 
Then we have $\inner{V_c}{\tcprj_{f(e)}(f(\gamma))}\geq 0$, 
and $b\geq a(\gamma)$ means that 
$d(\xi,f(e))\geq d(f(\gamma),f(e))$, in other words,
$f(\gamma) \in [\xi,f(e)]$. 
Therefore, for any $\gamma \in \Gamma$, 
$f(\gamma)$ must lie on a geodesic segment starting from $\xi$
and passing through $f(e)$; it cannot lie somewhere beyond $\xi$. 

Let $\xi'$ be the point in the closure of $f(\Gamma)$ closest to $\xi$.
Thus, for any $\varepsilon >0$, there exists a $\gamma_j \in \Gamma$
with $d(f(\gamma_j), \xi')< \varepsilon$. 
Suppose that there is $\gamma_0 \in \Gamma$ such
that  $f(\gamma_0)=\xi'$. 
If $\gamma_0=e$, then we have $d(f(\gamma),\xi)\geq d(f(e),\xi)$ for
 any $\gamma \in \Gamma$.
Hence $\Delta \varphi(e)=0$ implies that
$f(\Gamma)$ must become the point $f(e)$, that is, $f$ is a constant
 map. 
Suppose that $\gamma_0\not= e$. 
For every $\gamma \in \Gamma$,
$f(\gamma)$ lies in a geodesic segment starting from
$\xi'=f(\gamma_0)$ and passing through $f(e)\not= \xi'$. 
Then there is $V_0 \in S_{f(\gamma_0)}Y$ representing these geodesics,
and we get
\begin{equation*}
\begin{split}
   \int_{\Gamma} \inner{V_0}{\tcprj_{f(\gamma_0)} (f(\gamma))}
  \ d{\gamma_0}_*\mu(\gamma)
  = & \int_{\Gamma} d(f(\gamma_0), f(\gamma_0(\gamma_0^{-1}\gamma))) \ 
    d\mu(\gamma_0^{-1}\gamma) \\
  \geq & d(f(\gamma_0),f(e)) \mu(\gamma_0^{-1}) >0. 
\end{split}
\end{equation*}
Recalling Remark~\ref{rem:barycenter}, we see that this
contradicts the harmonicity of $f$.
Therefore $f(\gamma)\not=\xi'$ for any $\gamma \in \Gamma$. 
Note that, for sufficiently small $\varepsilon$, $f(\gamma_j)$ must lie
in the interior of a geodesic segment $[\xi', f(e)]$. 
Set 
$a_j(\gamma)=\inner{V_{c,f(\gamma_j)}}{\tcprj_{f(\gamma_j)}(f(\gamma))}$, 
where $V_{c,f(\gamma_j)}$ denotes the element in $S_{f(\gamma_j)}Y$
corresponding to the geodesic $\cc{f(\gamma_j)}{\xi'}$.

Note that, since $f(\Gamma)$ lies in a union of geodesic 
having $[\xi',f(e)]$ in common, we may regard that $f(\Gamma)$
lies in a $\cat{0}$ space locally isometric to $\R$ around 
$f(\gamma_j) \in (\xi',f(e))$, and hence
$\tcprj_{f(\gamma_j)}(f(\Gamma))$ lies in a geodesic line in
$TC_{f(\gamma_j)}Y$. 
Then, by Remark~\ref{rem:if_eudlidean} and Remark~\ref{rem:barycenter}, 
we see that $f$ satisfies
\begin{equation}
\label{eq:mu-harmonic}
\int_{\Gamma}  a_j(\gamma) d{\gamma_j}_*\mu(\gamma) = 0. 
\end{equation} 
Since each $f(\gamma)$ lies in a geodesic segment starting from
$\xi'$ and passing through $f(\gamma_j)$, we have
\begin{equation*}
 a_j(\gamma) = 
\begin{cases}
 d(f(\gamma_j), f(\gamma)) & a_j(\gamma)\geq 0 \\
 -d(f(\gamma_j), f(\gamma)) & a_j(\gamma) < 0, 
\end{cases}
\end{equation*}
and $a_j(\gamma) \leq d(f(\gamma_j),\xi') < \varepsilon$.  
Thus we have
\begin{equation*}
 \int_{\Gamma} \max\{a_j(\gamma), 0 \} 
 d{\gamma_j}_*\mu (\gamma) <  \varepsilon. 
\end{equation*}
Therefore, by (\ref{eq:mu-harmonic}), 
\begin{equation*}
\begin{split}
  \int_{\Gamma} d(f(e),f(\gamma')) d\mu(\gamma')
 & = \int_{\Gamma} d(f(\gamma_j), f(\gamma_j \gamma')) 
 d{\gamma_j}_* \mu(\gamma_j \gamma') \\
 & = \int_{\Gamma} |a_j(\gamma)| d{\gamma_j}_* \mu(\gamma)
< 2\varepsilon. 
\end{split}
\end{equation*}
Since $\varepsilon$ is arbitrary, this means that $f$ is a constant
map. 

$(2)$ 
Take any $\gamma \in \Gamma \setminus \{e\}$.
Let $p=f(e)$, $q=f(\gamma)$ in Proposition~\ref{prop:convexity_of_horosphere}. 
Then we get
\begin{equation*}
 \inner{V_c}{\tcprj_{f(e)}(f(\gamma))}
  \geq b_{\xi}(f(e),f(\gamma))
  = -\varphi(\gamma). 
\end{equation*}
As in the proof of $(1)$ above, by
 Proposition~\ref{prop:first_variation} and $\varphi(e)=0$, we see that 
\begin{equation*}
\begin{split}
0 & \geq \int_{\Gamma} \inner{V_c}{\tcprj_{f(e)}(f(\gamma))}
d\mu (\gamma) 
\geq \int_{\Gamma} -\varphi(\gamma) d\mu (\gamma) \\
& = \left(
   \varphi(e) -
    \int_{\Gamma} \varphi(\gamma) d\mu(\gamma)
  \right) \\
& = \Delta (\varphi(e))=0. 
\end{split}
\end{equation*}
Therefore we get
\begin{equation}
\label{eq:harmonic_at_e2}
\int_{\Gamma} \inner{V_c}{\tcprj_{f(e)}(f(\gamma))}
d\mu(\gamma)=0. 
\end{equation}
Moreover the equality in
Propostion~\ref{prop:convexity_of_horosphere} must hold 
for all $\gamma \in \Gamma \setminus \{e\}$: 
 \begin{equation*}
  \inner{V_c}{\tcprj_{f(e)}(f(\gamma))}
  = b_{\xi}(f(e),f(\gamma)). 
 \end{equation*}
 By Proposition~\ref{prop:convexity_of_horosphere}, for all 
 $\gamma \in \Gamma$, the convex hull of
 $[f(e),\xi(\omega)) \cup [f(e),f(\gamma)] \cup [f(\gamma),\xi(\omega))$
 is isometric to a half-infinite parallelogram in $\R^2$.   

$(3)$
Denote by $F_{\gamma}$ a half-infinite parallelogram bounded by
$[f(e),\xi) \cup [f(e),f(\gamma)] \cup [f(\gamma),\xi)$ 
obtained in $(2)$. 
Take any $\gamma \in \Gamma$ and fix it for a while.  Note that 
$\{f(\gamma_j\gamma)\}=\{\rho(\gamma_j)f(\gamma)\}$ also converges to
$\xi$ by our assumption. 
Since $F_{\gamma_j\gamma}$'s are half-infinite parallelograms,
$[f(e),\xi)$ and $[f(\gamma_j\gamma),\xi)$ are 
parallel to each other. 
Thus by taking $q_j \in [f(e), \xi)$ so that 
$d([f(e), \xi), f(\gamma_j\gamma))=d(q_j, f(\gamma_j\gamma))$, for
sufficiently large $j$, we may assume
\begin{equation*}
 d(q_j, f(\gamma_j\gamma))
 =d([f(e), \xi), [f(\gamma_j\gamma), \xi)), 
\end{equation*}
and 
\begin{equation}
\label{eq:flat_angle_sum}
\begin{split}
  & \angle_{f(\gamma_j\gamma)}(\xi, f(e))
  = \angle_{f(\gamma_j\gamma)}(\xi, q_j) + 
    \angle_{f(\gamma_j\gamma)}(q_j, f(e)), \\
 & \angle_{f(\gamma_j\gamma)}(\xi, q_j) =
 \angle_{q_j} (\xi, f(\gamma_j\gamma))=\frac{\pi}{2} . 
\end{split}
\end{equation}
Put 
$\theta_{\gamma_j\gamma}= \angle_{f(\gamma_j\gamma)}(q_j, f(e))$.  
Again noting that $F_{\gamma_j\gamma}$ is flat, we get 
\begin{equation*}
\begin{split}
  \cos \theta_{\gamma_j\gamma}& =
 \frac{d(q_j, f(\gamma_j\gamma))}{d(f(e),f(\gamma_j\gamma))}= 
 \frac{d([f(e), \xi),
 f(\gamma_j\gamma))}{d(f(e),f(\gamma_j\gamma))} \\
  & = \sin \angle_{f(e)}(\xi, f(\gamma_j\gamma))\to 0, 
\end{split}
\end{equation*}
since $f(\gamma_j\gamma) \to \xi$.
Therefore $\theta_{\gamma_j\gamma}\to \pi/2$.
By (\ref{eq:flat_angle_sum}), we see that the angle 
$\angle_{f(\gamma_j\gamma)}(\xi(\omega), f(e))$ tends to $\pi$.

Consider the union of a geodesic segment and a geodesic ray
\begin{equation*}
 \rho(\gamma_j^{-1})
  \left( [f(e), f(\gamma_j\gamma)] \cup [f(\gamma_j\gamma), \xi)
  \right)
 = [f(\gamma_j^{-1}), f(\gamma)]\cup [f(\gamma),
 \rho(\gamma_j^{-1})\xi), 
\end{equation*}
for each $j$. 
Recall that 
$\{\rho(\gamma_j^{-1})\xi\}$ and $\{f(\gamma_j^{-1})\}$ converge
to points $\xi^+, \xi^-\in \partial Y$ respectively by our
assumption. 
Therefore we may assume that a sequence of geodesic rays 
$\{\cc{f(\gamma)}{\rho(\gamma_j^{-1})\xi}\}$ converges to
$\cc{f(\gamma)}{\xi^+}$ unifomly on each bounded interval, 
where $\cc{f(\gamma)}{\rho(\gamma_j^{-1})\xi}$ denotes the
geodesic starting from $f(\gamma)$ and terminating at 
$\rho(\gamma_j^{-1})\xi$ as before. 
On the other hand, setting 
$\tilde c_j \colon \ray{0,\infty}\rightarrow Y$ to be
\begin{equation*}
 \tilde c_j(t)=
 \begin{cases}
  \cc{f(\gamma)}{f(\gamma_j^{-1})} (t) & t \leq
  d(f(\gamma),f(\gamma_j^{-1})) \\
  f(\gamma_j^{-1}) & t \geq d(f(\gamma),f(\gamma_j^{-1})), 
 \end{cases}
\end{equation*}
we see that $\{\tilde c_j\}$ converges to $\cc{f(\gamma)}{\xi^-}$
 uniformly on each bounded interval. 
Recalling that letting $j\to \infty$ gives
\begin{equation*}
 \angle_{f(\gamma)}(\rho(\gamma_j^{-1})\xi,f(\gamma_j^{-1}))
  = \angle_{f(\gamma_j\gamma)} (\xi, f(e)) \to \pi, 
\end{equation*}
we see that the images of $\cc{f(\gamma)}{\xi^+}$ and
$\cc{f(\gamma)}{\xi^-}$ form a straight geodesic line passing through
$f(\gamma)$ whose endpoints in $\partial Y$ are $\xi^+$ and $\xi^-$.

%
%
$(4)$ Now let $\tilde S$ be the union of the images of all geodesics
parallel to $c(\R)$, where $c$ is a geodesic satisfying $c(0)=f(e)$,
$c(-\infty)=\xi^-$, $c(\infty)=\xi^+$. 
Then, by the discussion above, there exists a geodesic
joining $\xi^+$ and $\xi^-$ passing through $f(\gamma)$ for each 
$\gamma \in \Gamma$, and thus
$f(\Gamma)\subset \tilde S$.  
By \cite[p.~183 2.14]{bridson-haefliger}, we see that 
$\tilde S$ is convex, and isometric to a product metric
space $\R \times S$. 
We may take $S$ to be a 
convex subset of the form $\nrprj_{c}^{-1}(c(0))$,
where $\nrprj_{c} \colon Y \rightarrow c(\R)$ is the nearest point
projection onto $c(\R)$ and $c$ is a geodesic satisfying $c(0)=f(e)$,
$c(-\infty)=\xi^-$, and $c(\infty)=\xi^+$ as above.  
Let $\{q_j\} \subset S$ and $q_j \to q$ in $Y$. Then
there is a geodesic $c_j$ satisfying $c_j(0)=q_j$, $c_j(-\infty)=\xi^-$,
and $c_j(\infty)=\xi^+$ for each $j$.  Since $c_j$'s are parallel to each
other, $d(c_j(t), c_{j'}(t))=d(c_j(0),c_{j'}(0))$ for any $t \in \R$ and
$j, j' \in\N$, and hence $\{c_j(t)\}$ is a Cauchy sequence for each $t$. 
Since $Y$ is complete, $\{c_j(t)\}$ converges and forms a geodesic
joining $\xi^-$ and $\xi^+$ and passing through $q$.  Therefore $S$ is
closed, and $\tilde S=\R \times S$ is also closed. 

Now we show that the convex hull $\ch{f(\Gamma)}$ of $f(\Gamma)$
splits as $\R \times S_0$. 
By definition, $\ch{f(\Gamma)}$ is the closure of the minimal convex
subset of $Y$ containing $f(\Gamma)$.  
Since $\tilde S$ is a closed convex subset and 
$f(\Gamma) \subset \tilde S$, $\ch{f(\Gamma)}$ is contained in 
$\tilde S$.
Note that $f(\Gamma)$ is invariant under the action of $\rho(\Gamma)$,
hence so is $\ch{f(\Gamma)}$. 
Recall that $\xi(\omega)=\lim f(\gamma_j)$. 
Then for any $q \in \ch{f(\Gamma)}$, we have
$[q,f(\gamma_j)]\subset \ch{f(\Gamma)}$ and, as a set
$[q,f(\gamma_j)]$ converges to $[q,\xi(\omega))$.  Since
$\ch{f(\Gamma)}$ is closed, we see that 
$[q,\xi(\omega)) \subset \ch{f(\Gamma)}$.  The same is true for
$\xi^- =\lim f(\gamma_j^{-1})$.  Since $\ch{f(\Gamma)}$ is
$\rho(\Gamma)$-invariant, 
$[q, \rho(\gamma_j^{-1})\xi(\omega)) \subset \ch{f(\Gamma)}$ for
any $q \in \ch{f(\Gamma)}$ and $j \in \N$. Since
$[q, \rho(\gamma_j^{-1})\xi(\omega))$ converges to $[q, \xi^+)$ by the 
definition of $\xi^+$, we see that $[q, \xi^+)$ is also contained in
$\ch{f(\Gamma)}$. 

Let $S_0 = \ch{f(\Gamma)} \cap S$.  Then $S_0$ is a closed convex
subset of $Y$. 
As we have seen above, for each $\gamma \in \Gamma$, 
$[f(\gamma),\xi^+)$ (resp.~$[f(\gamma),\xi^-)$) is contained in
$\ch{f(\Gamma)}$.
Since there is a geodesic $c_{\gamma}$ joining $\xi^+$ and $\xi^-$ passing
through $f(\gamma)$, $[f(\gamma),\xi^+)\cup [f(\gamma),\xi^-)$ must
coincide with the image $c_{\gamma}(\R)$ of $c_{\gamma}$. 
Therefore $c_{\gamma}(\R) \subset \ch{f(\Gamma)}$ and, 
by reparametrizing $c_{\gamma}$ so that $c_{\gamma}(0) \in S$, we see
that $c_{\gamma}(0) \in S \cap \ch{f(\Gamma)}=S_0$. 
Hence $f(\gamma) \in \R \times S_0 \subset \tilde S$ for every
$\gamma \in \Gamma$.
Since $S_0$ is closed and convex, so is $\R \times S_0$, and we see that 
$\ch{f(\Gamma)} \subset \R \times S_0$. 
On the other hand, for any $s \in S_0$, since 
$s \in \ch{f(\Gamma)}$, 
$[s,\xi^+)$ and $[s,\xi^-)$ are contained in $\ch{f(\Gamma)}$ as we
have seen above. 
Since there is a geodesic line joining $\xi^+$ and $\xi^-$ passing through
$s$, $[s,\xi^+) \cup [s,\xi^-)$ must coincide with the image
$\R \times \{s\}$ of this geodesic line. 
Therefore $\R \times S_0 \subset \ch{f(\Gamma)}$. 
We conclude that $\ch{f(\Gamma)}=\R \times S_0$.

Note that there is an element $g_a \in \isom{\ch{f(\Gamma)}}$ given by
\begin{equation*}
 g_a \colon (r,s)\mapsto (r+a,s), 
 \ (r,s) \in \ch{f(\Gamma)}=\R \times S_0, \ a \in \R, 
\end{equation*}
and recall that an isometry satisfying
\begin{equation*}
 d(g p, p)=d(g q,q)
\end{equation*}
 for any $p,q \in Y$ is called a {\it Clifford tranlation}; $g_a$ is a
 Clifford translation. 
 Let $H \subset \isom{\ch{f(\Gamma)}}$ be the set of all Clifford
 translation on $\ch{f(\Gamma)}$. 
 Then, by 
 \cite[p.~235 6.15 Theorem (3), (4), (5), (6)]{bridson-haefliger}, 
 $H$ is an abelian group and it can be identified with a Hilbert space,
 and $\ch{f(\Gamma)}$ splits as $\ch{f(\Gamma)}=H\times Z$; 
 each fiber isometric to $H$ is given as the $H$-orbit $Hx$ of 
 $x \in \ch{f(\Gamma)}$. 
 Since we have $g_a$, $H$ is a non-trivial Hilbert space.  
 Note that there is no Clifford translation in $\isom{Z}$, and hence $Z$
 does not admit any splitting $Z=\R \times Z'$.

 According to the proof of 
 \cite[p.~235, 6.15 Theorem (6)]{bridson-haefliger}, 
 for any $x \in \ch{(f(\Gamma))}$ and $g \in \isom{\ch{f(\Gamma)}}$,
 $g$ maps $Hx$ to $Hgx$. 
 Setting $x=(h,y) \in H \times Z$ and $y'=\nrprj_Z(gx)$,  we see that
 $g (H\times \{y\}) = H\times \{y'\}$. 
 Thus, by \cite[p.~56, 5.3 Proposition (4)]{bridson-haefliger},
 we see that $g$ splits as $g=g_1\times g_2$, where
 $g_1 \in \isom{H}$ and $g_2 \in \isom{Z}$. 
 Then it is easy to see that, along this splitting, 
 $\rho \colon \Gamma \rightarrow \isom{f(\Gamma)}$ splits
 as $\rho=\rho_1 \times \rho_2$, where
 $\rho_1 \colon \Gamma \rightarrow \isom{H}$ and 
 $\rho_2 \colon \Gamma \rightarrow \isom{Z}$. 

Note that, along the splitting $\rho=\rho_1\times \rho_2$, a
$\rho$-equivariant $\mu$-harmonic map $f$ also splits as
$f=f_1\times f_2$; setting $f(e)=(f_1(e), f_2(e))$ and 
$f_i(\gamma)=\rho_i(\gamma)f_i(e)$, $i=1,2$, we see that
$f(\gamma)=(f_1(\gamma),f_2(\gamma))$. 
In particular, each $f_i$
is a $\rho_i$-equivariant map for $i=1,2$. 
If there is a $\rho_i$-equivariant map $f'_i$ with strictly smaller
$\mu$-energy than $f_i$, then replacing $f_i$ by $f_i'$ and taking 
$f'=(f_1', f_2')$ decreases the energy of $f$ strictly.  
This contradicts the fact that $f$ is $\mu$-harmonic.
Hence $f_i$ is a $\rho_i$-equivariant $\mu$-harmonic map. 
Since the projection onto $Z$ does not increase the distance, 
$l_{\rho_2}(\Gamma)=0$ follows from the assumption
 $l_{\rho}(\Gamma)=0$.
\end{proof}

\medskip
Now we finish the proof of Theorem~\ref{thm:main-1} for proper $Y$.

\medskip

\noindent
{\it Proof of Theorem~\ref{thm:main-1}.}
Under our assumption, by Proposition~\ref{prop:mu-harmonicity}, 
there exists $\omega \in \Omega$ such that 
a subsequence $\{f(\gamma_j)\}$ of $\{f(\gamma_n(\omega)^{-1})\}$
converges to $\xi(\omega) \in \partial Y$, and 
\begin{equation*}
 \varphi_{\omega}(\gamma)= 
 \begin{cases}
  d(\xi(\omega),f(\gamma)) - d(\xi(\omega),f(e)) & \text{if }\xi(\omega)
  \in Y \\
  b_{\xi(\omega)}(f(\gamma), f(e))  & \text{if }\xi(\omega) \in \partial Y
 \end{cases}
\end{equation*}
satisfies $\Delta \varphi_{\omega}(e)=0$. 
If $\xi(\omega) \in Y$, then Lemma~\ref{lem:core} $(1)$ tells us that
$\rho(\Gamma)$ fixes a point in $Y$. 

Suppose $\xi(\omega) \in \partial Y$. 
Since $Y$ is proper, taking further subsequence if necessary, we may
assume that $\rho(\gamma_j^{-1})\xi(\omega)$ and $\{f(\gamma_j^{-1})\}$
converge to points $\xi^+, \xi^- \in Y\cup \partial Y$
respectively. 
Since $d(f(e),f(\gamma^{-1}))=d(f(\gamma),f(e))\to \infty$, 
$\xi^- \in \partial Y$ as well as $\xi^+$. 
Then by Lemma~\ref{lem:core} $(4)$, we see that we have
a splitting $\ch{f(\Gamma)}=H \times Z$, 
$\rho_2$-equivariant $\mu$-harmonic map 
$f_2 \colon \Gamma \rightarrow Z$ and $l_{\rho_2}(\Gamma)=0$. 
Note that $H=\R^k$ for some $k \in \N$ and that $Z$ is proper, since
$Y$, and hence $\ch{f(\Gamma)}$, is proper. 
Therefore, by Proposition~\ref{prop:mu-harmonicity}, we have 
$\xi \in Z \cup \partial Z$ such that
\begin{equation*}
 \varphi(\gamma) =
 \begin{cases}
  d(\xi,f(\gamma))-d(\xi,f(e)) & \text{if } \xi \in Z \\
  b_{\xi}(f(\gamma),f(e)) & \text{if } \xi \in \partial Z
 \end{cases}
\end{equation*}
satisfies $\Delta \varphi(e)=0$ again. 
If $\xi \in \partial Z$, then Lemma~\ref{lem:core} $(4)$ tells us that
$Z$ must admit further splitting $Z = \R^{k'}\times Z'$.
However, by our choice of $Z$ in Lemma~\ref{lem:core} $(4)$, this is
impossible. 
Thus $\xi$ must lie in $Z$, and Lemma~\ref{lem:core} $(1)$ implies that
there exists a point $p \in Z$ fixed by $\rho_2(\Gamma)$. 
Then we see that $\R^k \times \{p\}$ is left invariant under the action
of $\rho(\Gamma)=\rho_1(\Gamma)\times \rho_2(\Gamma)$, which means the
existence of $\rho(\Gamma)$-invariant flat subspace. 
This completes the proof of Theorem~\ref{thm:main-1} for proper $Y$.
\qed

\section{Proofs of 
Theorem \ref{thm:main-1} for nonproper $\cat{0}$
spaces and Thorem~\ref{thm:main-3}}
\label{sec:when-y-not}

In what follows, we consider the case that $Y$ is not proper. 
Before proceed, recall that we have used the properness of $Y$ to 
deduce:
\begin{itemize}
 \item $\mu$-harmonic map exists when $\rho(\Gamma)$ does not fixes a
       point in $\partial Y$, 
 \item $\{f(\gamma_j(\omega)^{-1})\}$ has a convergent subsequence 
       $\{f(\gamma_j)\}$, 
 \item $\{f(\gamma_j^{-1})\}$ has a convergent subsequence
       $\{f(\gamma_j')\}$, 
 \item $\{\rho(\gamma_j')\xi(\omega)\}$ has a convergent subsequence. 
\end{itemize}
We have to make a detour or find a substitute for these arguments. 
It is certainly possible when $Y$ is of finite telescopic dimension, and
almost possible when $Y$ is locally finite-dimensional as we will see in
\S \ref{sec:when-y-not} and \S \ref{sec:key-lemma-proof}.  

\subsection{A quick review of ultralimits}
\label{sec:quick-revi-ultr}
In what follows, we will frequently use the ultralimit of a sequence of
metric spaces. We start with a quick review of ultralimits.

%
%
Let $\lambda$ be a nonprincipal ultrafilter on $\N$, that is, a finitely
additive probablity measure on $\N$ satisfying
\begin{itemize}
 \item every $S \subset \N$ is $\lambda$-measurable, 
 \item for any $S \subset \N$ we have $\lambda (S)\in \{0,1\}$, 
 \item if $\# S <\infty$, then $\lambda (S)=0$. 
\end{itemize}
 See, for example, \cite[5.48 Exercise]{bridson-haefliger} for the
 existence of a nonpricipal ultrafilter. 
 For any bounded sequence $\{a_j\}$ of real numbers, there exists 
 $l \in \R$ such that, for any $\varepsilon>0$, 
 $\lambda (\{j \in \N \mid |a_j-l|<\varepsilon\})=1$, which we express
 as $l=\ulim_j a_j$. 

 Let $Y_j$ be a metric space, and $p_j \in  Y_j$, $j\in \N$. 
 Let $Y^{\infty}$ be a set of sequences $\{x_j\}$, $j \in \N$, such that 
 $d_{Y_j}(p_j,x_j)$ is bounded independent of $j \in \N$.
 We set $\{x_j\}\sim \{y_j\}$ if $\{x_j\}$ and $\{y_j\}$ satisfy 
 $\ulim_j d(x_j,y_j)=0$.  Then $\sim$ becomes an equivalence relation on
 $Y^{\infty}$, and we denote the set of equivalence classes by 
 $Y_{\infty}=\ulim_j (Y_j,p_j)$.  
 We also denote $x=\{x_j\}\in Y_{\infty}$ by $x=\ulim_j x_j$. 
 It is obvious that $d_{\lambda}(\{x_j\},\{y_j\})=\ulim_j d(x_j,y_j)$
 becomes a metric on $Y_{\infty}$. 
 Let $q_j$ be a point in $Y_j$ whose distance from $p_j$ is bounded
 independently of $j \in \N$, and $c_j\colon [0,d_j]\rightarrow Y_j$ a
 geodesic with $c_j(0)=q_j$.  
 Then, for each $t \in [0,\ulim_j d_j]$, $d(c_j(t),p_j)$ is
 bounded independently of $j\in \N$, and hence we get
 $c_{\infty}(t)=\ulim_j c_j(t)$. 
 Since $d(c_j(t),c_j(t'))=|t-t'|$ for any $j \in N$, we see that 
 $c_{\infty} \colon [0,\ulim_j d_j]\rightarrow Y_{\infty}$ turns out to
 be a geodesic.  
 Then it is easy to see that if each $Y_j$ is a complete $\cat{0}$
 space, then $Y_{\infty}=\ulim_j (Y_j,p_j)$ becomes a complete $\cat{0}$
 space.
 Note that, even when $d_j\nearrow\infty$,  we can define a geodesic ray 
 $c_{\infty}\colon \ray{0,\infty}\rightarrow Y_{\infty}$ as an
 ultralimit of a sequence of  geodesics 
 $\{c_j\colon [0,d_j]\rightarrow Y_j\}$. 
 In this case, we have a point 
 $\xi=c_{\infty}(\infty) \in \partial Y_{\infty}$. 
 We often express this $\xi$ as $\xi=\ulim_j c_j(d_j)$. 

 In what follows, we mainly consider an ultralimit of the form 
 $Y_{\infty}=\ulim (Y,o)$ for fixed $Y=(Y,d)$ and $o \in Y$, and 
 we use the same symbol $d$ for the metric on $Y_{\infty}$ instead of
 $d_{\lambda}$ for the sake of simplicity. 
 It can be shown that if $Y=(Y,d)$ is proper, then 
  $Y_{\infty}=(Y_{\infty},d)$ is isometric to $Y=(Y,d)$. 
 However, this is no longer true for nonproper metric space $Y$. 
 In that case, $Y$ becomes a convex closed proper subset of
 $Y_{\infty}$. 
 
 \begin{Example}
 \label{example:larger_ultralmit}
  Let $\lambda$ be a nonprincipal ultrafilter. 
  Consider a metric space $Y$ obtained by attaching a half line $L_n$ to
  $1/n \in \ray{0,\infty}$ for $n=1,2,\dots$. 
  Then $Y$ becomes a $\cat{0}$ space. 
  Let $t>0$ and take a point $p_n^t \in L_n$ so that
  $d(p_n^t,\ray{0,\infty})=t$. 
  Choose a basepoint 
  $o$ to be the endpoint $0$ of $\ray{0,\infty}$ to which we 
  attach $L_n$'s, and consider an ultralimit 
  $Y_{\infty}=\ulim_n (Y,o)$. 
  Then there must be a point $p^t=\ulim_n p_n^t$ in $Y_{\infty}$. 
  However, since, for any $q \in Y$,  $d(p_n^t,q)\geq t$ holds unless 
  $q \in L_n$, we see that 
  $\lambda( \{n \in \N \mid d(p_n^t,q)\geq t\})=1$, and hence $p^t$
  cannot be a point in $Y$. 
  It is clear that, for each $t>0$, $d(p^t,o)$ must be equal to
  $t$,  and, for $t\not=t'$, $d(p^t,p^{t'})=|t-t'|$ must hold. 
  Therefore $\{p^t\mid t \in \ray{0,\infty}\}$ is a half line $L$
  attached to $o$, the endpoint of $\ray{0,\infty}$, and $L$
  does not exist in $Y$.  
  Thus we see $Y \not=Y_{\infty}$. 
  Take an ultralimit one more time and consider  
  $(Y_{\infty})_{\infty}=\ulim_n (Y_{\infty},o)$. 
  Looking at $\{p_n^t\}$ again, we see that, since 
  $d(p_n^t,p^t)\geq 2t$ for any $n$, an ultralimit of this sequence
  cannot be $p^t$. 
  Indeed, we see that $\ulim_n p_n^t$ cannot be any point in $Y_{\infty}$
  as before, and there must be another half line $L'$ containing 
  $\ulim_n p_n^t$ and attached to $o$, which does not
  exist in $Y_{\infty}$.  
  Therefore we get $(Y_{\infty})_{\infty}\not= Y_{\infty}$ again. 
  Repeating this, we get larger and larger metric spaces. 
  This observation leads us to an argument often used in the proof of
  our theorems in what follows. 
 \end{Example}

%
%
\begin{Remark}
 Although $Y_{\infty}$ can be larger than $Y$ as the example above
 shows,  $Y_{\infty}$ inherits a certain dimension bound as we will see
 in Proposition~\ref{prop:dimension_of_ulim}, 
 which plays an important role in the proof of Theorem
 \ref{thm:main-1}. 
\end{Remark}

%
%
 Let $\rho\colon \Gamma \rightarrow \isom{Y}$ be a homomorphism. 
 For $p=\ulim_j p_j \in Y_{\infty}=\ulim (Y,o)$, we see that 
\begin{equation*}
 d(\rho(\gamma)p_j,o) \leq d(\rho(\gamma)p_j, \rho(\gamma)o) + 
 d(\rho(\gamma)o,o) = d(p_j,o)+d(\rho(\gamma)o,o)
\end{equation*}
 is bounded independent of $j\in \N$,  and hence we have
 $\ulim_j \rho(\gamma)p_j \in Y_{\infty}$. 
 Therefore,  the action of $\Gamma$ on $Y$ given by 
 $\rho$ can be extended to that on $Y_{\infty}$ by setting 
 $\rho(\gamma)(\ulim_j p_j)=\ulim_j \rho(\gamma)p_j$. 
 It is clear that this gives a homomorphism from $\Gamma$ into
 $\isom{Y_{\infty}}$, which we denote by the same symbol $\rho$. 

%
%
Also, as one easily imagine,  $\mu$-energy $\ene{\mu}$ behaves well under
taking ultralimit.
Although the proof is standard for experts, we give a proof of the
following proposition for the sake of completeness. 

\begin{Proposition}
 \label{prop:energy_limit}
 Let $Y$ be a $\cat{0}$ space, $o\in Y$, and 
 $\lambda$ a nonprincipal ultrafilter. 
 Suppose that a countable group $\Gamma$ acts on $Y$ via a
 homomorphism $\rho\colon \Gamma \rightarrow \isom{Y}$, and that the
 action given by $\rho$ has finite second moment with respect to $\mu$.
 Let $f_q \colon \Gamma \rightarrow Y$ be a $\rho$-equivariant map
 defined by $f_q(\gamma)=\rho(\gamma)q$, and 
 $Y_{\infty}=\ulim (Y,o)$. 
 Suppose that $p= \ulim_j p_j$ for a bounded sequence
 $\{p_j\}\subset Y$.  Then we have
 \begin{equation*}
  \ene{\mu} (f_p) = \ulim_j \ene{\mu}(f_{p_j}). 
 \end{equation*} 
\end{Proposition}

\begin{proof}
 Since
 \begin{equation*}
 \begin{split}
   d(p_j, \rho(\gamma)p_j) & \leq d(o, p_j) + d(o, \rho(\gamma)o)
    + d(\rho(\gamma)o, \rho(\gamma)p_j) \\
 & = 2d(o, p_j) + d(o, \rho(\gamma)o), 
 \end{split}
 \end{equation*}
 and $\{p_j\}$ is bounded, we have
 \begin{equation*}
 d(p_j, \rho(\gamma)p_j)^2 \leq C_1 d(o, \rho(\gamma)o)^2 + C_2 
 \end{equation*}
 for some positive constant $C_1$ and $C_2$. 
 Since $\Gamma$ is countable, we can take an increasing sequence 
 $D_1 \subset D_2 \subset \dots \subset D_n \subset \dots$ of finite
 subsets of $\Gamma$ so that $\bigcup_{n=1}^{\infty}D_n=\supp \mu$. 
 Then, by the inequality above, we have
 \begin{equation*}
 \int_{\Gamma \setminus D_n} d(p_j, \rho(\gamma)p_j)^2 d\mu(\gamma)
  \leq \int_{\Gamma \setminus D_n} C_1 d(o, \rho(\gamma)o)^2 + C_2 
   \ d\mu(\gamma)
 \end{equation*}
 for any $j \in \N$, and also
 \begin{equation*}
 \int_{\Gamma \setminus D_n} d(p, \rho(\gamma)p)^2 d\mu(\gamma)
 \leq \int_{\Gamma \setminus D_n} C_1 d(o, \rho(\gamma)o)^2+ C_2  
   \ d\mu(\gamma). 
 \end{equation*}
 Note that the right-hand side tends to $0$ as
 $n\to \infty$ since the action given by $\rho$ has finite second
 moment with respect to $\mu$.
 Therefore, for any $\varepsilon>0$,
 there exists $n$ such that
 \begin{equation*}
 \frac{1}{2} \int_{\Gamma \setminus D_n} d(p_j, \rho(\gamma)p_j)^2 
  d\mu(\gamma) < \frac{\varepsilon}{4} 
 \end{equation*}
 for any $j \in \N$, and 
 \begin{equation*}
 \frac{1}{2} \int_{\Gamma \setminus D_n} d(p, \rho(\gamma)p)^2 
  d\mu(\gamma) < \frac{\varepsilon}{4} 
 \end{equation*}
 hold. 
 On the other hand, since 
 $\ulim_j d(p_j, \rho(\gamma)p_j)^2=d(p,\rho(\gamma)p)^2$ and $\lambda$ is
 a nonprincipal ultrafilter, letting 
 \begin{equation*}
 A_{\gamma}=\left\{k \in \N \mid 
 |d(p_k,\rho(\gamma)p_k)^2-d(p,\rho(\gamma)p)^2|<\varepsilon 
 \right\}, 
 \end{equation*}
 we have $\lambda(A_{\gamma})=1$ for each $\gamma \in \Gamma$.  
 Since $D_n$ is a finite set, we obtain 
 $\lambda (\bigcap_{\gamma \in D_n} A_{\gamma}) =1$. 
 Note that, for any $j \in \bigcap_{\gamma \in D_n} A_{\gamma}$, 
 \begin{equation*}
 \left|\frac{1}{2}\int_{D_n} d(p_j,\rho(\gamma)p_j)^2 d\mu(\gamma)
     - \frac{1}{2}\int_{D_n} d(p,\rho(\gamma)p)^2 d\mu(\gamma) \right|
 < \frac{\varepsilon}{2}
 \end{equation*}
 holds.  Therefore, for any $\varepsilon>0$, by taking $n \in \N$ as
 above, we see that
 \begin{equation*}
 \begin{split}
  \left|\ene{\mu}(f_{p_j}) - \ene{\mu}(f_p)\right|
 & = \frac{1}{2} \left| 
    \int_{\Gamma \setminus D_n} d(p_j, \rho(\gamma)p_j)^2 d\mu(\gamma) 
    + \int_{D_n} d(p_j, \rho(\gamma)p_j)^2 d\mu(\gamma) \right. \\
 & \phantom{==}   \left.
    - \int_{\Gamma \setminus D_n} d(p, \rho(\gamma)p)^2 d\mu(\gamma) 
    - \int_{D_n}d(p, \rho(\gamma)p)^2 d\mu(\gamma) 
    \right| \\
 & \leq  \frac{1}{2} \int_{\Gamma \setminus D_n} 
        d(p_j, \rho(\gamma)p_j)^2 d\mu(\gamma)
    + \frac{1}{2}\int_{\Gamma \setminus D_n} 
       d(p, \rho(\gamma)p)^2 d\mu(\gamma)  \\
 & \phantom{==}
    + \frac{1}{2}
       \left|\int_{D_n} d(p_j, \rho(\gamma)p_j)^2 d\mu(\gamma) 
       - \int_{D_n}d(p, \rho(\gamma)p)^2 d\mu(\gamma) 
      \right|\\
 & < \frac{\varepsilon}{4} + \frac{\varepsilon}{4} +
  \frac{\varepsilon}{2}= \varepsilon
 \end{split}
 \end{equation*}
 holds for any $j \in \bigcap_{\gamma \in D_n}A_{\gamma}$. 
 This completes the proof. 
\end{proof}

%
%
Recall that $\rho(\Gamma) \subset \isom{Y}$ is said to be 
{\it reductive in the sense of Jost}
if there exists a closed, convex $\rho(\Gamma)$-invariant subset $C$ of
$Y$ satisfying either
\begin{itemize}
 \item $C$ is isometric to a finite-dimensional Euclidean space, or
 \item for any unbounded sequence $\{p_n\} \subset C$ , there exists
       $\gamma \in \Gamma$ such that $\{d(p_n, \rho(\gamma)p_n)\}$ is
       unbounded. 
\end{itemize}
If $C$ is isometric to a finite-dimensional Euclidean space, then we can
regard $C$ as $Y$. 
The existence of $\rho$-equivariant harmonic map in this case follows
from, for example, an argument similar to that in \cite{labourie}. 
Suppose that we have the second case in the definition of reductiveness
in the sense of Jost. 
Take any sequence $\{f_{p_j}\}$ that minimizes $\mu$-energy. 
Suppose that $\{p_j\}$ is unbounded. 
Then there exists $\gamma \in \Gamma$ such that 
$\{d(p_j,\rho(\gamma)p_j)\}$ is unbounded. 
Since $\supp \mu$ generates $\Gamma$, there exists 
$s_1, \dots, s_k \in \supp \mu$ such that $\gamma=s_1\dots s_k$. 
Since
\begin{equation*}
 d(p_j, \rho(\gamma)p_j) \leq
 d(p_j, \rho(s_1)p_j) + \dots + d(p_j,\rho(s_k)p_j)
\end{equation*}
as we have seen in the proof of Lemma~\ref{lem:integral}, 
$\{d(p,\rho(s_j)p)\}$ is unbounded for some $j$, and we conclude that
$\ene{\mu}(f_{p_j})$ is unbounded.  
A contradiction. 
Now since $\{p_j\}$ is bounded, by the proposition above,  we see that
$f_p$ for $p=\ulim p_j$ minimizes $\mu$-energy, and hence it is
a $\mu$-harmonic map into $Y_{\infty}$.
Note that $Y$ is a convex closed subset in $Y_{\infty}$. 
Then the nearest point projection $\nrprj_Y$ onto $Y$ is
$\rho$-equivariant, and does not increase
the distance, and hence the energy.  
Therefore $\nrprj\circ f_{p}$ is also
$\mu$-harmonic; we may assume $f_{p}$ is a map into $Y$. 
Summarizing this gives the following. 

\begin{Corollary}
\label{cor:existence}
 If $\rho(\Gamma)$ is 
 reductive in the sense of Jost, then there is a $\rho$-equivariant
 harmonic map $f \colon \Gamma \rightarrow Y$. 
\end{Corollary}

\begin{Remark}
 \label{rem:korevaar-schoen}
 The corollary above also follows from 
 \cite[Proposition 1.2]{korevaar-schoen2} without using ultralimits. 
\end{Remark}

%
%
The following lemma tells us that Busemann functions also behave
well under taking ultralimits. 

\begin{Lemma}
 \label{lem:ultralimit_of_buseman_fct}
 Take $o \in Y$. 
 Suppose that a sequence $\{x_j\} \subset Y$ satisfies
 $d(o,x_j)\to\infty$. 
 Then we have $\ulim_j x_j = \xi \in \partial Y_{\infty}$, where 
 $Y_{\infty}=\ulim (Y,o)$. 
 Furthermore,  we have, for any $z \in Y$, 
\begin{equation*}
 b_{\xi}(z,o)=\ulim \left(d(x_j,z)-d(x_j,o)\right). 
\end{equation*}
\end{Lemma}

\begin{proof}
 Take $z \in Y$.   
 Consider a geodesic ray $c$ with $c(0)=o$. 
 Applying \cite[p.~269, 8.21 Lemma]{bridson-haefliger} by setting
 $\rho=d(o,z)$, 
 we see that, for any $\varepsilon/3>0$, there exists $r>0$ such that
 \begin{equation*}
  d(c(r),z) -d(c(t),z) + (t-r) < \varepsilon/3
 \end{equation*}
 holds for any geodesic $c$ starting from $o$ and $t\geq r$.  
 This inequality can be rewritten as
 \begin{equation*}
  \left(d(c(r),z)-r\right) -
   \left(d(c(t),z)-t\right) < \varepsilon/3.
 \end{equation*}
 Since $t\geq r$ is arbitrary and $d(c(t),z)-t$ is easily seen to be
 nonincreasing by the triangle inequality 
 (\cite[p.~268, 8.18 Lemma (1)]{bridson-haefliger}),
 we see that
 \begin{equation*}
  |\left(d(c(r),z)-r\right) - b_{c(\infty)}(z, o)|
   \leq \varepsilon/3
 \end{equation*}
 for any geodesic ray $c$ starting from $o$. 
 Let $c_j$ be a geodesic joining $o$ and $x_j$ with $c_j(d_j)=x_j$, and
 set 
\begin{equation*}
  D(r)= \{j \in \N \mid d_j \geq r\}. 
\end{equation*}
 By our assumption, $d_j \to \infty$, and hence $\lambda(D(r))=1$. 
 Note that, as we have seen before, we have an ultralimit 
 $c_{\infty}\colon \ray{0,\infty}\rightarrow Y_{\infty}$ of $\{c_j\}$,
 and $\xi=c_{\infty}(\infty)=\ulim_j x_j \in \partial Y_{\infty}$. 
 Set
\begin{equation*}
  C(r,\varepsilon)  = \{j \in \N \mid 
   |d(c_j(r),z)-d(c_{\infty}(r),z)|
  < \varepsilon/3\}. 
\end{equation*}
 Then $\lambda(C(r,\varepsilon))=1$ for any $r$ and $\varepsilon$. 
 Thus $\lambda(D(r) \cap C(r,\varepsilon))=1$,
and for any $j \in D(r) \cap C(r,\varepsilon)$,
\begin{equation*}
\begin{split}
  & |d(c_j(d_j), z)-d_j -
 b_{\xi}(z,o)| \\
  \leq &
  |\left(d(c_j(d_j), z)-d_j\right)
 - \left(d(c_j(r), z)-r\right)| \\
 & +|\left(d(c_j(r),z)-r\right)-
  \left(d(c_{\infty}(r), z)-r\right)| \\
 & + |\left(d(c_{\infty}(r), z)-r\right) -
 b_{\xi}(z,o)| \\
 < &  \varepsilon/3  + 
 |d(c_j(r), z)-d(c_{\infty}(r), z)| +
 \varepsilon/3 < \varepsilon
\end{split}
\end{equation*}
 holds.  Since $\varepsilon>0$ is arbitrary, we get
\begin{equation*}
 \ulim  \left(d(c_j(d_j), z) - d_j\right)
  = b_{\xi}(z,o), 
\end{equation*}
 and this completes the proof. 
\end{proof}

\subsection{Existence of a $\mu$-harmonic function}
\label{sec:global_harmonicity}

 We have seen that there is a function $\varphi_{\omega}$ which
 satisfies $\Delta \varphi_{\omega}(e)=0$ in
 \S~\ref{sec:rate_of_escape_and_hmap} when $Y$ is proper. 
 In this section, we show that, even when $Y$ is not proper,  
 there does exist a function which is $\mu$-harmonic on whole $\Gamma$
 under the assumption that $\mu(\gamma)\not= 0$ for any 
 $\gamma \in \Gamma \setminus \{e\}$. 
 We will use this fact in \S~\ref{sec:key-lemma-proof}. 

%
%
 Suppose that $\Gamma$ acts on $\cat{0}$ space $Y$ via a homomorphism
 $\rho\colon \Gamma \rightarrow \isom{Y}$ and that the action given by
 $\rho$ has finite second moment. 
 Fix $o \in Y$. 
 Let $\varphi\colon \Gamma \rightarrow \R$ be a function satisfying 
 $\varphi(\gamma) \leq C_1 d(o,\rho(\gamma)o)^2+C_2$ for some
 $C_1,C_2>0$. 
 Since the action give by $\rho$ has finite second moment with respect
 to $\mu$, it is also true for $\mu^2$ as we have seen
 in \S~\ref{sec:rate_of_escape_and_hmap}. 
 Noting that $d(p,q) \leq 1+d(p,q)^2$ holds for any $p,q \in Y$,  we see
 that
\begin{equation*}
\begin{split}
  |\varphi(\gamma \gamma')| &\leq C_1 d(o,\rho(\gamma \gamma')o)^2 + C_2 \\
  & \leq C_1
 \left(d(o,\rho(\gamma)o)+d(\rho(\gamma)o,\rho(\gamma\gamma')o)\right)^2
 + C_2 \\
  & \leq C_1'd(o,\rho(\gamma')o)^2 + C_2'
\end{split}
\end{equation*}
 for some positive constant $C_1', C_2'$ which depend on $\gamma$. 
 Thus, fixing $\gamma \in \Gamma$, 
 we see that $\gamma' \mapsto \varphi(\gamma \gamma')$ is
 $\mu^2$-integrable by our assumption. 
 In the rest of this section, a function $\varphi$ under consideration
 is always assumed to satisfy
 $|\varphi(\gamma)| \leq C_1 d(o,\rho(\gamma)o)^2 + C_2$ for some
 $C_1,C_2>0$.

%
%
 Let us denote by $\Delta_2$ the $\mu^2$-Laplacian, which is defined by 
\begin{equation*}
 \Delta_2 \varphi (\gamma) = 
 \varphi(\gamma)- \int_{\Gamma} \varphi(\gamma\gamma') d\mu^2(\gamma'). 
\end{equation*}
 Then we have
\begin{equation*}
   \begin{split}
    -\Delta_2 \varphi(\gamma) & = 
    \int_{\Gamma} \varphi(\gamma \gamma') - \varphi(\gamma) d\mu^2(\gamma') \\
    & = \int_{\Gamma} \left(\int_{\Gamma} 
      \varphi(\gamma \gamma_1 \gamma_2) - \varphi(\gamma) d\mu(\gamma_2)\right)
      d\mu(\gamma_1) \\
    & = \int_{\Gamma} \left(\int_{\Gamma} 
      \left(\varphi(\gamma \gamma_1 \gamma_2) - \varphi(\gamma \gamma_1)\right)
      + \varphi(\gamma \gamma_1) - \varphi(\gamma)
      d\mu(\gamma_2)\right) d\mu(\gamma_1) \\
    & = \int_{\Gamma}
       \left(
     -\Delta \varphi(\gamma \gamma_1) + \varphi(\gamma \gamma_1) -
      \varphi(\gamma)\right)
      d\mu(\gamma_1) \\
    & =\int_{\Gamma}-\Delta \varphi(\gamma\gamma_1) d\mu(\gamma_1)
      -\Delta \varphi(\gamma), 
  \end{split}
 \end{equation*}
and hence we get the following: 

\begin{Lemma}
 \label{lem:mu^k-subharmonicity}
 Let  $\varphi\colon \Gamma \rightarrow \R$ be a $\mu$-subharmonic
 function. 

 $(1)$ $\varphi$ is $\mu^2$-subharmonic$;$ for any $\gamma \in \Gamma$,
 $-\Delta_2 \varphi(\gamma)\geq 0$ holds. 

 $(2)$ 
 $-\Delta_2 \varphi(\gamma) \geq -\Delta \varphi(\gamma)\geq 0$
 holds. 

 $(3)$ If $-\Delta_2 \varphi(\gamma)=0$ holds for $\gamma \in \Gamma$,
 then $-\Delta \varphi(\gamma)=0$ holds. 
\end{Lemma}

%
%
\begin{Proposition}
 \label{prop:k-step_harmonicity}
 Suppose that $\mu(\gamma)\not= 0$ holds for any 
 $\gamma \in \Gamma \setminus \{e\}$. 
 Let $\varphi \colon \Gamma \longrightarrow \R$ be a $\mu$-subharmonic
 function. 
 If  $-\Delta_2 \varphi(e)=0$ holds, then $\varphi$ is a $\mu$-harmonic
 function, that is, $-\Delta \varphi(\gamma)=0$ holds for any 
 $\gamma \in \Gamma$. 
\end{Proposition}

\begin{proof}
 By the computation above and $-\Delta_2\varphi(e)=0$, we get
\begin{equation*}
    0= -\Delta_2 \varphi(e) 
    = -\Delta \varphi(e) + 
    \int_{\Gamma}-\Delta \varphi (\gamma)  d\mu(\gamma). 
\end{equation*}
 On the other hand, since $\varphi$ is $\mu$-subharmonic, we have 
 $-\Delta\varphi(\gamma)\geq 0$. 
 Therefore, by our assumption that $\mu(\gamma)\not= 0$ for any 
 $\gamma \in \Gamma$, we see that $-\Delta \varphi(\gamma)=0$ for any
 $\gamma \in \Gamma$. 
 This completes the proof.
 \end{proof}

 Now take a $\mu$-harmonic map $f$ and let
 $\varphi(\gamma) = d(f(e), f(\gamma))$. 
 Then $\varphi$ is a $\mu$-subharmonic function by
 Proposition~\ref{prop:pullback-is-subharmonic}. 
 Set
\begin{equation*}
  \alpha_{2,n}:= \int_{\Omega} -\Delta_2 \varphi(\gamma_{2n}(\omega))
  \ d\mathbb{P}(\omega) 
  = \int_{\Gamma} -\Delta_2 \varphi(\gamma) d\mu^{2n}(\gamma). 
\end{equation*}
 Then, by the same computation as in
 \S~\ref{sec:rate_of_escape_and_hmap},  we see that
\begin{equation*}
  L^{2(n+1)}(f)= L^{2n}(f) + \alpha_{2,n}
\end{equation*}
 holds. 
 Thus if $l_{\rho}(\Gamma)=0$, then we have
 $\liminf_n\alpha_{2,n}=0$, and by Lebesgue's convergence theorem and 
 $-\Delta_2 \varphi(\gamma)\geq 0$ which follows from
 Lemma~\ref{lem:mu^k-subharmonicity} $(1)$, we get 
\begin{equation}
 \label{eq:0_for_each_k}
 \liminf_n -\Delta_2 \varphi(\gamma_{2n}(\omega)) = 0
  \quad \text{a.e.}
\end{equation}
 Set 
\begin{equation*}
   A = 
   \{ \omega \in \Omega \mid \ 
   \liminf_n -\Delta_2 \varphi(\gamma_{2n}(\omega))=0 \},
\end{equation*}
 then we see that $\mathbb{P}(A)=1$. 
 Take $\omega \in A$.  
 Then, for any $j$, we can find $n_j \in \N$ such that
\begin{equation*}
  -\Delta_2 \varphi(\gamma_{n_j}(\omega)) \leq 1/j. 
\end{equation*}
 By Lemma~\ref{lem:mu^k-subharmonicity} $(2)$, we also have
\begin{equation*}
  -\Delta \varphi(\gamma_{n_j}(\omega)) \leq 1/j. 
\end{equation*}
 Letting $j\to \infty$, we see that
\begin{equation}
 \label{eq:converges_to_harmonic}
   \lim_{j\to \infty} -\Delta_k \varphi(\gamma_{n_j}(\omega))=0
\end{equation}
    for $k=1,2$.  Let us define
\begin{equation*}
   \varphi_{\omega,n_j} (\gamma) 
  = d(f(\gamma_{n_j}(\omega)^{-1}), f(\gamma))
    - d(f(\gamma_{n_j}(\omega)^{-1}), f(e)). 
\end{equation*}
Note that, by the triangle inequality, $\varphi_{\omega,n_j}$ satisfies
$|\varphi_{\omega,n_j}(\gamma)|\leq d(f(e),f(\gamma))$ for each 
$j \in \N$ and $\gamma \in \Gamma$. 
Thus $\{\varphi_{\omega,n_j}\}$ can be viewed as a sequence in the
product space 
$\prod_{\gamma \in \Gamma}[-d(f(e),f(\gamma)),d(f(e),f(\gamma))]$, 
which is sequentially compact. 
Since the topology of this product space agrees with that given
by the pointwise convergence of real valued functions on $\Gamma$,
by taking a subsequence if necessary, we may
assume that $\{\varphi_{\omega,n_j}\}$ converges pointwise to
$\varphi_{\omega}$.
 Since, by fixing $j$, the second term in the right-hand side of the
 definition of $\varphi_{\omega,n_j}$ is independent of 
 $\gamma\in \Gamma$, we have
\begin{equation*}
 \begin{split}
 \Delta_2 \varphi_{\omega,n_j} (\gamma)
 & = d(f(\gamma_{n_j}(\omega)^{-1}), f(\gamma))
 - \int_{\Gamma}
 d(f(\gamma_{n_j}(\omega)^{-1}), f(\gamma \gamma')) d\mu^2(\gamma') \\
 & = d(f(e), f(\gamma_{n_j}(\omega)\gamma)) - 
 \int_{\Gamma}d(f(e), f(\gamma_{n_j}(\omega)\gamma\gamma'))d \mu^2(\gamma') \\
 & = \Delta_2 \varphi(\gamma_{n_j}(\omega)\gamma), 
 \end{split}
\end{equation*}
 in particular, we get
\begin{equation*}
 \Delta_2 \varphi_{\omega, n_j} (e)
 = \Delta_2 \varphi(\gamma_{n_j}(\omega)). 
\end{equation*}
 Since 
\begin{equation*}
 |\varphi_{\omega,n_j}(\gamma)|
 =|d(f(\gamma_j), f(\gamma)) - d(f(\gamma_j), f(e))|
 \leq d(f(e), f(\gamma)), 
\end{equation*}
 where $\gamma_j=\gamma_{n_j}(\omega)^{-1}$, 
 and $\int_{\Gamma} d(f(e),f(\gamma)) d\mu^2(\gamma) < \infty$,
 we see that
\begin{equation*}
 \lim_{j\to \infty} \int_{\Gamma} \varphi_{\omega,n_j}(\gamma)
  d\mu^2(\gamma)
  = \int_{\Gamma} \varphi_{\omega}(\gamma) d\mu^2(\gamma)
\end{equation*}
 by Lebesgue's dominated convergence theorem. 
 Combining with (\ref{eq:converges_to_harmonic}), we get 
\begin{equation*}
 \Delta_2 \varphi_{\omega}(e)
  =\lim_j \Delta_2 \varphi_{\omega,n_j}(e) = 0. 
\end{equation*}
 Then, by Proposition~\ref{prop:k-step_harmonicity}, we see
 that $\Delta \varphi_{\omega}(\gamma)=0$ for any $\gamma \in \Gamma$. 
 Now let us denote $\xi(\omega)=\ulim_j f(\gamma_j)$. 
 If $\xi(\omega) \in Y_{\infty}$, then, since $\varphi_{\omega,n_j}$
 converges pointwise to $\varphi_{\omega}$ while 
 $\ulim_j\varphi_{\omega,n_j}(\gamma)=d(\xi(\omega),f(\gamma))-
 d(\xi(\omega),f(e))$, 
 we see that 
 $\varphi_{\omega}(\gamma)=d(\xi(\omega),f(\gamma))-d(\xi(\omega),f(e))$. 
 On the other hand, if $\xi(\omega)\in \partial Y_{\infty}$, then 
 Lemma~\ref{lem:ultralimit_of_buseman_fct} tells us that 
 $\varphi_{\omega}(\gamma)=b_{\xi(\omega)}(f(\gamma),f(e))$. 
 Thus we get a $\mu$-harmonic function of the form
\begin{equation*}
  \varphi_{\omega}(\gamma)=
 \begin{cases}
   d(f(\gamma),\xi(\omega)) - d(f(e),\xi(\omega)) & \text{if } 
  \xi(\omega) \in Y_{\infty}, \\
   b_{\xi(\omega)}(f(\gamma),f(e)) & \text{if } 
  \xi(\omega) \in \partial Y_{\infty}. 
 \end{cases}
\end{equation*}
What we have shown is the following generalization of
Proposition~\ref{prop:mu-harmonicity}. 

\begin{Proposition}
\label{prop:global-mu-harmonicity}
 Let $Y=(Y,d)$ be a $\cat{0}$ space. 
 Take $o \in Y$ and set $Y_{\infty}=\ulim (Y,o)$. 
 Let $\Gamma$ be a countable group equipped with a symmetric
 probability measure $\mu$ whose support generates $\Gamma$. 
 Suppose that $\Gamma$ acts on $Y$ via a homomorphism 
 $\rho\colon \Gamma \rightarrow \isom{Y}$ and that 
 there exists a $\rho$-equivariant $\mu$-harmonic map
 $f\colon \Gamma \rightarrow Y$. 
 If $l_{\rho}(\Gamma)=0$, then, for almost all $\omega \in \Omega$,
 there exists 
 $\xi(\omega) \in Y_{\infty}\cup \partial Y_{\infty}$ such that a
 function 
 $\varphi_{\omega} \colon \Gamma \rightarrow \R$ defined by 
 \begin{equation*}
  \varphi_{\omega}(\gamma) =
   \begin{cases}
    d(\xi(\omega), f(\gamma))-d(\xi(\omega), f(e)) 
    & \text{if } \xi(\omega) \in    Y_{\infty} \\
    b_{\xi(\omega)}(f(\gamma), f(e)) 
    & \text{if } \xi(\omega) \in \partial Y_{\infty}
   \end{cases}
 \end{equation*}
 is $\mu$-harmonic$;$ the function $\varphi_{\omega}$ satisfies 
 $\Delta \varphi_{\omega}(\gamma)=0$ for any $\gamma\in \Gamma$. 
\end{Proposition}

\subsection{$\cat{0}$ spaces with certain dimension bounds}
\label{sec:finite_dim_properties}

Let us denote the metric ball of radius $r$ centered at $p$ by
$B(p,r)$. 
The {\it radius} $\rad{Y}{B}$ of $B \subset Y$ is defined to be the
infimum of $r$ such that $B\subset B(p,r)$ for some $p\in Y$. 
It is known that, for any bounded subset $B$ of a $\cat{0}$ space $Y$,
there exists a unique point $p_0$ such that 
$B\subset B(p_0, \rad{Y}{B})$ 
(see \cite[p.~179, 2.7 Proposition]{bridson-haefliger}). 
Such a point $p_0$ is called the {\it circumcenter} of $B$. 

The {\it geometric dimension} of a metric space with curvature bounded
from above is defined in an inductive manner in \cite{kleiner} and
studied there in detail.  
For a $\cat{0}$ space $Y$, 
according to \cite[Theorem 1.3]{caprace1}, $Y$
has geometric dimension at most $n$ if and only if any subset $B$ of
finite diameter satisfies 
\begin{equation}
\label{eq:geom_dim}
 \rad{Y}{B} \leq \sqrt{\frac{n}{2(n+1)}} \diam{B},  
\end{equation}
where $\diam{B}$ denotes the diameter of $B$.
Note that the equality is achieved by a regular Euclidean $n$-simplex. 
We say $Y$ is {\it locally finite-dimensional} if, for each $R>0$ and 
$p \in Y$, there exists $n \in \N$ such that $B(p,R) \subset Y$ has
geometric dimension at most $n$.  (Note that a convex subset of a
$\cat{0}$ space is again a $\cat{0}$ space, and that it inherits the
upper bound on the geometric dimension from the ambient
$\cat{0}$space.  Thus the choice of a point $p$ is irrelevant
here.)
A $\cat{0}$ space $Y$ is said to have {\it telescopic dimension} at most
$n$ if every asymptotic cone of $Y$ has geometric dimension at most
$n$. 
For example, a $\cat{0}$ space $Y$ has telescopic 
dimension at most $1$
if and only if $Y$ is hyperbolic in the sense of Gromov
(\cite{caprace1}). 
Again \cite[Theorem 1.3]{caprace1} tells us that $Y$ has telescopic
dimension at most $n$ if and only if, for any $\delta>0$, there exists
$D>0$ such that every bounded subset $B$ with $\diam{B} > D$
satisfies 
\begin{equation}
\label{eq:tele_dim}
 \rad{Y}{B} \leq 
 \left(\delta +\sqrt{\frac{n}{2(n+1)}} \right) \diam{B}. 
\end{equation}

The first part of the following proposition has already been pointed out
in \cite[Lemma A.2]{caprace1-2}; P.-E.~Caprace and A.~Lytchack used the
term {\it ultracompletion} for the ultralmit of a constant sequence of
pointed metric space $(Y,p_0)$ there.  
We include, however,  the proof of the result here for the sake of
completeness.  

\begin{Proposition}
 \label{prop:dimension_of_ulim}
 Let $Y$ be a $\cat{0}$ space and $o \in Y$.  Denote by
 $Y_{\infty}$ an ultralimit $\ulim (Y,o)$ with respect to a
 nonprincipal ultrafilter $\lambda$. 
 
 $(1)$ If $Y$ has telescopic dimension at most $n$, so does
 $Y_{\infty}$. 

 $(2)$ Suppose that, for any $R>0$, there exsits $n \in \N$ such that
 any $B \subset B(o, R)$ has geometric dimensin at most $n$. 
 Then $Y_{\infty}$ has the same property. 
\end{Proposition}

\begin{proof}
 $(1)$ Assume that $Y$ has telescopic dimension at most $n$. 
 We first show that, for any $\delta>0$, there exists $D>0$ such that any
 finite subset $B \subset Y_{\infty}$ with $\diam{B}>D$ satisfies
 (\ref{eq:tele_dim}). 

 Take any finite subset $B =\{p^{(1)}, \dots, p^{(k)}\}$ in
 $Y_{\infty}$ with $\diam{B}>D$, 
 and fix an expression $p^{(i)}=\ulim_j p^{(i)}_j$ with $p^{(i)}_j\in Y$
 for each $p^{(i)} \in B$. 
 Let $B_j=\{p^{(1)}_j, \dots, p^{(k)}_j\}$ be the set of $j$th elements
 in the expression of each $p^{(i)} \in B$. 
 Since $B_j$ is a finite set, there is $1\leq j(0)< j(1) \leq k$ such
 that $d(p^{(j(0))}_{j}, p^{(j(1))}_{j})=\diam{B_j}$. 
 For $m=0,1$ there exists unique $l_m \in \{1, \dots, k\}$ such that
 $\lambda(\{i \in \N \mid  j(m)=l_m\}) = 1$, 
 since $\lambda$ is an ultrafilter.  
 Then we see that  $\ulim_i p^{(j(m))}_{i}=p^{(l_m)}$ and 
\begin{equation*}
 \ulim \diam{B_j}=\ulim
  d(p^{(j(0))}_i,p^{(j(1))}_i)=d(p^{(l_0)},p^{(l_1)})
  \leq \diam{B}.  
\end{equation*} 
 On the other hand, since, for any 
 $p=\ulim_j p_j,\ p'= \ulim_j p_j' \in B$,
 we have 
\begin{equation*}
 d(p,p') = \ulim_j d(p_j, p_j') \leq \ulim_j \diam{B_j}, 
\end{equation*}
 which implies $\diam{B} \leq \ulim_j \diam{B_j}$.  
 Hence we get $\diam{B}= \ulim_j \diam{B_j}$. 
 In particular, we may assume that $\diam{B_j}> D$ for $\lambda$-almost
 all $j \in \N$. Let $p^{(0)}_j$ be a circumcenter of $B_j$.  
 Take $R>0$ so that $B(o, R)$ contains $B$.  Then we may assume that
 $B_j \subset B(o, R+1)$, and that $p^{(0)}_j$'s lie in $B(o,R+1)$. 
 Thus we obtain a point
 $p^{(0)}= \ulim_j p^{(0)}_j \in Y_{\infty}$. 
 Since, for any $p_j \in B_j$, we have 
 $d(p^{(0)}_j, p_j) \leq \rad{Y}{B_j}$, we see that 
 $d(p^{(0)},p) \leq \ulim_j \rad{Y}{B_j}$ for any 
 $p \in B$, which implies that
 $\rad{Y_{\infty}}{B} \leq  \ulim_j \rad{Y}{B_j}$. 
 Since each $B_j$ is a subset of $Y$ with $\diam{B_j}>D$, we
 have
 \begin{equation*}
  \ulim \rad{Y}{B_j}
   \leq \left( \delta + \sqrt{\frac{n}{2(n+1)}}\right)
  \ulim \diam{B_j}
 \end{equation*}
 by (\ref{eq:tele_dim}), hence we get
 \begin{equation}
 \label{eq:finite_set}
 \rad{Y_{\infty}}{B} 
  \leq \left( \delta + \sqrt{\frac{n}{2(n+1)}}\right) \diam{B}. 
\end{equation}

 Now suppose that $Y_{\infty}$ has telescopic
 dimension greater than $n$.  Then, by \cite[Theorem 7.1]{kleiner},
 there exsits a sequence of pointed subsets $\{(Y_j, p_j)\}$ of
 $Y_{\infty}$ such that $\ulim_j (Y_j,p_j)$ contains $\R^{n+1}$.  
 Therefore we can find 
 $S=\{q^{(0)}, \dots, q^{(n+1)}\}\subset \ulim_j (Y_j,p_j)$ that forms a
 regular Euclidean $(n+1)$-simplex with arbitrarily large diameter. 
 Take $D>0$ for $\delta$  satisfying
\begin{equation}
\label{eq:delta}
 0< \delta < \sqrt{\frac{n+1}{2(n+2)}} - \sqrt{\frac{n}{2(n+1)}}, 
\end{equation}
 and choose $S$ so that $\diam{S}>D$. 
 Expressing $q^{(i)}=\ulim_j q^{(i)}_j$ 
 with $q^{(i)}_j \in Y_j \subset Y_{\infty}$, and setting 
 $S_j=\{q^{(0)}_j, \dots, q^{(n+1)}_j\}$, we see that 
\begin{equation*}
 \rad{\ulim_j(Y_j,p_j)}{S}\leq \ulim \rad{Y_j}{S_j},
  \quad
 \diam{S} = \ulim_j \diam{S_j}, 
\end{equation*}
 and we may assume $\diam{S_j}>D$ as before.  
 Together with (\ref{eq:finite_set}) and (\ref{eq:delta}), we get 
\begin{equation*}
 \rad{\ulim_j (Y_j,p_j)}{S} 
  \leq \left( \delta + \sqrt{\frac{n}{2(n+1)}}\right) \diam{S}
  < \sqrt{\frac{n+1}{2(n+2)}} \diam{S}.   
\end{equation*}
 This contradicts our assumption that $S$ is a regular Euclidean
 $(n+1)$-simplex. 
 Therefore we conclude that $Y_{\infty}$ has telescopic
 dimension at most $n$. 

\medskip

 $(2)$ Now assume that $Y$ is locally finite-dimensional. 
 Fix $R>0$ and take any finite subset 
 $B=\{p^{(1)}, \dots, p^{(k)}\} \subset Y_{\infty}$ so
 that $B \subset B(o,R)$.  
 Fix an expression $p^{(i)}=\ulim_j p^{(i)}_j$ with $p^{(i)}_j \in Y$,  
 and define $B_j$ as before. 
 We may assume $p^{(i)}_j \in B(o, R+1) \subset Y$, and hence
 there exists $n \in \N$ such that $B_j$ satisfies (\ref{eq:geom_dim})
 with $n$. Thus $B$ also satisfies (\ref{eq:geom_dim}) with the same
 $n$ as we have seen above. 

 Suppose that $B(o,R)\subset Y_{\infty}$ has geometric dimension greater
 than $n$. 
 Then, by \cite[Theorem A]{kleiner}, there are sequences 
 $r_j \to \infty$, $S_j \subset B(o,R)$, and $p \in B(o,R)$ such that
 $d(S_j,p)\to 0$ and $r_j S_j$ converges to the unit ball in $\R^{n+1}$
 in the Gromov-Hausdorff topology.  
 Then we can find a sequence of $n+2$ points in $B(o,R)$ which converges
 to the vertices of a regular Euclidean $(n+1)$-simplex by scaling them
 with $r_j$'s.  
 Then this contradicts the fact that any finite set of
 $B(o,R)$ satisfies (\ref{eq:geom_dim}). Thus $B(o,R)$ has
 geometric dimension at most $n$. 
\end{proof}

%
%
 If $Y$ is proper and $\rho(\Gamma)$ does not fix a point in 
 $\partial Y$, then 
 $\rho(\Gamma)$ is reductive in the sense of Jost as we have seen 
 in the end of \S~\ref{sec:harmonic_maps},
 and hence $\{f_j(e)\}$ is bounded for any minimizing sequence $\{f_j\}$.  
 In particular, we can find a harmonic map, by passing to a subsequence,
 as a limit map of $\{f_j\}$ (see also Corollary~\ref{cor:existence}). 
 This is not true in general for nonproper $\cat{0}$ space; we cannot
 conclude that $\rho(\Gamma)$ is 
 reductive in the sense of Jost,  
 assuming that $\rho(\Gamma)$ does not fix a point in $\partial Y$. 
 However we can show the following, which is related to 
 \cite[Lemma A.7]{caprace1-2}.

\begin{Proposition}
 \label{thm:existence_finite_teledim}
 Let $\Gamma$ be a countable group equipped with a symmetric
 and nondegenerate probability measure $\mu$.
 Let $Y$ be a $\cat{0}$ space with finite telescopic dimension
 and $\rho \colon \Gamma \rightarrow \isom{Y}$ a homomorphism. 
 Assume that the action given by $\rho$ has finite second
 moment with respect to $\mu$.
 If $\rho(\Gamma)$ does not fix a point in $\partial Y$, then there
 exists a $\rho$-equivariant $\mu$-harmonic map. 
\end{Proposition}

\begin{proof}
 Denote by $f_p \colon \Gamma \rightarrow Y$ a $\rho$-equivariant map
 with $f(e)=p$.
 First note that, since $p\mapsto \ene{\mu}(f_p)$ is a continuous convex
 function on $Y$ as we have seen in \S~\ref{sec:harmonic_maps}, for any
 positive real number $c$, the sublevel set 
 \begin{equation*}
  E_c:= \{p \in Y \mid \ene{\mu}(f_p) \leq c\}
 \end{equation*}
 is a (possibly empty) closed convex subset of $Y$. 
 Take a decreasing sequence $\{c_j\}$ of real numbers so that 
 $\lim_j c_j = \inf_{p \in Y} \ene{\mu}(f_p)$. 
 Then we obtain a descending sequence of closed convex subset
 $\{E_j=E_{c_j}\}_{j \in \N}$. 
 If $\bigcap_j E_j\not= \emptyset$, then, for any  
 $q \in \bigcap_j E_j$, 
 $\ene{\mu}(f_q)=\inf_{p \in Y} \ene{\mu}(f_p)$, and hece $f_q$ is
 $\mu$-harmonic. 

 Suppose that $\bigcap_j E_j=\emptyset$. 
 Then, since $Y$ has finite telescopic dimension, 
 \cite[Lemma 5.5]{caprace1} tells us that 
 $\bigcap_j \partial E_j$ is a nonempty set whose intrinsic 
 radius is at most $\pi/2$.  
 Since $Y$ has finite telescopic dimension, $\partial Y$ is finite
 dimensional $\cat{1}$ space, and hence by 
 \cite[Proposition 1.4]{balser-lytchak}, the set of intrinsic
 circumcenters of 
 $\bigcap_j \partial E_j\subset \partial Y$ has a unique
 circumcenter $\xi_0$. 
 We will show that $\bigcap_j \partial E_j$ is $\rho(\Gamma)$-invariant
 in what follows. 
 Then the unique circumcenter $\xi_0$ of 
 $\bigcap_j \partial E_j$ must be fixed by $\rho(\Gamma)$. 
 And this completes the proof of the proposition. 

 Take any $\gamma \in \Gamma$ and fix it for a while.  
 Then by our assumption on $\mu$, there exists 
 $s_1, \dots, s_k \in \supp \mu$ such that $\gamma = s_1\dots s_k$. 
 Fix $c_j$,  and take a point $p \in E_j$. 
 Then, since $\ene{\mu}(f_p)\leq c_j$, we see that 
 \begin{equation*}
  d(p, \rho(s_i)p) \leq \sqrt{\frac{2\ene{\mu}(f_p)}{\mu(s_i)}}
  \leq \sqrt{\frac{2c_j}{\min_i \mu(s_i)}} =: C
 \end{equation*}
 for any $p \in E_j$. 
 Therefore we obtain
 \begin{equation*}
  d(p,\rho(\gamma)p) \leq d(p, \rho(s_1)p) + \dots + d(p,\rho(s_k)p)
   \leq Ck, 
 \end{equation*}
 and hence we have
 $d(\rho(\gamma)E_j, E_j) \leq Ck$. 
 This shows that 
 $\rho(\gamma)\partial E_j =\partial\rho(\gamma)E_j\subset\partial E_j$,
 and hence 
\begin{equation*}
 \rho(\gamma)\bigcap_j \partial E_j = \bigcap_j \rho(\gamma)\partial E_j 
  \subset \bigcap_j \partial E_j. 
\end{equation*}
 In the same way, we have 
 $\rho(\gamma^{-1})\bigcap_j \partial E_j \subset \bigcap_j\partial E_j$ 
 and applying $\rho(\gamma)$ to both sides of this inclusion gives us 
 $\bigcap_j \partial E_j \subset \rho(\gamma)\bigcap_j\partial E_j$. 
 Thus we obtain 
 $\rho(\gamma) \bigcap_j \partial E_j = \bigcap_j \partial E_j$. 
 Since this is true for any $\gamma \in \Gamma$, we see that 
 $\bigcap_j \partial E_j$ is left invariant under the action of
 $\rho(\Gamma)$. 
 \end{proof}

\subsection{Proof of Theorem \ref{thm:main-1}}
\label{sec:when-y-finite-dim}

 This subsection is devoted to the proof of Theorem~\ref{thm:main-1}.

 Assume that 
 $Y$ has finite telescopic dimension and that
 $\rho \colon \Gamma \rightarrow \isom{Y}$
 is a homomorphism with $l_{\rho}(\Gamma)=0$. 
 We may also assume that $\supp \mu \supset \Gamma \setminus \{e\}$ as
 we have explained in the 
 beginning of
 \S~\ref{sec:convexity_of_distance_and_Busemann_fct}. 
 Since we are assuming that $\rho(\Gamma)$ fixes no point in 
 $\partial Y$, by Proposition~\ref{thm:existence_finite_teledim}, we have a
 $\rho$-equivariant $\mu$-harmonic map $f\colon \Gamma\rightarrow Y$. 
  Then, by Proposition~\ref{prop:global-mu-harmonicity}, for almost
 every $\omega \in \Omega$,  there exists a subsequence $\{n_j\}$
 such that  $\xi(\omega)=\ulim_j f(\gamma_{n_j}(\omega)^{-1})$ gives us
 a $\mu$-harmonic function of the form
\begin{equation*}
  \varphi_{\omega}(\gamma)=
 \begin{cases}
   d(f(\gamma),\xi(\omega)) - d(f(e),\xi(\omega)) & \text{if } 
  \xi(\omega) \in Y_{\infty}, \\
   b_{\xi(\omega)}(f(\gamma),f(e)) & \text{if } 
  \xi(\omega) \in \partial Y_{\infty}. 
 \end{cases}
\end{equation*}
 Let us denote $\gamma_{n_j}(\omega)^{-1}$ by $\gamma_j$. 
 If $d(f(e),f(\gamma_j))$ is bounded, then we have 
 $\xi(\omega)= \ulim f(\gamma_j) \in Y_{\infty}$. 
 While if $d(f(e),f(\gamma_j))\to \infty$, then we have 
 $\ulim f(\gamma_j) = \xi (\omega) \in \partial Y_{\infty}$. 

%
%
 Assume first that  $\xi(\omega) \in Y_{\infty}$. 
 Then, by Lemma~\ref{lem:core} $(1)$,
 we see that extended action of $\rho(\Gamma)$ on $Y_{\infty}$ has
 a fixed point.
 Hence every $\rho(\Gamma)$-orbit in $Y_{\infty}$ (and hence in $Y$) is
 bounded.  
 Take any bounded orbit in $Y$. 
 Then the circumcenter of this orbit must be fixed by $\rho(\Gamma)$.
 Thus $\rho(\Gamma)$ fixes a point in $Y$. 

%
%
 Now assume that $\xi(\omega) \in \partial Y_{\infty}$. 
 Under our assuption that $Y$ 
 is of finite telescopic dimension, we can
 show that sequences $\{f(\gamma_j)\}$, $\{f(\gamma_j^{-1})\}$ and
 $\{\rho(\gamma_j^{-1})\xi(\omega)\}$ actually converges in the cone
 topology by passing to subsequences.  
 Indeed, we will show the following theorem, which claims a certain
 compactness property of $\cat{0}$ spaces of finite telescopic dimension
 or of locally finite dimension. 

%
%
\begin{Theorem}
 \label{thm:compactness}
 Let $Y$ be a $\cat{0}$ space either of finite telescopic dimension
 or locally finite-dimensional, and $\Gamma$ a countable
 group with a symmetric probability measure $\mu$ 
 with $\supp\mu\supset \Gamma \setminus \{e\}$. 
 Let $\rho\colon \Gamma \rightarrow \isom{Y}$ be a homomorphism, and
 suppose that there exists a $\rho$-equivariant $\mu$-harmonic map 
 $f\colon \Gamma \rightarrow Y$. 
 Let $\{\gamma_j\}\subset \Gamma$ satisfy
 $d(f(e),f(\gamma_j))\to \infty$, and $\xi=\ulim_j f(\gamma_j)$. 
 If a function $\gamma \mapsto b_{\xi}(f(\gamma),f(e))$ is
 $\mu$-harmonic,  then $\xi \in \partial Y$, and there exists a
 subsequence $\{\gamma_{j_k}\}$ of $\{\gamma_j\}$ such that
 $\{f(\gamma_{j_k})\}$ converges to $\xi$ in the cone topology. 
\end{Theorem}

\begin{Remark}
 The theorem above can be regarded as a sort of compactness theorem. 
 As was pointed out in \cite{bader-duchesne-lecureux}, a $\cat{0}$ space
 with finite telescopic dimension has a certain compactness property; 
 if one considers the weakest topology that makes any convex closed
 subset (w.r.t.~usual topology) is closed, then, for a $\cat{0}$ space
 $Y$ with finite telescopic dimension, $Y \cup \partial Y$ becomes
 compact (see also \cite{monod}). 
 Although we do not use this compactificaiton in the present
 paper, there might be some relation between our compactness theorem
 above. 
\end{Remark}

%
%
Now we proceed to the proof of Theorem~\ref{thm:main-1}, assuming
Theorem~\ref{thm:compactness}.  

\medskip

 By Lemma~\ref{lem:ultralimit_of_buseman_fct} and
 Proposition~\ref{prop:global-mu-harmonicity}, we see that a function defined by 
 $\varphi_{\omega}(\gamma)=b_{\xi(\omega)}(f(\gamma),f(e))$
 is $\mu$-harmonic on whole $\Gamma$. 
 Take $\gamma \in \Gamma$,  and let 
 $\rho_{\gamma}\colon \gamma'\mapsto \rho(\gamma\gamma'\gamma^{-1})$. 
 Then, for any $\rho$-equivariant map $f'$, a new map $f_{\gamma}'$
 defined by $f_{\gamma}'\colon \gamma'\mapsto f(\gamma \gamma')$
 satisfies
\begin{equation*}
 \rho_{\gamma}(\gamma')f_{\gamma}'(\gamma'')
 = \rho(\gamma \gamma' \gamma^{-1})f'(\gamma \gamma'')
 = f'(\gamma \gamma' \gamma'') = f_{\gamma}'(\gamma' \gamma''), 
\end{equation*}
 that is, $f_{\gamma}'$ is $\rho_{\gamma}$-equivariant.  Also we
 have
\begin{equation*}
\begin{split}
  E_{\mu}(f_{\gamma}')
  & =\frac{1}{2}\int_{\Gamma} d(f_{\gamma}'(e),f_{\gamma}'(\gamma'))^2
  d\mu(\gamma')
  = \frac{1}{2}\int_{\Gamma} d(f'(\gamma e),f'(\gamma\gamma'))^2
  d\mu(\gamma') \\
  & = \frac{1}{2 }\int_{\Gamma}d(f'(e),f'(\gamma'))^2 d\mu(\gamma')
  = E_{\mu}(f'). 
\end{split}
\end{equation*}
 Therefore we see that $f_{\gamma}$ defined from a $\rho$-equivariant
 $\mu$-harmonic map $f$ becomes a $\rho_{\gamma}$-equivariant
 $\mu$-harmonic map. 
 Note also that 
\begin{equation*}
 \varphi\colon \gamma' \mapsto b_{\xi(\omega)}(f_{\gamma}(\gamma'),f_{\gamma}(e))
  =b_{\xi(\omega)}(f(\gamma \gamma'), f(e)) + 
   b_{\xi(\omega)}(f(e),f(\gamma))
\end{equation*}
 is $\mu$-harmonic, since
 $\varphi_{\omega}\colon \gamma \mapsto b_{\xi(\omega)}(f(\gamma),f(e))$
 is $\mu$-harmonic on whole $\Gamma$ and
 $b_{\xi(\omega)}(f(e),f(\gamma))$ is a constant independent of
 $\gamma'$. 
 Then we see that, 
 by applying Lemma~\ref{lem:core} $(2)$ to
 $f_{\gamma}$ and  $\varphi$, the convex hull of 
 $\ray{f(\gamma),\xi(\omega)}\cup[f(\gamma),f(\gamma \gamma')] 
 \cup \ray{f(\gamma \gamma'),\xi(\omega)}$
 is a half-infinite parallelogram for any 
 $\gamma, \gamma' \in \Gamma$. 

 Now we show that $\xi(\omega)=\ulim_j f(\gamma_j)$ and 
 $\xi^-=\ulim_j f(\gamma_j^{-1})$
 satisfy the assumption of Theorem~\ref{thm:compactness}, and also that 
 there exists a sequence $\{f(\tilde \gamma_j)\}$ such that 
 $\ulim_j f(\tilde \gamma_j)=\xi^+$ and $\{f(\tilde \gamma_j)\}$
 satisfies the assumption of Theorem~\ref{thm:compactness}. 
 Once we accomplish this, the proof completes as follows:  
 We see that each sequence converges to its limit in the cone topology.
 Hence we can apply Lemma~\ref{lem:core} (3), (4) and 
 obtain a splitting $\ch{f(\Gamma)}=H\times Z$.
 Furthermore we have a homomorphism 
 $\rho_2\colon \Gamma \rightarrow \isom{Z}$ with $l_{\rho_2}(\Gamma)=0$
 and $\rho_2$-equivariant $\mu$-harmonic map 
 $f_2 \colon \Gamma \rightarrow Y$. 
 Since $\ch{f(\Gamma)}$ has finite telescopic dimension, so does $Z$.
 Now we can apply the argument we have applied to $Y$ to $Z$.
 Then we see that either $Z$ splits or $\rho_2(\Gamma)$ in
 Lemma~\ref{lem:core} (4) fixes a point $p$ in $Z$.
 However, by our choice of $Z$, the former cannot occur. 
 Thus we conclude that 
 $\rho(\Gamma)=\rho_1(\Gamma)\times \rho_2(\Gamma)$ leaves 
 a flat subspace $H\times \{p\}$ invariant.

 By our choice of $\{f(\gamma_j)\}$ and
 Lemma~\ref{lem:ultralimit_of_buseman_fct}, it is clear that 
 $\{f(\gamma_j)\}$ satisfies the assumption of
 Theorem~\ref{thm:compactness}, and hence, by passing to a subsequence
 if necessary, we may assume that $\{f(\gamma_j)\}$ converges to
 $\xi(\omega)$ in the cone topology. 
 By Lemma~\ref{lem:core} $(2)$, we know that, for any
 $\gamma \in \Gamma$, the convex hull of 
 $\ray{f(e),\xi(\omega)}\cup [f(e),f(\gamma)]\cup
 \ray{f(\gamma),\xi(\omega)}$ 
 is isometric to a half-infinite parallelogram in $\R^2$.

 Now we show that Lemma~\ref{lem:core} $(3)$ is true even if the
 convergences to $\xi^{\pm}$ is given by ultralimits.
 Since $\{f(\gamma_j)\}$ converges to $\xi(\omega)$ in the cone
  topology, by the same argument as in the proof of Lemma~\ref{lem:core}
 $(3)$,
 we see that, for any $\gamma \in\Gamma$, 
 $\angle_{f(\gamma_{j} \gamma)} (\xi(\omega),f(e))\to \pi$. 
 Consider a geodesic 
 $c_j=\cc{f(\gamma_j\gamma)}{\xi(\omega)}$.
 Fix $\gamma \in \Gamma$ and 
 take a geodesic $\bar c_j$ with $\bar c_j(0)=f(\gamma_j \gamma)$ and 
 $\bar c_j(d(f(e),f(\gamma_j \gamma)))=f(e)$. 
 Since $\angle_{f(\gamma_{j} \gamma)} (\xi(\omega),f(e))\to \pi$, 
 $d(c_j(t), \bar c_j(t)) \to 2t$ for any $t>0$. 
 Consider geodesic rays  $c=\ulim_j \rho(\gamma_j^{-1})c_j$ and 
 $\bar c = \ulim_j \rho(\gamma_j^{-1})\bar c_j$ in $Y_{\infty}$. 
 Then $c(0)=\bar c(0)=f(\gamma)$ and $d(c(t),\bar c(t))=2t$ for any 
 $t>0$, and hence $c$ and $\bar c$ form a straight geodesic line in
 $Y_{\infty}$, which we denote by $c_{\gamma}$. 
 This geodesic line $c_{\gamma}$ can be parametrized so that
 $c_{\gamma}(\infty)=\xi^+ = \ulim_j \rho(\gamma_j^{-1})\xi(\omega)$ and 
 $c_{\gamma}(-\infty)=\xi^- = \ulim_j \rho(\gamma_j^{-1})f(e)$, where 
 $\xi^+$ and $\xi^-$ are points in $\partial Y_{\infty}$.  
 Note that both $\xi^+$ and $\xi^-$ do not depend on $\gamma$ by
 definition, and we get a family of geodesic lines joining $\xi^+$ and
 $\xi^-$ in $Y_{\infty}$ passing through each $f(\gamma)$. 

 Then, by 
 \cite[p.~182, 2.13 The Flat Strip Theorem]{bridson-haefliger}, 
 we see that, for any $\gamma, \gamma' \in \Gamma$,
 the convex hull of $c_{\gamma}(\R) \cup c_{\gamma'}(\R)$ is a flat strip
 $\tilde F_{\gamma \gamma'}$. 
 In particular, the intersection of the horosphere centered at $\xi^+$
 and $\tilde F_{\gamma \gamma'}$ is a geodesic segment perpendicular to
 $c_{\gamma}(\R)$ and $c_e(\R)$, and the same is true for $\xi^-$. 
 Therefore we obtain
\begin{equation*}
 b_{\xi^-}(f(\gamma),f(e))=-b_{\xi^+}(f(\gamma),f(e)). 
\end{equation*}
 Since both $\gamma \mapsto b_{\xi^{\pm}}(f(\gamma),f(e))$ are 
 $\mu$-subharmonic, both of them must be $\mu$-harmonic. 
 By Lemma~\ref{lem:ultralimit_of_buseman_fct}, we see that
 $\{f(\gamma_j^{-1})\}$ satisfies the assumption of
 Theorem~\ref{thm:compactness}.  
 Therefore, by passing to a subsequence if necessary, we see that 
 $\{f(\gamma_j^{-1})\}$ converges to $\xi^-$ in the cone topology. 
 
 Now consider $\xi^+ =\ulim_j \rho(\gamma_j^{-1})\xi(\omega)$. 
 Since we know that $f(\gamma_j)$ converges to $\xi(\omega)$ in the cone
 topology,  for any $j$, we can find $n(j) \in \N$ so that 
\begin{equation*}
 \angle_{f(\gamma_j)}(f(\gamma_{n(j)}),\xi(\omega))\leq 1/j,\ 
 \text{ and }d(f(\gamma_j),f(\gamma_{n(j)}))\geq j, 
\end{equation*} 
which implies
\begin{equation*}
 \angle_{f(e)}(f(\gamma_j^{-1} \gamma_{n(j)}),
  \rho(\gamma_j^{-1})\xi(\omega)) 
 \leq 1/j, \ \text{and }\ 
 d(f(e),f(\gamma_j^{-1}\gamma_{n(j)})) \geq j. 
\end{equation*}
 By applying Lemma~\ref{lem:core} $(2)$ to
 $\rho_{\gamma_j}$-equivariant $\mu$-harmonic map $f_{\gamma_j}$, we see
 that the convex hull of 
  $\ray{f(\gamma_j),\xi(\omega)}\cup [f(\gamma_j),f(\gamma_{n(j)})] 
 \cup \ray{f(\gamma_{n(j)}),\xi(\omega)}$
 is a half-infinite prallelogram, hence
 the same is true for the convex hull of 
 $\ray{f(e),\rho(\gamma_j^{-1})\xi(\omega)} \cup 
 [f(e),f(\gamma_j^{-1}\gamma_{n(j)})] \cup 
 \ray{f(\gamma_j^{-1}\gamma_{n(j)}),\rho(\gamma_j^{-1})\xi(\omega)}$.
 Therefore, setting 
 $\tilde \gamma_j= \gamma_j^{-1}\gamma_{n(j)}$ and
 $\xi_j=\rho(\gamma_j^{-1})\xi(\omega)$, we get 
\begin{equation*}
 d(\cc{f(e)}{f(\tilde \gamma_j)}(t), 
 \cc{f(e)}{\xi_j}(t)) \leq
 2t \sin \frac{1}{2j}
\end{equation*}
 for $t \in [0,j]$, and hence
\begin{equation*}
 \ulim_j \cc{f(e)}{f(\tilde\gamma_j)}(t) = \ulim_j \cc{f(e)}{\xi_j}(t)
 = \cc{f(e)}{\ulim_j \xi_j}(t)
\end{equation*}
 for each $t$.  In particular, we see that
 $\ulim_j f(\tilde \gamma_j)=\ulim_j \xi_j =\xi^+$. 
 Since we already know that $\gamma \mapsto b_{\xi^+}(f(\gamma),f(e))$
 is $\mu$-harmonic,
 we see that 
 $\{f(\tilde \gamma_j)\}$ satisfies the assumption of
 Theorem~\ref{thm:compactness}, and in particular 
 $\xi^+ \in \partial Y$. 
 This completes the proof of Theorem \ref{thm:main-1}.  \qed

\subsection{Proof of Theorem~\ref{thm:main-3}}
\label{sec:proof_of_thm_main-3}
 Let $Y$ be a locally finite-dimensional $\cat{0}$ space, and
 $\rho\colon \Gamma \rightarrow \isom{Y}$ a homomorphism. 
 If $\rho(\Gamma)$ is 
 reductive in the sense of Jost, 
 then we know that there exists a $\rho$-equivariant $\mu$-harmonic map
 by Corollary~\ref{cor:existence}. 
 Once the existence of a $\rho$-equivariant $\mu$-harmonic map is
 achieved, 
 since our compactness theorem (Theorem~\ref{thm:compactness})
 is also valid for locally finite-dimensional $\cat{0}$ spaces,
 the rest of the proof of Theorem~\ref{thm:main-3} is 
 completely the same as the proof of Theorem~\ref{thm:main-1} we have
 seen above.

\section{Proof of a compactness theorem} 
\label{sec:key-lemma-proof}

This section is devoted to the proof of Theorem~\ref{thm:compactness}. 
We start with some lemmas on flat subsets in $\cat{0}$ spaces.

\subsection{Some lemmas on flat subsets}
\label{sec:some-lemmas-flat}

%
%
\begin{Lemma}
\label{lem:parallelogram}
 Suppose that $c_i\colon [0,a]\rightarrow Y$ is a geodesic $(i=1,2)$,
 and that $t\mapsto d(c_1(t),c_2(t))$ is constant.  Then the convex hull
 of $c_1([0,a])\cup c_2([0,a])$ is isometric to a parallelogram in
 $\R^2$. 
\end{Lemma}

\begin{proof}
 Take comparison triangles of 
 $\tri{c_1(0)}{c_1(a)}{c_2(a)}$ and
 $\tri{c_1(0)}{c_2(0)}{c_2(a)}$ in $\R^2$, and 
 identify two edges corresponing to $[c_1(0), c_2(a)]$.  
 Since $d(c_1(0),c_1(a))=d(c_2(0),c_2(a))$ and
 $d(c_1(0),c_2(0))=d(c_1(a),c_2(a))$, 
 we get a parallelogram in $\R^2$. 
 
 Take any $t \in (0,a)$ and fix it for a while. 
 Let $\overline{y}$ be a point  in $\R^2$ given by 
 $\widebar{y}=[\widebar{c_1(t)}, \widebar{c_2(t)}]\cap 
 [\overline{c_1(0)}, \overline{c_2(a)}]$, 
 and $y \in [c_1(0),c_2(a)]$ a corresponding point in $Y$. 
 Then we have $d(c_i(t),y)\leq d(\overline{c_i(t)}, \overline{y})$ for
 $i=1,2$, and hence
\begin{equation*}
 d(c_1(t),c_2(t)) \leq d(c_1(t),y)+d(c_2(t),y) \leq
 d(\widebar{c_1(t)}, \widebar{y}) + d(\widebar{c_2(t)}, \widebar{y})
 =d(\widebar{c_1(t)}, \widebar{c_2(t)}). 
\end{equation*}
 On the other hand, by our assumption, 
\begin{equation*}
  d(c_1(t),c_2(t))=d(c_1(0),c_2(0))=d(\widebar{c_1(0)},\widebar{c_2(0)})
 =d(\widebar{c_1(t)}, \widebar{c_2(t)}). 
\end{equation*} 
 Thus we see that $d(c_i(t),y)=d(\widebar{c_i(t)},\widebar{y})$ for
 $i=1,2$. 
 Therefore, by \cite[p.~181, 2.10 Exercise (1)]{bridson-haefliger}, 
 $\tri{c_1(0)}{c_1(a)}{c_2(a)}$ and
 $\tri{c_1(0)}{c_2(0)}{c_2(a)}$ are both flat, that is, the convex hull
 of these triangles in $Y$ are isometric to the convex hull of their
 comparison triangles in $\R^2$ respectively. 
 In particular, the angles $\angle_{c_1(a)}(c_1(0),c_2(a))$ and 
 $\angle_{c_2(0)}(c_1(0),c_2(a))$ are equal to the corresponding angles
 of the parallelogram in $\R^2$ above.

 Now consider triangles $\tri{c_1(0)}{c_1(a)}{c_2(0)}$ and 
 $\tri{c_2(a)}{c_1(a)}{c_2(0)}$ and follow the argument above. 
 Then we see that both triangles are flat, and hence the angles
 $\angle_{c_1(0)}(c_1(a),c_2(0))$ and 
 $\angle_{c_2(a)}(c_1(a),c_2(0))$ are equal to the corresponding angles
 of the parallelogram in $\R^2$ again. 
 Thus we see that the sum of 
 $\angle_{c_1(a)}(c_1(0),c_2(a))$, $\angle_{c_2(0)}(c_1(0),c_2(a))$, 
 $\angle_{c_1(0)}(c_1(a),c_2(0))$, and $\angle_{c_2(a)}(c_1(a),c_2(0))$
 is equal to $2\pi$. 
 Then, by 
 \cite[p.~181, 2.11 The flat quadrilateral theorem]{bridson-haefliger}, 
 we see that the convex hull of $c_1([0,a]) \cup c_2([0,a])$ is flat and
 isometric to the parallelogram in $\R^2$. 
\end{proof}

\begin{Remark}
 \label{rem:different_from_sandwich_lemma}
 Since $d(c_1(t),c_2(t))$ need not be equal to $d(c_1(t), c_2([0,a]))$,
 the lemma above is slightly different from the Sandwich lemma
 \cite[p.~182, 2.12 Exercise (2)]{bridson-haefliger}, and we do not
 obtain a rectangle as the convex hull in the statement. 
\end{Remark}

%
%
Recall that we have often denoted a geodesic starting from 
$x \in Y$ and terminating at $y \in Y \cup\partial Y$ by $\cc{x}{y}$; we
will use this notation frequently in what follows.
Also if the convex hull of two geodesic rays 
$\ray{p,\eta}\cup \ray{p,\xi}$ is isometric to a flat sector in $\R^2$, 
we denote this convex hull by $\sect[p]{\eta}{\xi}$. 
We also denote by $[\eta_1,\eta_2] \subset \partial Y$ a geodesic
segment joining $\eta_1$ and $\eta_2$ with respect to the angular metric
$d_{\angle}$ on $\partial Y$. 

\begin{Lemma}
\label{lem:parallelogram_and_sliding}
 Let $Y$ be a complete $\cat{0}$ space. 
 Suppose that, for $p_i \in Y$ $(i=1,2)$ and 
 $\eta_j \in \partial Y$ $(j=1,2)$ with 
 $0<d_{\angle}(\eta_1,\eta_2)<\pi$,
 the convex hull of 
 $\ray{p_i,\eta_1}\cup \ray{p_i,\eta_2}$ is isometric to a flat sector
 in $\R^2$. 
 Assume further that,  for each $p \in [p_1,p_2]$ and
 $\xi \in [\eta_1,\eta_2]\subset \partial Y$, a geodesic ray $[p,\xi)$
 exists and that $\bigcup_{p \in [p_1,p_2]} [p,\xi)$ is isometric to a
 half-infinite parallelogram in $\R^2$ 
 $($it may degenerate to a half line$)$, 
 which we denote by $F^{\xi}$. 

 $(1)$ Take $s_1, s_2 >0$, and denote by $R_{s_1s_2}$ the convex hull of 
 $[\cc{p_1}{\eta_1}(s_1), \cc{p_1}{\eta_2}(s_2)] \cup
  [\cc{p_2}{\eta_1}(s_1), \cc{p_2}{\eta_2}(s_2)]$. 
 Then $R_{s_1s_2}$ is isometric to a prallelogram in $\R^2$. 

 $(2)$ For each $p \in [p_1,p_2]$, the convex hull of 
 $\ray{p,\eta_1}\cup\ray{p,\eta_2}$ is isometric to 
 a flat sector in $\R^2$.  

 $(3)$ 
 Let $\cc{p}{\xi}\colon \ray{0,\infty}\rightarrow Y$ be a
 geodesic ray joining $p\in Y$ and $\xi \in [\eta_1,\eta_2]$. 
 Consider a map 
 $\trans{s}{\xi} \colon q \mapsto \cc{q}{\xi}(s)$, $s>0$, 
 defined on 
 $\bigcup_{p \in [p_1,p_2]} \sect[p]{\eta_1}{\eta_2}$. 
 Then $\trans{s}{\xi}$ maps $R_{s_1s_2}$ isometrically onto its image. 

 $(4)$ Take $s_0>0$. 
 For any $p \in [p_1,p_2]$ and 
 $q_j \in F^{\eta_j} \cap R_{s_0s_0}$, $j=1,2$ 
 $\tri{p}{q_1}{q_2}$ is a flat triangle. 
\end{Lemma}

\begin{proof}
 $(1)$ Let $c\colon [0,a]\rightarrow Y$ be a geodesic joining $p_1$ and
 $p_2$. 
 First note that, since we have
 $\angle_{p_1}(\eta_1,\eta_2)=\angle_{p_2}(\eta_1,\eta_2)=d_{\angle}(\eta_1,\eta_2)$
 (see Remark~\ref{rem:sector}),
 $\sect[p_1]{\eta_1}{\eta_2}$ and $\sect[p_2]{\eta_1}{\eta_2}$ are
 isometric to each other. 
 Therefore we have
 $d(\cc{p_1}{\eta_1}(s_1),\cc{p_1}{\eta_2}(s_2))=
  d(\cc{p_2}{\eta_1}(s_1),\cc{p_2}{\eta_2}(s_2))$, 
 and we denote this value by $d$. 

\noindent
\begin{minipage}[c]{6.5cm}
 Consider a geodesic 
 $\tilde c_i \colon [0,d]\rightarrow \sect[p_i]{\eta_1}{\eta_2}$ 
 with $\tilde c_i(0)=c_{p_i}^{\eta_1}(s_1)$ and 
 $\tilde c_i(d)=c_{p_i}^{\eta_2}(s_2)$,
 $i=1,2$. Then, for any $t \in [0,d]$, 
 $\angle_{p_1}(\eta_1, \tilde c_1(t))= \angle_{p_2}(\eta_1, \tilde
 c_2(t))$,  
 and we see that there exists $\xi_t\in [\eta_1,\eta_2]$ such that 
 $\tilde c_i(t) \in [p_i,\xi_t)$ for $i=1,2$. 
 Since we are assuming that there exists $F^{\xi}$ for each $\xi$, 
 $[\tilde c_1(t), \tilde c_2(t)]$ is parallel to
 $[p_1,p_2]$ in $F^{\xi_t}$ 
 and given by
 $\tau \mapsto \cc{c(\tau)}{\xi_t}(s)$ for some $s$. 
 In particular, 
 $d(\tilde c_1(t),\tilde c_2(t))=d(p_1,p_2)$
 for any $t \in [0,d]$. 
\end{minipage} 
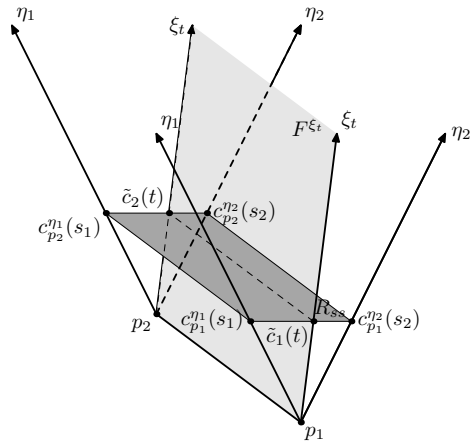
\begin{minipage}[c]{6cm}
\hspace{1cm}
\vspace{0.5cm} \hspace{4cm}
\includegraphics[scale=.8]{fig.4}
\end{minipage}

\vspace{0.5cm}
\noindent
 Therefore, by Lemma~\ref{lem:parallelogram}, we see that 
 the convex hull of 
\begin{equation*}
  [\tilde c_1(0),\tilde c_1(d)] \cup  [\tilde c_2(0),\tilde c_2(d)] 
 =[\cc{p_1}{\eta_1}(s_1),\cc{p_1}{\eta_2}(s_2)] \cup
  [\cc{p_2}{\eta_1}(s_1),\cc{p_2}{\eta_2}(s_2)]
\end{equation*} 
 is isometric to a flat parallelogram, which we denote by $R_{s_1s_2}$.

 $(2)$  By $(1)$, we see that 
 $d(\cc{p}{\eta_1}(s),\cc{p}{\eta_2}(s))=
 d(\cc{p_1}{\eta_1}(s),\cc{p_1}{\eta_2}(s))$
 for any $p \in [p_1,p_2]$ and $s \in [0,\infty)$. 
 In particular, for any $p \in [p_1,p_2]$, 
 $\angle_p(\eta_1, \eta_2)=\angle_{p_1}(\eta_1,\eta_2)=
 d_{\angle}(\eta_1,\eta_2)$. 
 Therefore, by \cite[p.~283, 9.9 Corollary]{bridson-haefliger}, 
 $\ray{p,\eta_1} \cup \ray{p,\eta_2}$ bounds a flat sector. 

 $(3)$ Let $q_i^j=\cc{p_i}{\eta_j}(s_j)$ for $i,j=1,2$. 
 We know that the convex hull of 
 $[q_1^1, q_1^2] \cup [q_2^1,q_2^2]$ is a flat parallelogram by (1). 
 Since $\sect[p_i]{\eta_1}{\eta_2}$ is a flat sector, 
 $\cc{q_i^1}{\xi}$ and $\cc{q_i^2}{\xi}$ are parallel to each other, and
 hence $t\mapsto d(\cc{q_i^1}{\xi}(t),\cc{q_i^2}{\xi}(t))$ is constant. 

\noindent
\begin{minipage}[c]{6cm}
  Fix $t\geq 0$ for a while, and let $\tilde c_i$ be a geodesic joining
 $\cc{q_i^1}{\xi}(t)$ and $\cc{q_i^2}{\xi}(t)$. 
 Since $\sect[p_i]{\eta_1}{\eta_2}$, $i=1,2$, are isometric to
 each other, for each $\tau$ and $i=1,2$, 
 $\tilde c_i (\tau) \in F^{\xi_{\tau}}$ for the same 
 $\xi_{\tau} \in [\eta_1, \eta_2]$. 
 Then, by the same argument as we have seen in the proof of $(1)$
 above,  we see that 
 $d(\tilde c_1(\tau),\tilde c_2(\tau))=d(p_1,p_2)$ for any $\tau$, 
 and hence the convex hull of 
 $[\cc{q_1^1}{\xi}(t), \cc{q_1^2}{\xi}(t)] \cup
  [\cc{q_2^1}{\xi}(t), \cc{q_2^2}{\xi}(t)]$
 is isometric to a flat parallelogram. 
\end{minipage}
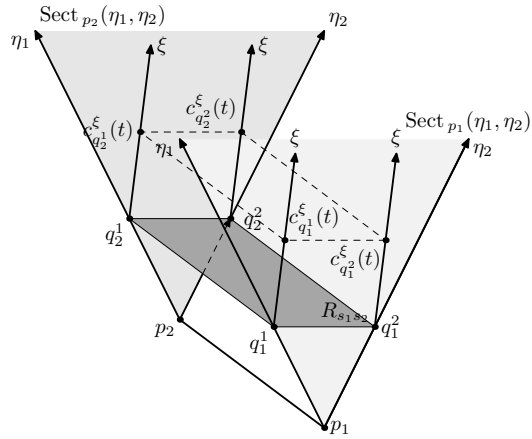
\begin{minipage}[c]{6cm}
 \hspace{5cm}
 \includegraphics[scale=.8]{fig.5}
\end{minipage}

 Now we vary $t$. 
 Note that $d(\cc{q_1^1}{\xi}(t), \cc{q_2^2}{\xi}(t))$ and 
 $d(\cc{q_1^2}{\xi}(t), \cc{q_2^1}{\xi}(t))$ are bounded convex
 function defined on $[0,\infty)$, and hence both are nonincreasing. 
 These functions give the length of two
 diagonal lines of parallelograms with fixed side length.  Therefore, if
 one decreases, the other must increase.  Hence we see that both
 functions are constant. 
 Therefore, for all $t$,  the convex hull of 
 $[\cc{q_1^1}{\xi}(t), \cc{q_1^2}{\xi}(t)] \cup
  [\cc{q_2^1}{\xi}(t), \cc{q_2^2}{\xi}(t)]$
 are isometric to the same flat parallelogram. 
 Since $\trans{s}{\xi}$ maps 
 $[\cc{q_i^1}{\xi}(0),\cc{q_i^2}{\xi}(0)]$ to
 $[\cc{q_i^1}{\xi}(s), \cc{q_i^2}{\xi}(s)]$, this shows $(3)$.

 $(4)$  Replacing $p_1$ and $p_2$ if neccesary, we may assume 
 $p=c(t_0) \in [p_1,p_2)$. 
 Let $q_i\in F^{\eta_i}$, $(i=1,2)$ and consider 
 $\tri{p}{q_1}{q_2}$. 
 Denote by $c_i$ $(i=1,2)$ a geodesic joining $p$ and $q_i$. 
 Since $F^{\eta_i}$ is flat, $c_i$ is lying in $F^{\eta_i}$. 
 Let us identify $F^{\eta_i}$ with a half-infinite parallelogram  
 $\widebar{F}_i \subset \R^2$ of the form 
\begin{equation*}
  \widebar{F}_i=
 \{a \widebar p +b\widebar q_i\mid a\in [0,d(p_1,p_2)],\ 
 b\geq0\}, 
\end{equation*}
 where we identify $\cc{p_1}{p_2}(1)$ with $\bar p$ and 
 $\cc{p_1}{\eta_i}(1)$ with $\bar q_i$.  
 Note that every geodesic segment in $\widebar{F}_i$ starting from 
 $a_0 \bar p$ can be expressed
 as $t\mapsto (a_0+a_it)\widebar p + (b_i t) \widebar q_i$. 
 Therefore we see that there exists $a_i, b_i \in \R$ such that $c_i$
 can be expressed as 
 $c_i(t)=\cc{c(t_0+ a_i t)}{\eta_i}(b_i t)$. 
 Note that $c_i(s/b_i)$  lies in $R_{ss}$. 
 By considering an isometric translation $\trans{s}{\eta_1}$ in $(3)$, 
 we see that, for any $s_1>0$, $s>0$, and $t\in [0,a]$,
 \begin{equation*}
   \angle_{\cc{c(t)}{\eta_1}(s_1)}(\cc{p_2}{\eta_1}(s_1),
   \cc{c(t)}{\eta_2}(s_1))
 =  \angle_{\cc{c(t)}{\eta_1}(s_1+s)}(\cc{p_2}{\eta_1}(s_1+s), 
 \cc{c(t)}{\eta_2}(s_1+s))
\end{equation*}
 and hence 
 $\alpha(s,t)= \angle_{\cc{c(t)}{\eta_1}(s)}(\cc{p_2}{\eta_1}(s), 
  \cc{c(t)}{\eta_2}(s))$
 is independent of $s$. 
 Furthermore, since 
 $\left\{[\cc{c(t)}{\eta_1}(s),\cc{c(t)}{\eta_2}(s)]\mid t \in
 [0,a]\right\}$  
 forms a family of parallel geodesic segments in $R_{ss}$,
 $\alpha(s,t)$ does not depending on $t$ either. 
 Note that $d(c_1(s/b_1),c_2(s/b_2))$ is equal to the length of
 a diagonal line of a parallelogram $R_{ss}'$ in $R_{ss}$ spanned over 
 an interval  
 $[c(t_0+(a_1s/b_1)), c(t_0+(a_2s/b_2))] \subset [p_1,p_2]$,  
 which is bounded by 
 $[\cc{c(\bar t_1)}{\eta_1}(s),\cc{c(\bar t_1)}{\eta_2}(s)]$ and 
 $[\cc{c(\bar t_2)}{\eta_1}(s),\cc{c(\bar t_2)}{\eta_2}(s)]$, where 
 $\bar t_i=t_0+(a_is/b_i)$, $i=1,2$. 

 
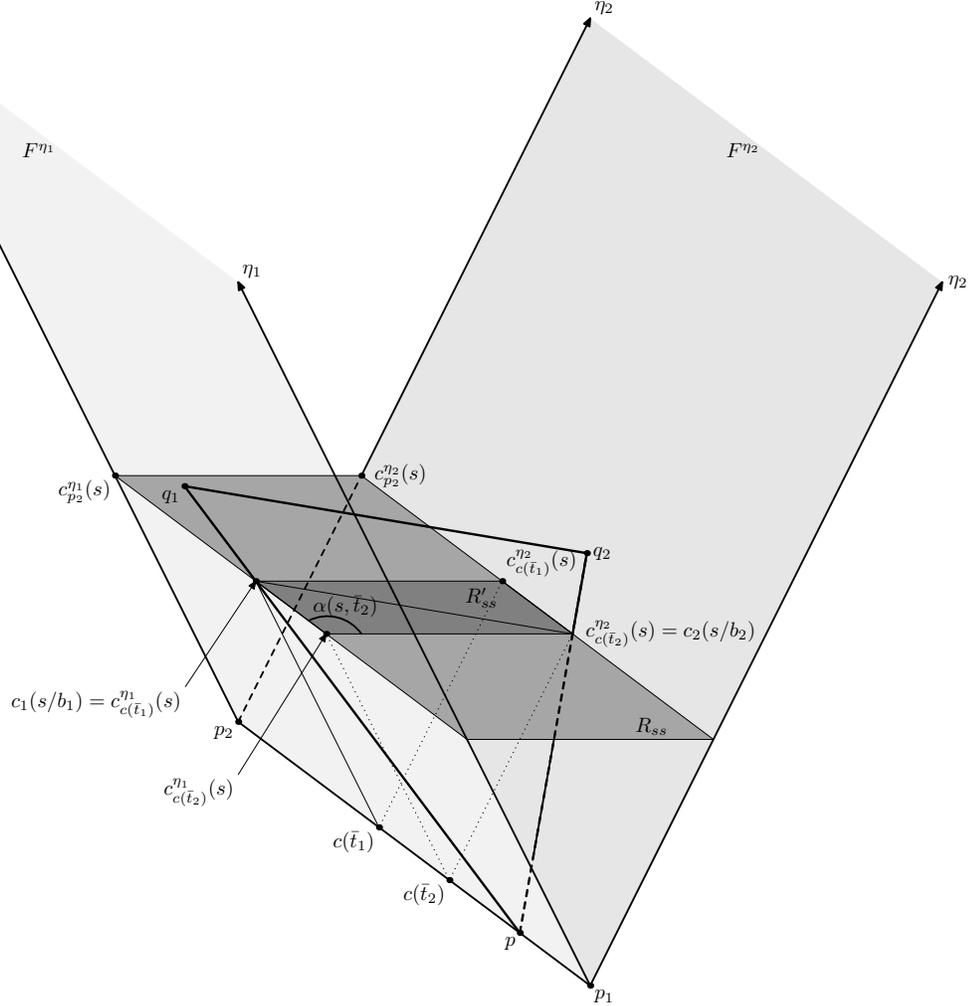
\hspace{9.2cm}
 \includegraphics[scale=.78]{fig.6}

 \vspace{0.3cm}
 \noindent
 Then we see that the length of two pairs of sides of $R_{ss}'$ are
\begin{equation*}
  d(\cc{c(\bar t_1)}{\eta_i}{s}, \cc{c(\bar t_2)}{\eta_i}(s))=
  d(c(\bar t_1),c(\bar t_2)) = |\bar t_1- \bar t_2| =
 \left|\frac{a_1}{b_1}-\frac{a_2}{b_2} \right|s, \ (i=1,2),  
\end{equation*} 
 and
 \begin{equation*}
  d(\cc{c(\bar t_i)}{\eta_1}(s), \cc{c(\bar t_i)}{\eta_2}(s))
  = 2s \sin \frac{d_{\angle}(\eta_1,\eta_2)}{2}, \ (i=1,2), 
 \end{equation*}
 since 
 $\ch{\ray{c(\bar t_i),\eta_1} \cup \ray{c(\bar t_i),\eta_2}}$
 is a flat sector by $(2)$. 
 Recalling that $\alpha(s,t)$ is constant, we see that
 the length of diagonal line of $R_{ss}'$ is also linear in $s$: 
\begin{equation*}
 d(\cc{c(\bar t_1)}{\eta_1}(s),\cc{c(\bar t_2)}{\eta_2}(s))
 =   d(c_1(s/b_1),c_2(s/b_2)) = Cs 
\end{equation*} 
 for some constant $C$. 
 Since $d(p,c_i(s/b_i))=s/b_i$ is also linear in $s$, 
 by considering a comparison triangle and taking 
 \cite[p.~181, 2.10 Exercises (1)]{bridson-haefliger} into account, 
 we conclude that $\tri{p}{c_1(s/b_1)}{c_2(s/b_2)}$ is flat.  
 This completes the proof. 
\end{proof}

\noindent
\begin{minipage}[c]{9cm}
 \begin{Remark}
 \label{remark:degenerate_to_half_line}
  Let $p_1=(0,0), p_2=(1,0) \in \R^2$ and denote 
  $(\infty,0) \in \partial \R^2$ by $\xi$. 
  Take $\eta_i \in \partial \R^2$ so that 
  $\angle(\xi, \eta_i)=\theta < \pi/2$. 
  Then 
  $\ch{\ray{p_2,\eta_1},\ray{p_2,\eta_2}}\subset
  \ch{\ray{p_1,\eta_1},\ray{p_1,\eta_2}} \subset \R^2$, 
  and $F^{\xi}$ degenerates to a half line $\ray{p_1,\xi}$. 
  Even in this case, the lemma above is valid. 
\end{Remark}
\end{minipage}
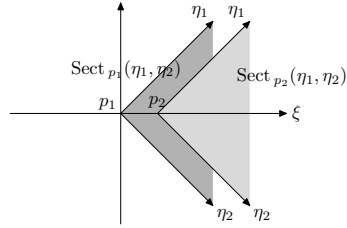
\begin{minipage}[b]{6cm}
 \hspace{2cm}
 \includegraphics[scale=.7]{fig.7}
\end{minipage} 

\begin{Remark}
\label{remark:when_angle_is_pi}
 Suppose that $d_{\angle}(\eta_1,\eta_2)=\pi$ in the setting of the
 lemma above.  
 Then we have two half-infinite parallelograms $F^{\eta_i}$, $i=1,2$,
 and $F^{\eta_1}\cup F^{\eta_2}$ is bounded by a two geodesic lines
 joining $\eta_1$ and $\eta_2$ and passing through $p_1$ and $p_2$
 respectively. 
 Then the set $F^{\eta_1}\cup F^{\eta_2}$ consists of parallel geodesic
 lines joining $\eta_1$ and $\eta_2$.
 By \cite[p.~182, The Flat Strip Theorem]{bridson-haefliger}, we see
 that $F^{\eta_1}\cup F^{\eta_2}$ forms a flat strip in $Y$.  
 Therefore $(1)$, $(3)$, and  $(4)$ in the lemma above clearly hold. 
 For $(2)$, $\sect[p]{\eta_1,\eta_2}$ degenerates to a geodesic
 line for each $p \in [p_1,p_2]$, and the assertion is also valid in
 this sense.  
\end{Remark}

\subsection{Construction of a cone $\cone{m}$.}
\label{sec:construction-conem}

Now we return to the setting of Theorem~\ref{thm:compactness}. 
Consider $\{f(\gamma_j)\} \subset f(\Gamma)$ with 
$d(f(e),f(\gamma_j))\to \infty$, and let $\xxi{1}=\ulim f(\gamma_j)$. 
We denote $f(\gamma_j)$ by $x_j$, and set $x_0=f(e)$. 
We also denote $Y^{(1)}=\ulim (Y,x_0)$, and 
$Y^{(i)}=\ulim (Y^{(i-1)},x_0)$ inductively. 
We often use the expression $\uulim[i]_j x_j$ to emphasize that we are
taking $i$th ultralimit; it should be taken on $Y^{(i-1)}$. 
As we will see, if $\{x_j\}$ does not converge to $\xxi{1}$ in the cone
topology, then $\xxi{i}=\uulim[i]x_j \not=\xxi{1}$ for $i>1$, and 
$\{\xxi{1}, \dots \xxi{m}\}$
spans a Euclidean cone over a regular Euclidean $(m-1)$-simplex 
(Proposition~\ref{prop:simplex}), from which we conclude that $\{x_j\}$
actually converges to $\xxi{1}$ and derive
Theorem~\ref{thm:compactness}. 

%
%
\begin{Setting}
 \label{setting}
 A $\cat{0}$ space $Y$ is either of finite telescopic dimension
 or locally finite-dimensional. 
 And a countable group $\Gamma$ with a random walk $\mu$ acts
 isometrically on $Y$ via a homomorphism
 $\rho\colon\Gamma \rightarrow \isom{Y}$.  
 Futhermore, from what we have seen so far, we may assume the following: 
 \begin{itemize}
 \item  A symmetric probability measure $\mu$ on $\Gamma$ satisfies
	$\supp \mu \supset \Gamma \setminus \{e\}$. 
 \item  There exists a $\rho$-equivariant
	$\mu$-harmonic map $f \colon \Gamma \rightarrow Y$. 
 \item  There is a sequence $\{\gamma_j\}\subset \Gamma$ such that
	$d(f(e),f(\gamma_j))\to \infty$, 
	and, as in the proof of
	Lemma~\ref{lem:ultralimit_of_buseman_fct}, we have 
	$\ulim f(\gamma_j)= \xxi{1} \in \partial Y^{(1)}$. 
	We denote $f(\gamma_j)$ by $x_j$, and $\uulim[i]x_j=\xxi{i}$. 
 \item  A function 
        $\varphi_{\xxi{1}}\colon \gamma \mapsto
	b_{\xxi{1}}(f(\gamma),f(e))$ 
	defined on $\Gamma$ is a $\mu$-harmonic function. 
	The same is true for 
        $\varphi_{\xxi{i}}\colon \gamma \mapsto
	b_{\xxi{i}}(f(\gamma),f(e))$ 
	for any $i$ by Lemma~\ref{lem:ultralimit_of_buseman_fct} and
        the proof of Proposition~\ref{prop:global-mu-harmonicity}.
 \item  For any $x=f(\gamma), x'=f(\gamma') \in f(\Gamma)$ and $i$, 
	by applying Lemma~\ref{lem:core} $(2)$ to
	$\rho_{\gamma}$-equivariant $\mu$-harmonic map $f_{\gamma}$ as
	in \S~\ref{sec:when-y-finite-dim}, 
        we see that the convex hull $F_{xx'}^{(i)}$ of 
	$\ray{x,\xxi{i}}\cup [x,x']\cup \ray{x',\xxi{i}}$
	is isometric to a half-infinite parallelogram in $\R^2$.
 \end{itemize}
 Also, in order to prove Theorem~\ref{thm:compactness}, we assume that
 $\zure=\ulim_j \angle_{f(e)}(x_j,\xxi{1})>0$, which will lead us to a
 contradiction in \S~\ref{sec:the_proof}.  
\end{Setting}

%
%
\begin{Lemma}
\label{lem:zure<pi} 
 Under our setting, we have 
 $\ulim_j \angle_{x_0} (x_j,\xxi{i})< \pi$ for any $i$. 
\end{Lemma} 

\begin{proof}
 What we have to show is that there exists $\delta > 0$ such that
\begin{equation*}
  \lambda\left(\{j \mid
 \angle_{x_0}(x_j,\xxi{i})\leq \pi-\delta \right\})=1. 
\end{equation*}
 Suppose not.  Then, since $\lambda$ is an ultrafilter, we get
 $\lambda (A_{\delta})=1$ for any $\delta>0$, where
 \begin{equation*}
 A_{\delta} = 
 \{i \mid \angle_{x_0}(x_j,\xxi{i})>\pi -\delta \}. 
 \end{equation*}
 Take $\delta$ so that $\delta < \pi/2$. 
 Since $f$ is $\mu$-harmonic and 
 $\int_{\Gamma}b_{\xxi{i}}(f(\gamma),x_0) d\mu(\gamma)=0$,  we have 
 \begin{equation*}
 \int_{\Gamma} \inner{\tcprj_{x_0}(f(\gamma))}{V_c} 
 d\mu (\gamma)=0
 \end{equation*}
 as in the proof of Lemma~\ref{lem:core} $(2)$, where 
 $V_c \in S_{x_0}Y^{(i)}$ corresponds to a geodesic ray with
 $c(0)=x_0$ and $c(\infty)=\xxi{i}$. 
 Since $\inner{\tcprj_{x_0} (x_j)}{V_c} < 0$ for  any 
 $j \in A_{\delta}$ by our assumption, 
 there must be $\gamma' \in \Gamma$ satisfying
 $\inner{\tcprj_{x_0} (f(\gamma'))}{V_c} > 0$.  
 Let  $d=d(x_0,f(\gamma'))$, and $c'$ a geodesic joining $x_0$ and
 $f(\gamma')$. 
 Take a geodesic $c_j$ joining $x_0$ and $x_j$.  
 Since we are assuming $d(x_0,x_j) \to \infty$, 
 we may assume that $d(x_0,x_j)> d$
 holds and hence $c_j(d)$ is defined for any $j \in A_{\delta}$. 
 Since $\ulim c_j(t) = c(t)$, we get
 \begin{equation*}
 \ulim d(c'(t),c_j(t)) = d(c'(t),c(t)), 
 \end{equation*}
 which means that, by setting 
 \begin{equation*}
 B_{\delta'}= \{j \in \N \mid  
 |d(c'(t),c_j(t))-d(c'(t),c(t))| < \delta'\}, 
 \end{equation*}
 we have $\lambda (B_{\delta'})=1$ for any $\delta'>0$.  
 Since $\lambda$ is an ultrafilter, we get 
 $\lambda (A_{\delta}\cap B_{\delta'})=1$. 
 By the triangle inequality on $(S_{x_0}Y^{(i)}, \angle_{x_0})$, 
 we get
 \begin{equation*}
 \angle_{x_0}(c_j(t), c'(t)) \geq 
 \angle_{x_0}(c_j(t), c(t)) - \angle_{x_0}(c'(t),c(t))
 \geq \pi -\delta - \angle_{x_0}(c'(t),c(t)), 
 \end{equation*}
 for any $j \in A_{\delta}$. 
 While $\inner{\tcprj_{x_0}(f(\gamma'))}{V_c} >0$ implies
 $\angle_{x_0}(c'(t), c(t))< \pi/2$, we see that 
 for a sufficiently small $\delta>0$, 
 \begin{equation*}
 \angle_{x_0}(c_j(t), c'(t)) \geq \pi/2 + \alpha
 \end{equation*}
 holds for some $\alpha>0$ and any $j \in A_{\delta}$.  
 Since the comparison angle is nonincreasing as $t\to 0$
 (\cite[p.~184, 3.1 Proposition]{bridson-haefliger}), 
 this implies that 
 $d(c_j(t),c'(t))^2 \geq (2+\alpha')t^2$ for some $\alpha'>0$ and 
 any $j \in A_{\delta}$ and $t>0$.  
 Since $\lambda(A_{\delta})=1$, we see that 
 $\ulim d(c_j(t),c'(t))^2 = d(c(t),c'(t))^2> 2t^2$. 
 On the other hand, for sufficiently small $t>0$, we have
 \begin{equation*}
 d(c'(t),c(t))^2 \leq 2t^2
 \end{equation*}
 by $\inner{\tcprj_{x_0}(f(\gamma'))}{V_c} >0$. 
 A contradiction. 
\end{proof}

%
%
\begin{Lemma}
 \label{lem:if_angle_is_positive}
 Under our setting, for any $i,j \in \N$ and  $x \in f(\Gamma)$, 
 the convex hull of $\ray{x,\xxi{i}}\cup \ray{x,\xxi{j}}$ is a flat
 sector $\sect[x]{\xxi{i}}{\xxi{j}}$ with angle $\zure$. 
\end{Lemma}

\begin{proof}
 By our assumption, for each $x$ and $j$, 
 $\ray{x,\xxi{1}} \cup [x,x_j]\cup \ray{x_j,\xxi{1}}$
 bounds a half inifinite parallelogram, 
 which we denote by $F_{xx_j}^{(1)}$.
 Let $\cc{x}{x_j} \colon [0,d(x,x_j)]\rightarrow Y$ be a geodesic joining
 $x$ and $x_j$, and 
 $\cc{x}{(1)}\colon \ray{0,\infty}\rightarrow Y^{(1)}$ a geodesic ray
 joining $x$ and $\xxi{1}$, which is obtained as 
 $\cc{x}{(1)}(t)=\uulim[1]_j \cc{x}{x_j}(t)$. 
 Set $\theta_x = \ulim \angle_{x}(x_j,\xxi{1})$ and fix $x$ for a while. 
 Then, since $F_{xx_j}^{(1)}$ is flat, 
 $\angle_{x}(x_j, \xxi{1})=
 2\arcsin \left(d(\cc{x}{(1)}(t),\cc{x}{x_j}(t))/2t\right)$
 for any $t>0$, 
 and thus $\ulim_j \angle_{x} (x_j,\xxi{1})=\theta_x$ 
 implies that 
 \begin{equation}
 \label{eq:euclidean_angle}
  \ulim_j d(\cc{x}{(1)}(t),\cc{x}{x_j}(t)) = 2t\sin \frac{\theta_x}{2}
 \end{equation}
 for any $t>0$. 

 Now consider $Y^{(2)}=\ulim (Y^{(1)},p_0)$. 
 Then we have a geodesic ray $\cc{x}{(2)} = \uulim[2]_j \cc{x}{x_j}$ as
 before.  
 Then $\xxi{2}=\cc{x}{(2)}(\infty)$. 
 By recalling (\ref{eq:euclidean_angle}) with $x=f(e)$ and our
 assumption $\zure>0$, $\cc{f(e)}{(2)}$ cannot coincide with
 $\cc{f(e)}{(1)}$, and hence $\xxi{2}\not= \xxi{1}$. 
 Note that, for each $t>0$, we have 
 $d(\cc{x}{(1)}(t),\cc{x}{(2)}(t))=2t \sin (\theta_x/2)$, hence
 $\angle_{x}(\cc{x}{(1)}, \cc{x}{(2)})=\theta_x$ by
 \cite[p.~184, 3.1 Proposition]{bridson-haefliger}. 
 Therefore, each triangle
 $\tri{x}{\cc{x}{(1)}(t)}{\cc{x}{(2)}(t)}$ is isometric to a
 Euclidean triangle by \cite[p.~180 2.9 Proposition]{bridson-haefliger}, 
 and hence $\ray{x,\xxi{1}} \cup \ray{x,\xxi{2}}$ bounds a flat
 sector with angle $\theta_x$ in $Y^{(2)}$. 
 Thus $\theta_x=d_{\angle}(\xxi{1},\xxi{2})$, and
 we see that $\theta_x=\zure$ for all $x \in f(\Gamma)$. 
 Since we know $\zure<\pi$ by Lemma~\ref{lem:zure<pi}, 
 these flat sectors do not degenerate to a geodesic line. 
 By our assumption, 
 $\ray{x,\xxi{2}}\cup [x,x_j]\cup\ray{x_j,\xxi{2}}$ also bounds 
 a half inifite parallelogram. 
 Note that, for fixed $x,x'\in f(\Gamma)$, when we take an ultralimit for
 the first time and obtain $Y^{(1)}$, 
 we have 
 $d(\cc{x}{x'}(t),\cc{x}{(1)}(t))=\ulim_j d(\cc{x}{x'}(t),\cc{x}{x_j}(t))$, 
 while when we take it for the second time and obtain
 $Y^{(2)}$, we must have 
 $d(\cc{x}{x'}(t),\cc{x}{(2)}(t))=\ulim_j d(\cc{x}{x'}(t), \cc{x}{x_j}(t))$. 
 In particular, 
 $\angle_{x}(x_j, \xxi{2})=\angle_{x}(x_j, \xxi{1})$
 for all $j$ and $x\in f(\Gamma)$, and we have 
 $\ulim_j \angle_{x}(x_{j}, \xxi{2})=
 \ulim_j \angle_{x}(x_j, \xxi{1})= \zure$. 
 Thus, in $Y^{(3)}$, we see that
 $\ray{x,\xxi{i}} \cup \ray{x,\xxi{j}}$ bounds 
 a flat sector with angle $\zure$ for each $i,j=1,2,3$ with $i\not=j$. 
 Repeat this $n$ times, and denote by $\xxi{i}=\uulim[i]x_j$, 
 an ultralimit of $\{x_j\}$  taken in $Y^{(i-1)}$.
 Then we get a flat sector $\sect[x]{\xxi{i}}{\xi^{(l)}}$
 bounded by $\ray{x,\xxi{i}} \cup \ray{x,\xi^{(l)}}$ with angle
 $\zure$ as $\uulim[l]F_{xx_j}^{(i)}$ for any $1 \leq i<l \leq n$. 
\end{proof}

%
%
 Take any $x =f(\gamma) \in f(\Gamma)$, 
 $\eta \in [\xxi{1},\xxi{2}] \subset \partial Y^{(2)}$, and 
 $t'$ so that 
 $\cc{x}{\eta}(t') \in [\cc{x}{(1)}(t),\cc{x}{(2)}(t)]$. 
 Since $\tri{x}{\cc{x}{(1)}(t)}{\cc{x}{(2)}(t)}$ lies in a flat
 sector $\sect[x]{\xxi{1}}{\xxi{2}}$ by
 Lemma~\ref{lem:if_angle_is_positive},  we have
\begin{equation*}
 d(x,\cc{x}{\eta}(t')) = (1-s)d(x,\cc{x_k}{(1)}(t))
 + s d(x,\cc{x}{(2)}(t)), 
\end{equation*}
 where $s\in [0,1]$ is determined by 
 $d(\cc{x}{(1)}(t),\cc{x}{\eta}(t'))=s
 d(\cc{x}{(1)}(t),\cc{x}{(2)}(t))$. 
 Now take any $p \in Y$. 
 Then by the conexity of the distance function, we see that 
\begin{equation*}
 d(p,\cc{x}{\eta}(t')) \leq
 (1-s)d(p,\cc{x}{(1)}(t)) + s d(p,\cc{x}{(2)}(t)). 
\end{equation*}
 Thus, by letting $t',t\to \infty$, we get
\begin{equation*}
\begin{split}
  & b_{\eta}(p, x)  = \lim_{t'\to \infty}
  d(p,\cc{x}{\eta}(t'))-d(x,\cc{x}{\eta}(t')) \\
 \leq & \lim_{t\to \infty} \left(\left(
 (1-s)d(p,\cc{x}{(1)}(t)) + sd(p,\cc{x}{(2)}(t))\right) \right.\\
 & \left. \phantom{\lim_{t\to \infty} \ \ } -
  \left((1-s)d(x,\cc{x}{(1)}(t))+sd(x,\cc{x}{(2)}(t))\right) \right)\\
 =& \lim_{t\to \infty}\left(
 (1-s) \left(d(p,\cc{x}{(1)}(t)) -t \right)
  + s\left(d(p,\cc{x}{(2)}(t))-t \right)
 \right)\\
 =& (1-s)b_{\xxi{1}}(p,x)+ sb_{\xxi{2}}(p,x). 
\end{split}
\end{equation*}
 Since $\xxi{i}=\uulim[i]_j x_j$, for any $p \in Y$, we have
 $b_{\xxi{1}}(p,x) = b_{\xxi{2}}(p,x)$ by
 Lemma~\ref{lem:ultralimit_of_buseman_fct}, and hence
\begin{equation*}
  b_{\eta}(p, x) \leq b_{\xxi{i}}(p,x)
\end{equation*}
 for $i=1,2$. 

 Since $p \mapsto b_{\eta}(p,x)$ is a convex function and 
 $f$ is a $\mu$-harmonic map, we see that a function
 $\varphi_{\eta}\colon \gamma' \mapsto b_{\eta}(f(\gamma'),f(\gamma))$
 is $\mu$-subharmonic; for any $\gamma' \in \Gamma$, we have
\begin{equation*}
 \Delta \varphi_{\eta}(\gamma') 
 = b_{\eta}(f(\gamma'),f(\gamma)) - 
  \int_{\Gamma} b_{\eta}(f(\gamma' \gamma''),f(\gamma)) d\mu(\gamma'')
  \leq 0. 
\end{equation*}
 On the other hand, since  we already know that 
 $\gamma' \mapsto b_{\xxi{i}}(f(\gamma'),f(\gamma))$ is $\mu$-harmonic,
 noting $b_{\xi}(x,x)=0$, we see that
\begin{equation}
 \label{eq:buseman_eta}
 \begin{split}
  0 & \geq b_{\eta}(f(\gamma),f(\gamma)) 
     - \int_{\Gamma} b_{\eta}(f(\gamma \gamma'),f(\gamma)) d\mu(\gamma') \\
    & \geq  b_{\xxi{i}}(f(\gamma),f(\gamma)) 
     - \int_{\Gamma} b_{\xxi{i}}(f(\gamma \gamma'),f(\gamma)) d\mu(\gamma') 
     = 0. 
 \end{split}
\end{equation}
 Noting that the computation above is valid for any 
 $x=f(\gamma) \in f(\Gamma)$, and that 
\begin{equation*}
\begin{split}
  &  b_{\eta}(f(\gamma'),f(\gamma)) 
  -\int_{\Gamma} b_{\eta}(f(\gamma' \gamma''),f(\gamma)) d\mu(\gamma'') \\
  =&   b_{\eta}(f(\gamma'),f(\gamma'))+ b_{\eta}(f(\gamma'),f(\gamma)) \\
   & \ -\int_{\Gamma} b_{\eta}(f(\gamma' \gamma''),f(\gamma')) + 
    b_{\eta}(f(\gamma'),f(\gamma)) \ d\mu(\gamma'') \\
 =&   b_{\eta}(f(\gamma'),f(\gamma')) 
  -\int_{\Gamma} b_{\eta}(f(\gamma' \gamma''),f(\gamma')) d\mu(\gamma'')
\end{split}
\end{equation*}
 holds, we see a that $\gamma' \mapsto b_{\eta}(f(\gamma'),f(\gamma))$
 is also $\mu$-harmonic. 
 Note that, for any $\gamma, \gamma'\in \Gamma$, we also obtain 
 $b_{\eta}(f(\gamma'),f(\gamma))=b_{\xxi{i}}(f(\gamma'),f(\gamma))$
 by the equality in (\ref{eq:buseman_eta}).
 In particular, by 
 Lemma~\ref{lem:core} $(2)$, 
 for any $x=f(\gamma)$ and $x'=f(\gamma')$,  the convex hull of 
 $\ray{x,\eta}\cup [x,x']\cup \ray{x',\eta}$ is a half-infinite
 parallelogram, which can be expressed as 
 $F_{xx'}^{\eta}=\bigcup_{p \in [x,x']}c_p^{\eta}([0,\infty))$. 

 Take $p_j \in [x,x_j]$ so that $d(x,p_j)=d$ for fixed $d$. 
 Then $\uulim[3]_j c_{p_j}^{\eta}$ turns out to be a geodesic ray
 $c_{\infty}$ with $c_{\infty}(0)=\uulim[3]_j p_j = q\in \ray{x, \xxi{3}}$. 
 Since $d(c_{\infty}(t), \cc{x}{\eta}(t))=d$ holds for any $t \geq 0$,  
 $c_{\infty}(\infty)=\eta$ and $c_{\infty}$ should be denoted by
 $c_q^{\eta}$. 
 Set 
 $F_{x(3)}^{\eta}=\bigcup_{q \in \ray{x,\xxi{3}}}
 c_q^{\eta}([0,\infty))$. 
 Then $F_{x(3)}^{\eta}$ can be viewed as an ultralimit of
 $F_{xx_j}^{\eta}$.  
 Since each $F_{xx_j}^{\eta}$ is flat, 
 $\tri{x}{c_{x}^{x_j}(t)}{\cc{x}{\eta}(t)} \subset
 F_{xx_j}^{\eta}$ forms a flat triangle, and hence 
 $\uulim[3]_j \tri{x}{c_{x}^{x_j}(t)}{\cc{x}{\eta}(t)}$ is also flat. 
 Letting $t\to \infty$, we see that 
 $F_{x(3)}^{\eta}$ is also flat, and hence we may write
 $F_{x(3)}^{\eta}=\sect[x]{\eta}{\xxi{3}}$. 
 By the same argument as we have just seen, we conclude that, for any 
 $\eta \in [\xxi{1},\xxi{2}]$ and 
 $\eta' \in [\eta,\xxi{3}]\subset \partial Y^{(3)}$, the convex hull of 
 $\ray{x,\eta'}\cup [x,x'] \cup \ray{x',\eta'}$ is a half-infinite
 parallelogram. And we get 
 $\sect[x]{\eta'}{\xxi{4}}=F_{x(4)}^{\eta'} \subset Y^{(4)}$ by taking
 an ultralimit $\uulim[4]_jF_{xx_j}^{\eta'}$. 
 Repeating this leads us to the following definition:

\noindent
\begin{minipage}[c]{7cm}
  Let
 \begin{equation*}
\begin{split}
  & C_{x}(i,j)=\sect[x]{\xxi{i}}{\xxi{j}}, \\
 & C_{x}(i, \dots, j, k)
 = \bigcup_{\eta \in \partial C(i, \dots, j)} 
    \sect[x]{\eta}{\xxi{k}},
\end{split} 
\end{equation*}
 and set 
 \begin{equation*}
  \cone[x]{m}=C_{x}(1, \dots, m).
 \end{equation*}
 Note that, at this moment, $\angle_{x}(\eta,\xxi{m})$ could be
 equal to $\pi$ and $\sect[x]{\eta}{\xxi{m}}$ could degenerate to
 a geodesic line.
\end{minipage}
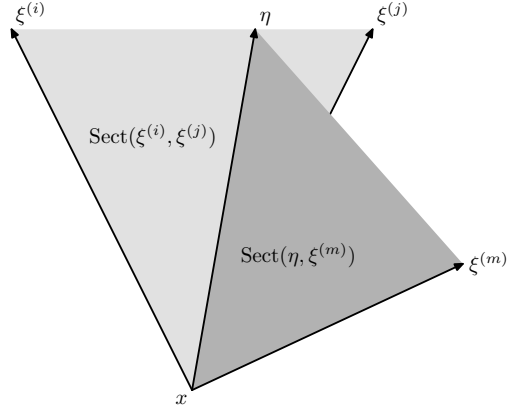
\begin{minipage}[c]{6cm}
\hspace{3cm}
\includegraphics[scale=.8]{fig.8} 
\end{minipage}

\subsection{Nondegeneracy of a cone $\cone[x]{m}$}
\label{sec:nondegeneration_of_cone}

In this subsection, we prepare for the proof of
Proposition~\ref{prop:simplex}, which asserts that $\cone[x]{m}$ is a
Euclidean cone over a Euclidean regular $(m-1)$-simplex; we show that a
cone $\cone[x]{m}$ does not degenerate under our assumption $\zure>0$.

%
%

\begin{Lemma}
\label{lem:sector_exists}
 Under our setting, for any $x \in f(\Gamma)$, $q \in \ray{x,\xxi{m}}$,
 and $\eta, \eta' \in \partial \cone[x]{m}$, 
 $\ray{q,\eta}\cup \ray{q,\eta'}$ bounds a flat sector or a flat 
 half-plane,  which we denote by $\sect[q]{\eta}{\eta'}$. 
 In particular, for any $\eta, \eta' \in \partial \cone[x]{m}$, a
 geodesic $[\eta,\eta']$ with respect to $d_{\angle}$ is contained in
 $\partial \cone[x]{m}$, namely, $\partial \cone[x]{m}$ is convex with
 respcet to $d_{\angle}$. 
\end{Lemma}

\begin{proof}
 If $m=2$, then 
 $C_{x}^{(2)}=\sect[x]{\xxi{1}}{\xxi{2}}$. 
 For any $q \in C^{(2)}$ and $\eta, \eta' \in \partial C_{x}^{(2)}$, 
 we have
 $\ray{q,\eta}\cup \ray{q,\eta'} \subset
 \sect[x]{\xxi{1}}{\xxi{2}}$, 
 and since $\sect[x]{\xxi{1}}{\xxi{2}}$ is a flat sector, the convex
 hull of $\ray{q,\eta}\cup \ray{q,\eta'}$ must be a flat sector 
 $\sect[q]{\eta}{\eta'}$ contained in $\sect[x]{\xxi{1}}{\xxi{2}}$, 
 and the assertion is obvious.

 Suppose that the assertion is true for $\cone[x]{m-1}$.
 Take any $\eta, \eta' \in \partial \cone[x]{m}$, and take 
 $\tilde\eta, \tilde\eta' \in \cone[x]{m-1}$ so that
 $\eta \in \partial \sect[x]{\xxi{m}}{\tilde \eta}$ and 
 $\eta' \in \partial \sect[x]{\xxi{m}}{\tilde \eta'}$. 
 (If $\eta=\xxi{m}$, then any choice of 
 $\tilde\eta \in \partial \cone[x]{m-1}$ makes sense in what follows.)
 Suppose first $d{\angle}(\tilde\eta,\tilde\eta')<\pi$.
 Then, by our assumption, for any $x \in f(\Gamma)$, 
 $\ch{\ray{x,\tilde\eta}\cup\ray{x,\tilde\eta'}}$ is a nondegenerate
 flat sector. 
 Also, as we have seen in the last paragraph of the preceding
 subsection, for any 
 $\xi \in [\tilde\eta,\tilde\eta']\subset \partial \cone[x]{m-1}$, 
 we have a half-infinite parallelogram 
 $F_{xx_j}^{\xi}=\bigcup_{q' \in [x,x_j]}\cc{q'}{\xi}(\ray{0,\infty})$. 
 By applying Lemma~\ref{lem:parallelogram_and_sliding} $(2)$ to
 $\sect[x]{\tilde\eta}{\tilde\eta'}\cup\sect[x_j]{\tilde\eta}{\tilde\eta'}$, 
 namely, by setting $p_1=x$, $p_2=x_j$, $\eta_1=\tilde\eta$, and
 $\eta_2=\tilde\eta'$ in Lemma~\ref{lem:parallelogram_and_sliding}, 
 we see that, for any $q_j \in [x,x_j]$, 
 $\ch{\ray{q_j,\tilde\eta}\cup \ray{q_j,\tilde\eta'}}$ is a flat sector. 
 By taking an ultralimit, we also see that, 
 for any  $q \in \ray{x,\xxi{m}}$, 
 $\ch{\ray{q,\tilde\eta}\cup \ray{q,\tilde\eta'}}$ is a flat sector. 
 Furthermore, 
 $F_{xx_j}^{\xi}=\bigcup_{q' \in [x,x_j]}\cc{q'}{\xi}(\ray{0,\infty})$
 being a half-infinite parallelogram implies that,  
 by taking an ultralimit again, 
 for any $\xi \in [\tilde\eta,\tilde\eta']$, 
 $F_{x(m)}^{\xi}=
 \bigcup_{q \in \ray{x,\xxi{m}}}\cc{q}{\xi}(\ray{0,\infty})$ 
 is a flat sector. 
 Therefore we can apply Lemma~\ref{lem:parallelogram_and_sliding} to 
 $\sect[x]{\tilde\eta}{\tilde\eta'} \cup 
 \sect[\cc{x}{(m)}(T)]{\tilde\eta}{\tilde\eta'}$
 for any $T>0$. 
 Note that 
\begin{equation*}
  F_{x (m)}^{\tilde\eta}=\sect[x]{\xxi{m}}{\tilde\eta}, 
   \quad
 F_{x (m)}^{\tilde\eta'}= \sect[x]{\xxi{m}}{\tilde\eta'}, 
\end{equation*} 
 and that, for any $q \in [x,\xxi{m})$, we have
\begin{equation*}
  c_q^{\eta}(\ray{0,\infty}) \subset F_{x (m)}^{\tilde\eta}, \quad
  c_q^{\eta'}(\ray{0,\infty}) \subset F_{x (m)}^{\tilde\eta'}. 
\end{equation*}
 Therefore, for $s>0$, 
 we can take $t(s)$ and $t'(s)$ so that
\begin{equation*}
  \cc{q}{\eta}(t(s)) \in R_{ss} \cap F_{x (m)}^{\tilde\eta}, 
 \ \text{ and }\ 
  \cc{q}{\eta'}(t'(s)) \in R_{ss} \cap F_{x (m)}^{\tilde\eta'}, 
\end{equation*} 
 where $R_{ss}$ is the convex hull of 
 $[\cc{x}{\tilde\eta}(s),\cc{x}{\tilde\eta'}(s)] \cup 
  [\cc{\cc{x}{(m)}(T)}{\tilde\eta}(s),\cc{\cc{x}{(m)}(T)}{\tilde\eta'}(s)]$ 
 for suitable $T>0$.
 Applying Lemma~\ref{lem:parallelogram_and_sliding} $(4)$ to 
 $\sect[x]{\tilde\eta}{\tilde\eta'} \cup 
 \sect[\cc{x}{(m)}(T)]{\tilde\eta}{\tilde\eta'}$, 
 we see that $\tri{q}{c_q^{\eta}(t(s))}{c_q^{\eta'}(t'(s))}$ is a flat
 triangle. 
 Letting $s \to \infty$ shows that $\ch{\ray{q,\eta}\cup \ray{q,\eta'}}$
 is a flat sector. 
 Also, since 
\begin{equation*}
  \ch{\ray{q,\eta}\cup \ray{q,\eta'}} 
  = \sect[q]{\eta}{\eta'} \subset 
 \bigcup_{q \in \ray{x,\xxi{m}}}\sect[q]{\tilde\eta}{\tilde\eta'}
 = \bigcup_{\xi \in [\tilde\eta,\tilde\eta']} \sect[x]{\xi}{\xxi{m}}, 
\end{equation*} 
 we see that a geodesic 
 $[\eta,\eta']=\partial\sect[q]{\eta}{\eta'}$ lies in
 $\partial\cone[x]{m}$. 

 Now suppose that $\angle_{x}(\tilde\eta,\tilde\eta')=\pi$.  
 Then
  $\sect[x]{\tilde\eta}{\xxi{m}}\cup \sect[x]{\tilde\eta'}{\xxi{m}}$
 is a half-plane that can be expressed as 
 $\bigcup_{q \in \ray{x,\xxi{m}}} 
 (\ray{q,\tilde\eta}\cup \ray{q,\tilde\eta'})$
 as we have seen in Remark~\ref{remark:when_angle_is_pi}. 
 Clearly, this half-plane contains $\ray{x,\eta}, \ray{x,\eta'}$. 
 Thus the assertion is also clear in this case.
\end{proof}

%
%
 The lemma above tells us that $\cone[x_k]{m}$ is not degenerating in
 the following sense. 

\begin{Proposition}
 \label{prop:less_than_pi}
 Take any $x\in f(\Gamma)$. 
 Under our setting, the following are true. 

 $(1)$ For any $\eta \in \partial \cone[x]{m-1}$, 
 $0<\angle_{x}(\eta,\xxi{m}) < \pi$. 

 $(2)$ For any $\eta_1, \eta_2 \in \partial \cone[x]{m-1}$ with
 $\eta_1\not=\eta_2$, 
 $\sect[x]{\eta_1}{\xxi{m}}\cap\sect[x]{\eta_2}{\xxi{m}}
 =\ray{x,\xxi{m}}$.

 $(3)$ For any $\eta_1, \eta_2 \in \partial \cone[x]{m-1}$ and 
 $q,q' \in \ray{x,\xxi{m}}$ with $\eta_1\not= \eta_2$ and 
 $q\not=q'$, 
 $\sect[q]{\eta_1}{\eta_2} \cap \sect[q']{\eta_1}{\eta_2}=\emptyset$. 
\end{Proposition}

\begin{proof}
$(1)$ Suppose $\angle_{x}(\eta,\xxi{m})=0$. 
 Since $\ch{\ray{x,\eta}\cup \ray{x,\xxi{m}}}$ is a flat sector by
 Lemma~\ref{lem:sector_exists}, this means
 $\ray{x,\eta}=\ray{x,\xxi{m}}$. 
 Thus, if we take $q_j=\cc{x}{x_j}(1)$,  we see that 
 $\uulim[m]_j q_j= \q{m}=\cc{x}{(m)}(1) \in\ray{x,\eta}\subset Y^{(m-1)}$. 
 Since $\q{m}, q_j \in Y^{(m-1)}$, this means that 
 $\ulim_j d(q_j,\q{m})=0$. 
 Hence $\q{m+1}=\uulim[m+1]_jq_j = \q{m} \in \ray{x,\eta}$ also holds,
 and we see that 
 $[x,\q{m}] \subset \ray{x,\xxi{m}}\cap \ray{x,\xxi{m+1}}$. 
 In particular, we get $\angle_{x}(\xxi{m},\xxi{m+1})=0$. 
 This contradicts our assumption $\zure>0$. 
 Therefore we conclude that $\angle_{x}(\eta,\xxi{m})>0$. 

 Now we show that $\angle_{x}(\eta,\xxi{m})<\pi$. 
 Take $\q{m}$, $\q{m+1}$, and $q_j$ as above. 
 We have $\uulim[m]_j q_j=\q{m}$. 
 Suppose that $\angle_{x}(\eta,\xxi{m})=\pi$. 
 Then, for $\tilde q=\cc{x}{\eta}(1)$, we have
 $[\q{m},\tilde q]=[\q{m},x]\cup [x,\tilde q]$ and 
\begin{equation*}
 2=d(\q{m},\tilde q)=d(\uulim[m]_j q_j, \tilde q)=
 \ulim_j d(q_j,\tilde q). 
\end{equation*}
 From this, we see that $\q{m+1}=\uulim[m+1]_j q_j=\cc{x}{(m+1)}(1)$
 also satisfies
\begin{equation*}
 d(\q{m+1},\tilde q)=\ulim_j d(q_j,\tilde q) =2, 
\end{equation*}
 and also we have $d(\q{m+1},x)=d(x,\tilde q)=1$ by our choice of $q_j$
 and $\tilde q$. 
 Therefore we see that 
 $[\q{m+1},\tilde q]=[\q{m+1},x]\cup [x,\tilde q]$. 
 In particular, 
 $[x,\tilde q] \subset [\q{m},\tilde q]\cap [\q{m+1},\tilde q]$
 and $\tri{\q{m}}{\q{m+1}}{\tilde q}$ is a half-degenerate triangle. 
 
 On the other hand, noting Remark~\ref{remark:when_angle_is_pi}, 
 we see that $F_{xx_j}^{\eta}\cup F_{xx_j}^{(m)}$ forms a flat strip in
 this case. 
 Thus we see that 
 $\tri{q_j}{\tilde q}{\q{m}}\subset F_{xx_j}^{\eta}\cup F_{xx_j}^{(m)}$
 is flat for each $j$. 
 Since $d(x,x_j)\to \infty$ and 
 $\ulim_j \angle_{x}(x_j,\xxi{m})=\zure>0$, the width of 
 $F_{xx_j}^{\eta}\cup F_{xx_j}^{(m)}$ tends to $\infty$, and hence, 
 for almost all $j \in \N$ with respect to $\lambda$, we may assume that 
 $\cc{\tilde q}{q_j}$ is defined for $t=2$, and hence
 $\tri{\tilde q}{\cc{\tilde q}{(m)}(2)}{\cc{\tilde q}{q_j}(2)} 
 \subset F_{xx_j}^{\eta}\cup F_{xx_j}^{(m)}$ 
 is a flat triangle. 
 Thus we have
\begin{equation}
\label{eq:half_triangle}
 d(\cc{\tilde q}{\q{m}}(1), \cc{\tilde q}{q_j}(1)) =
 \frac{1}{2}d(\cc{\tilde q}{\q{m}}(2), \cc{\tilde q}{q_j}(2))
 = \frac{1}{2}d(\q{m},\cc{\tilde q}{q_j}(2)). 
\end{equation}
 However, we know that 
\begin{equation*}
 \ulim_j d(\q{m},\cc{\tilde q}{q_j}(2))= 
 d(\q{m},\q{m+1})=2 \sin \frac{\zure}{2} > 0
\end{equation*}
 and 
\begin{equation*}
 \ulim_j d(\cc{\tilde q}{\q{m}}(1), \cc{\tilde q}{q_j}(1))
  = d(\cc{\tilde q}{\q{m}}(1), \cc{\tilde q}{\q{m+1}}(1)) 
  =d (x,x)=0
\end{equation*}
 holds.  This contradicts (\ref{eq:half_triangle}). 
 Therefore $\angle_{x}(\eta,\xxi{m})$ cannot be equal to $\pi$. 

$(2)$ Suppose that there exists a point 
 $q \in \sect[x]{\eta_1}{\xxi{m}}\cap \sect[x]{\eta_2}{\xxi{m}}
 \setminus \ray{x,\xxi{m}}$. 
 Since 
 $F=\sect[x]{\eta_1}{\xxi{m}}\cap \sect[x]{\eta_2}{\xxi{m}}$ is
 convex, $[x,q] \subset F$, and we can find 
 $\eta_i' \in [\eta_i,\xxi{m}]$ such that 
 $[x,q]\subset \ray{x,\eta_i'}$. 
 If $\eta_1'\not=\eta_2'$, then we get
 $\angle_{x}(\eta_1',\eta_2')=0$, and this contradicts
 Lemma~\ref{lem:sector_exists}, which says that 
 $\sect[x]{\eta_1'}{\eta_2'}$ is a flat sector. 
 Therefore we see that $\eta_1'=\eta_2'$ and hence 
 $F$ is a flat sector; we can express
 $F=\sect[x]{\eta}{\xxi{m}}$. 
 
 We may assume 
 $\angle_{x}(\eta_1,\xxi{m})\geq \angle_{x}(\eta_2,\xxi{m})$
 in what follows. 
 Suppose that $\eta\not= \eta_1, \eta_2$. 
 Then $\angle_x(\eta,\xxi{m})<\angle_x(\eta_i,\xxi{m})$ for $i=1,2$. 
 Take $q' \in (x,\xxi{m})$ and consider a geodesic ray 
 $\ray{q',\eta_2} \subset \sect[x]{\eta_2}{\xxi{m}}$. 
 Since $\ray{q',\eta_2}$ is parallel to $\ray{x,\eta_2}$ and 
 $\angle_{x}(\eta,\xxi{m})<\angle_{x}(\eta_2,\xxi{m})$, 
 it must intersect with $(x,\eta)$ at some point in
 $\sect[x]{\eta}{\xxi{m}}\subset \sect[x]{\eta_i}{\xxi{m}}$. 
 Denote this point by $q''$. 
 Since we are assuming
 $\angle_{x}(\eta_1,\xxi{m})\geq \angle_{x}(\eta_2,\xxi{m})$, 
 a geodesic segment $[q',q'']$ can be extended to a geodesic ray in 
 $\sect[x]{\eta_1}{\xxi{m}}$; we denote its endpoint by $\eta_1''$. 
 Since $[q',q''] \subset \ray{q',\eta_2}\cap \ray{q',\eta_1''}$,  
 we see that $\angle_{q'}(\eta_2,\eta_1'')=0$. 
 On the other hand, by our choice of $\eta$ and our assumption,  
 we have $\eta_1''\not=\eta_2$, and hence, 
 by Lemma~\ref{lem:sector_exists},  we see that
 $\ch{\ray{q',\eta_1''}\cup \ray{q',\eta_2}}$ is a flat sector. 
 This contradicts $\angle_{q'}(\eta_2,\eta_1'')=0$. 
 Thus we see that $\eta=\eta_2$ holds, and we may
 assume that
 $\ray{x,\eta_2}\subset \sect[x]{\eta_1}{\xxi{m}}$.

 Let $\p{m}=\cc{x}{(m)}(1)$ and $q_1 = \cc{x}{\eta_1}(1)$.
 Then we see that $[q_1,\p{m}]$ and $\ray{x,\eta_2}$ intersect at a
 point. 
 Denote this ponit by $q_2$. 
 Then $[q_1,q_2] \subset \sect[x]{\eta_1}{\eta_2}\subset Y^{(m-1)}$. 
 Take $p_j = \cc{x}{x_j}(1)$.  
 Then we see that $\uulim[m]_jp_j=\p{m}$ and
 $\uulim[m]_j[q_1,p_j]=[q_1,\p{m}]$. 
 Since $\eta_1\not= \eta_2$, there exists 
 $\bar q \in (q_1,q_2)\subset [q_1,\p{m}]$, and 
 we can take $p_j' \in (q_1,p_j)$ so that $\uulim[m]_j p_j' = \bar q$. 
 Since $p_j', \bar q  \in Y^{(m-1)}$, $\{p_j'\}$ actually converges to
 $\bar q$ in $Y^{(m-1)}$. 
 Therefore we also have $\uulim[m+1]p_j'=\bar q \in Y^{(m-1)}$, in
 particular, $[q_1,\bar q]\subset [q_1,\p{m}]\cap [q_1,\p{m+1}]$, where
 $\p{m+1}=\uulim[m+1]p_j$. 
 By applying  Lemma~\ref{lem:parallelogram_and_sliding} $(4)$ to 
 $\sect[x]{\eta_1}{\xxi{m}}\cup \sect[x_j]{\eta_1}{\xxi{m}}$, we see
 that $\tri{q_1}{\p{m}}{p_j}$ is a flat triangle for any $j$. 
 Thus so is $\tri{q_1}{\p{m}}{\p{m+1}}$.
 Since $\angle_{x}(\xxi{m},\xxi{m+1})=\zure>0$, we know that
 $\p{m+1}\not= \p{m}$. 
 However, $[q_1,\bar q]\subset [q_1,\p{m}]\cap [q_1,\p{m+1}]$ implies
 $\angle_{q_1}(\p{m},\p{m+1})=0$.  
 Since $d(q_1,\p{m})=d(q_1,\p{m+1})=\ulim_j d(q_1,p_j)$ and 
 $\tri{q_1}{\p{m}}{\p{m+1}}$ is flat, this means that 
 $\p{m}=\p{m+1}$. 
 A contradiction. 
 Therefore we conclude that 
 $\sect[x]{\eta_1}{\xxi{m}}\cap\sect[x]{\eta_2}{\xxi{m}}=\ray{x,\xxi{m}}$.

 $(3)$ Suppose that 
 $p \in \sect[q]{\eta_1}{\eta_2}\cap \sect[q']{\eta_1}{\eta_2}$
 and $p\not=q$. 
 Since $\sect[q]{\eta_1}{\eta_2}$ is a flat sector, we see that 
 $\ray{p,\eta_i} \subset \sect[q]{\eta_1}{\eta_2}$ for $i=1,2$, 
 and hence 
 $\sect[p]{\eta_1}{\eta_2} \subset \sect[q]{\eta_1}{\eta_2}$.  
 If $p \in \ray{x,\xxi{m}}$, then $[q,p]\subset \ray{x,\xxi{m}}$. 
 Since $\sect[q]{\eta_1}{\eta_2}$ is a flat sector, there exists 
 $\eta \in [\eta_1,\eta_2]\subset \partial \cone[x]{m-1}$ such that 
 $[q,p] \subset \ray{q,\eta}$. 
 Then we see that $\angle_{q}(\eta,\xxi{m})=0$. 
 However, since
 $\ch{\ray{q,\eta}\cup \ray{q,\xxi{m}}}$ is a flat sector
 by Lemma~\ref{lem:sector_exists}, 
 we get $\angle_{x}(\eta,\xxi{m})=\angle_{q}(\eta,\xxi{m})=0$. 
 This contradicts $(1)$ above. 
 Thus $p \not\in \ray{x,\xxi{m}}$. 

 Then there exists 
 $\eta, \eta' \in [\eta_1,\eta_2] \subset \partial \cone[x]{m-1}$
 such that $p \in \ray{q,\eta}$ and $p \in \ray{q',\eta'}$ hold. 
 Suppose that $\eta\not= \eta'$. 
 Then, since $\ray{q,\eta} \subset \sect[x]{\eta}{\xxi{m}}$ and
 $\ray{q',\eta'} \subset \sect[x]{\eta'}{\xxi{m}}$, this means that 
 $p \in \sect[x]{\eta}{\xxi{m}}\cap \sect[x]{\eta'}{\xxi{m}}
 \setminus \ray{x,\xxi{m}}$, 
 and this contradicts $(2)$ above. 
 Thus we have $\eta=\eta'$ and 
 $p \in \sect[x]{\eta}{\xxi{m}}$. 
 Since $\sect[x]{\eta}{\xxi{m}}$ is a flat sector, 
 $\ray{q,\eta}$ and $\ray{q',\eta}$ cannot intersect unless 
 either $q \in \ray{q',\eta}$ or $q' \in \ray{q,\eta}$. 
 Since $q,q' \in \ray{x,\xxi{m}}$, this again implies
 either $\angle_{q}(\eta,\xxi{m})=0$ or 
 $\angle_{q'}(\eta,\xxi{m})=0$. 
 Since $\ch{\ray{q,\xxi{m}}\cup \ray{q,\eta}}$ and 
 $\ch{\ray{q',\xxi{m}}\cup \ray{q',\eta}}$ are both flat
 sectors by Lemma~\ref{lem:sector_exists}, this means that 
 $\eta=\xxi{m}$ and $p \in \ray{x,\xxi{m}}$, which is impossible. 
 Thus we conclude that 
 $\sect[q]{\eta_1}{\eta_2}\cap \sect[q']{\eta_1}{\eta_2}=\emptyset$. 
\end{proof}

%
%
\begin{Lemma}
 \label{lem:tau_extends}
 For any $x\in f(\Gamma)$, $\xi \in \partial C_{x}^{(m)}$ and $s\geq 0$, 
 $\trans{s}{\xi}\colon q \mapsto \cc{q}{\xi}(s)$ gives an isometric
 translation on $\cone[x]{m}$, that is,  $\trans{s}{\xi}$ satisfies
 $d(\trans{s}{\xi}(q), \trans{s}{\xi}(q'))=d(q,q')$ for any 
 $q,q' \in \cone[x]{m}$ and
 $\trans{s}{\xi}(\ray{q,\eta})=\ray{\trans{s}{\xi}(q),\eta}$ for any 
 $q \in \cone[x]{m}$ and $\eta \in \partial \cone[x]{m}$. 
\end{Lemma}

\begin{proof}
 We show that $\trans{s}{\xi}$ defined in
 Lemma~\ref{lem:parallelogram_and_sliding} induces an
 isometric translation from $\cone[x]{m}$ into $\cone[x]{m}$ by
 induction on $m$. 
 If $m=2$, then 
 $\cone[x]{2}=\sect[x]{\xxi{1}}{\xxi{2}}$ and the
 assertion is clear. 

 Suppose that the assertion is true for $\cone[x]{m-1}$.
 First note that, for any flat sector $\sect[q]{\eta}{\eta'}$, by
 identifying it with a flat sector in $\R^2$ with vertex at the origin,
 we can find a point $\bar q' \in \R^2$ corresponding to 
 $q' \in \sect[q]{\eta}{\eta'}$. 
 We also have $\bar p$ and $\bar p'$ in $\R^2$ corresponding to
 $\cc{q}{\eta}(1)$ and $\cc{q}{\eta'}(1)$ respectively. 
 Then through this identification, for any $s,t \geq 0$, we can write 
 \begin{equation}
 \label{eq:tau_commutative}
  \trans{s}{\eta'}\circ \trans{t}{\eta}(q') = s\bar p'+t\bar p+\bar q'
 = t\bar p+s\bar p' + \bar q' = \trans{t}{\eta}\circ \trans{s}{\eta'}(q'). 
 \end{equation}
 Note that this makes sense even if $\angle_q(\eta,\eta')=\pi$
 and $\sect[q]{\eta}{\eta'}$ degenerates to a geodesic line,
 although the expression $s\bar p'+t\bar p$ is not unique. 

 Take $\xi \in \partial \cone[x]{m}$. 
 Then there exists $\eta_1 \in \cone[x]{m-1}$ such that 
 $\xi\in \partial \sect[x]{\eta_1}{\xxi{m}}$. 
 We show that there exists $a_1, a_2\geq 0$ such that we can express  
 \begin{equation*}
 \trans{s}{\xi}=\tau_{a_1 s}^{\eta_1}\circ \trans{a_2 s}{\xxi{m}}, 
 \end{equation*}
 and that two maps on the right-hand side are isometric translations. 

 Since $\ray{x,\xi} \subset \sect[x]{\eta_1}{\xxi{m}}$, 
 we can express as $\cc{x}{\xi}(t)=(a_1t)p_1 + (a_2t)p^{(m)}$, where
 $p_1=\cc{x}{\eta_1}(1)$ and $p^{(m)}=\cc{x}{(m)}(1)$. 
 By the description (\ref{eq:tau_commutative}), 
 we see that 
 $\trans{s}{\xi}(q)=\trans{a_1 s}{\eta_1}\circ \trans{a_2s}{\xxi{m}}(q)$
 holds for $q \in \sect[x]{\eta_1}{\xxi{m}}$. 
 Now take any $q' \in \cone[x]{m}$. Then there exists
 $\eta_2 \in \partial \cone[x]{m-1}$ such that 
 $q' \in \sect[x]{\eta_2}{\xxi{m}}$ holds, and we can express
 $q'=\cc{q}{\eta_2}(t)$ for some $q \in \ray{x,\xxi{m}}$ and 
 $t\geq 0$.
 Since $q, \trans{a_2s}{\xxi{m}}(q)\in \ray{x,\xxi{m}}$, 
 the convex hulls of $\ray{q,\xi}\cup \ray{q,\eta_2}$, 
 $\ray{\trans{a_2s}{\xxi{m}}(q),\eta_1}\cup
 \ray{\trans{a_2s}{\xxi{m}}(q),\eta_2}$, 
 and $\ray{q,\eta_2}\cup \ray{q,\xxi{m}}$
 are flat sectors expressed as $\sect[q]{\xi}{\eta_2}$, 
 $\sect[\trans{a_2s}{\xxi{m}}(q)]{\eta_1}{\eta_2}$, and
 $\sect[q]{\eta_2}{\xxi{m}}$, respectively, by
 Lemma~\ref{lem:sector_exists}.
 Then we have
 \begin{equation*}
 \begin{split}
   \trans{s}{\xi}(\cc{q}{\eta_2}(t)) & 
   = \trans{t}{\xi}\circ \trans{t}{\eta_2}(q)   
   \overset{(1)}{=}\trans{t}{\eta_2} \circ \trans{s}{\xi}(q)
  \overset{(4)}{=} \trans{t}{\eta_2} \circ \trans{a_1s}{\eta_1} \circ
  \trans{a_2s}{\xxi{m}}(q) \\
  & \overset{(2)}{=} \trans{a_1s}{\eta_1}\circ \trans{t}{\eta_2} \circ
      \trans{a_2s}{\xxi{m}}(q) 
   \overset{(3)}{=} \trans{a_1s}{\eta_1}\circ
  \trans{a_2s}{\xxi{m}}\circ \trans{t}{\eta_2}(q)  \\
  & = \trans{a_1s}{\eta_1}\circ \trans{a_2s}{\xxi{m}}
  (\cc{q}{\eta_2}(t)), 
 \end{split}
 \end{equation*}
 where we deduce equalities $(1)$, $(2)$, and $(3)$ 
 by applying (\ref{eq:tau_commutative})
 on $\sect[q]{\xi}{\eta_2}$, 
 $\sect[\trans{a_2s}{\xxi{m}}(q)]{\eta_1}{\eta_2}$, and 
 $\sect[q]{\eta_2}{\xxi{m}}$,  respectively, and $(4)$ follows
 from the fact that 
 $\trans{s}{\xi}(q)=\trans{a_1 s}{\eta_1}\circ 
 \trans{a_2 s}{\xxi{m}}(q)$
 holds for $q \in \ray{x,\xxi{m}}\subset 
 \sect[x]{\eta_1}{\xxi{m}}$, 
 which we have just seen above. 
 Thus we see that, for any $q' \in \cone[x]{m}$, 
 $\trans{t}{\xi}(q')=\trans{a_2t}{\eta_1}\circ\trans{a_1t}{\xxi{m}}(q')$
 holds. 

 Next we show that $\trans{t}{\xxi{m}}$ and $\trans{s}{\eta_1}$ are
 both isometries. 
 First we consider $\trans{t}{\xxi{m}}$.
 Take any $q_i \in \cone[x]{m}$, $i=1,2$. 
 Then we have $\eta_i' \in \cone[x]{m-1}$, $i=1,2$, so that 
 $q_i \in \sect[x]{\eta_i'}{\xxi{m}}$, and there exists
 $\hat \eta_i \in [\eta_i',\xxi{m}]$ such that 
 $q_i \in \ray{x,\hat \eta_i} \subset
 \sect[x]{\eta_i'}{\xxi{m}}$. 
 We can express $q_i=\cc{x}{\hat \eta_i}(t_i)$ for some $t_i\geq 0$, and
 we know that there exists a flat sector
 $\sect[q]{\hat \eta_1}{\hat \eta_2}$ for each 
 $q \in \ray{x, \xxi{m}}$ by Lemma~\ref{lem:sector_exists}. 
 Noting that 
 \begin{equation*}
 \trans{t}{\xxi{m}}(q_i)=
 \trans{t}{\xxi{m}}\circ \trans{t_i}{\hat \eta_i}(x)
 =\trans{t_i}{\hat \eta_i}\circ \trans{t}{\xxi{m}}(x)
 =\trans{t_i}{\hat \eta_i}(\cc{x}{\xxi{m}}(t)), 
 \end{equation*}
 and applying Lemma~\ref{lem:parallelogram_and_sliding} $(1)$
 with  $p_1=x$, $p_2=\cc{x}{\xxi{m}}(t)$, and 
 $\eta_i=\hat \eta_i$, 
 we conclude that the convex hull of 
 $[q_1,q_2]\cup [\trans{t}{\xxi{m}}(q_1), \trans{t}{\xxi{m}}(q_2)]$
 is a flat parallelogram. 
 In particular, we have 
 $d(q_1, q_2)=d(\trans{t}{\xxi{m}}(q_1),\trans{t}{\xxi{m}}(q_2))$. 
 Therefore $\trans{t}{\xxi{m}}$ preserves the distance. 

 Now consider $\trans{s}{\eta_1}$. 
 Note that $\trans{s}{\eta_1}$ is an isometric translation on
 $\cone[x]{m-1}$ by our assumption. 
 Take any $q \in \ray{x,\xxi{m}}$ and choose $q_j \in [x,x_j]$
 satisfying $\uulim[m]_j q_j=q$. 
 Then, for any $\eta' \in \partial \cone[x]{m-1}$, we have
 \begin{equation}
 \begin{split}
   \label{eq:tau_commute_with_ulim}
  \uulim[m]_j \trans{s}{\eta_1}(\cc{q_j}{\eta'}(t))
  & = \uulim[m]_j \cc{\cc{q_j}{\eta'}(t)}{\eta_1}(s)
  = \cc{\cc{q}{\eta'}(t)}{\eta_1}(s)
  =\trans{s}{\eta_1}(\cc{q}{\eta'}(t)) \\
  & =\trans{s}{\eta_1}(\uulim[m]_j \cc{q_j}{\eta'}(t)). 
 \end{split} 
 \end{equation}
 On the other hand, applying
 Lemma~\ref{lem:parallelogram_and_sliding} $(3)$ to 
 $\bigcup_{\tilde q \in [x,x_j]}\sect[\tilde q]{\eta_1}{\eta'}$, 
 we see that
 $\trans{s}{\eta_1}$ maps $R_{0t} \subset F_{xx_j}^{\eta'}$ 
 isometrically onto its image for any $t\geq 0$, 
 where $R_{0t}$ is the convex hull of 
 $[x,x_j]\cup [\cc{x}{\eta'}(t), \cc{x_j}{\eta'}(t)]$ in this case.
 Hence
 $\trans{s}{\eta_1}$ maps whole 
 $F_{xx_j}^{\eta'}=\bigcup_{t\geq 0}R_{0t}$ isometrically onto its
 image. 
 By (\ref{eq:tau_commute_with_ulim}) and 
\begin{equation*}
  \uulim[m]_j F_{xx_j}^{\eta'}=
  \uulim[m]_j \bigcup_{\tilde q \in [x,x_j]}
   \cc{\tilde q}{\eta'}(\ray{0,\infty})= 
  \sect[x]{\eta'}{\xxi{m}}, 
\end{equation*} 
we see that
 \begin{equation*}
 \trans{s}{\eta_1}(\sect[x]{\eta'}{\xxi{m}}) =
  \ulim_j \trans{s}{\eta_1}(F_{xx_j}^{\eta'}). 
 \end{equation*}
 Thus we conclude that 
 $\trans{s}{\eta_1}(\sect[x]{\eta'}{\xxi{m}})$ is a flat sector, 
 and that $\trans{s}{\eta_1}$ maps $\sect[x]{\eta'}{\xxi{m}}$
 isometrically onto a flat sector 
 $\trans{s}{\eta_1}(\sect[x]{\eta'}{\xxi{m}})$ for any 
 $\eta' \in \partial \cone[x]{m-1}$. 
 Take any pair of points $q',q'' \in \cone[x]{m}$. 
 Then there exists $q_i \in \cone[x]{m-1}$, $t_i\geq 0$, and 
 $\eta_i' \in \partial \cone[x]{m-1}$, $i=1,2$,
 such that 
\begin{equation*}
 q'= \cc{q_1}{(m)}(t_1), \quad 
 q'' = \cc{q_2}{(m)}(t_2), \quad q_i \in \ray{x,\eta_i'},\ i=1,2. 
\end{equation*}
 We may assume $t_1 \leq t_2$. 
 Let $\tilde q'=\cc{q_2}{(m)}(t_1)$ and $\tilde q''=\cc{q_1}{(m)}(t_2)$. 
 Take $s_i$, $i=1,2$, so that 
  $q'=\cc{\cc{x}{(m)}(t_1)}{\eta_1'}(s_1)$ and 
 $q''=\cc{\cc{x}{(m)}(t_2)}{\eta_2'}(s_2)$. 
 Then we have $\tilde q' = \cc{\cc{x}{(m)}(t_1)}{\eta_2'}(s_2)$ 
 and $\tilde q'' = \cc{\cc{x}{(m)}(t_2)}{\eta_1'}(s_1)$, and the 
 the convex hull of $[q',\tilde q'] \cup [\tilde q'', q'']$ can be
 expressed as $R_{s_1s_2}$.
 By applying Lemma~\ref{lem:parallelogram_and_sliding} $(1)$ to
 $\bigcup_{q \in [\cc{x}{(m)}(t_1),\cc{x}{(m)}(t_2)]}
 \sect[q]{\eta_1'}{\eta_2'}$, 
 we see that $R_{s_1s_2}$ in this set is a flat
 parallelogram and $q'$ and $q''$ form a pair of diagonal
 vertices. 
 Another pair of diagonal vertices of $R_{s_1s_2}$ is given by 
 $\tilde q'$ and $\tilde q''$.
 The sides of $R_{s_1s_2}$ take the form of
\begin{equation*}
\begin{split}
 & [q',\tilde q']=R_{s_1s_2} \cap
  \sect[\cc{x}{(m)}(t_1)]{\eta_1'}{\eta_2'}, \quad 
   [q'',\tilde q'']=R_{s_1s_2} \cap
  \sect[\cc{x}{(m)}(t_2)]{\eta_1'}{\eta_2'}, \\
 & [q',\tilde q'']=R_{s_1s_2}\cap\sect[x]{\eta_1'}{\xxi{m}},\quad  
   [\tilde q',q'']=R_{s_1s_2}\cap\sect[x]{\eta_2'}{\xxi{m}}. 
\end{split}
\end{equation*}

\vspace{-1.5cm}
\hspace{3.5cm}
\includegraphics[scale=.8]{fig.9}
\vspace{1.5cm}

\noindent
 Since $\trans{s}{\eta_1}$ is an isometric translation on
 $\cone[x]{m-1}$ and $\trans{s}{\xxi{m}}$ is an isometry onto its
 image as we have seen, recalling (\ref{eq:tau_commutative}), we
 see that, for any $s\geq 0$, 
\begin{equation*}
\begin{split}
  d(\trans{s}{\eta_1}(q'),\trans{s}{\eta_1}(\tilde q'))
 & = d(\trans{s}{\eta_1}\circ \trans{t_1}{\xxi{m}}(q_1),
    \trans{s}{\eta_1}\circ \trans{t_1}{\xxi{m}}(q_2)) \\
 & =d(\trans{t_1}{\xxi{m}}\circ \trans{s}{\eta_1}(q_1), 
  \trans{t_1}{\xxi{m}}\circ \trans{s}{\eta_1}(q_2)) \\
 & = d(q_1,q_2). 
\end{split}
\end{equation*}
 Replacing $t_1$ with $t_2$, we see that 
\begin{equation*}
 d(\trans{s}{\eta_1}(q''),\trans{s}{\eta_1}(\tilde q''))= 
 d(q_1,q_2). 
\end{equation*}
 As we have seen above, for any $\eta'\in\partial\cone[x]{m-1}$, 
 $\trans{s}{\eta_1}$ maps a flat sector $\sect[x]{\eta'}{\xxi{m}}$
 isometrically onto its image.
 Hence we also have 
\begin{equation*}
\begin{split}
  d(\trans{s}{\eta_1}(q'),\trans{s}{\eta_1}(\tilde q''))
 & = d(q',\tilde q'')=d(\cc{q_1}{(m)}(t_1),\cc{q_1}{(m)}(t_2))
   =|t_2-t_1|,  \\
  d(\trans{s}{\eta_1}(\tilde q'),\trans{s}{\eta_1}(q''))
 & = d(\tilde q',q'')
   =d(\cc{q_2}{(m)}(t_1),\cc{q_2}{(m)}(t_2))=|t_2-t_1|. 
\end{split}
\end{equation*}
 Thus we see that $\trans{s}{\eta_1}(R_{s_1s_2})$ is a
 parallelogram having the same side length for any $s\geq 0$. 
 Looking at 
\begin{equation*}
 s \mapsto d(\trans{s}{\eta_1}(q'),\trans{s}{\eta_1}(q'')) \
 \text{ and } 
 s \mapsto d(\trans{s}{\eta_1}(\tilde q'),
    \trans{s}{\eta_1}(\tilde q'')), 
\end{equation*}
 we conlude that both are constant functions as in the proof of 
 Lemma~\ref{lem:parallelogram_and_sliding} $(3)$. 
 Therefore 
 $d(\trans{s}{\eta_1}(q'), \trans{s}{\eta_2}(q''))=d(q',q'')$
 for any $q', q'' \in \cone[x]{m}$, and hence 
 $\trans{s}{\eta_1}$ maps $\cone[x]{m}$ isometrically onto its image. 

 Finally, we show that $\trans{s}{\xi}$ is a translation. 
 First note that,  by (\ref{eq:tau_commutative}), we see that, for 
 $q \in \ray{x,\xxi{m}}$ and $\eta \in \partial \cone[x]{m}$, 
 \begin{equation*}
   \trans{s}{\xi}(\cc{q}{\eta}(t)) 
  = \trans{s}{\xi}\circ \trans{t}{\eta}(q)
  =\trans{t}{\eta}\circ \trans{s}{\xi}(q) 
  = c_{\trans{s}{\xi}(q)}^{\eta}(t). 
 \end{equation*}
 Varying $t \in \ray{0,\infty}$, we get 
 \begin{equation*}
  \trans{s}{\xi}(\ray{q,\eta}) = \ray{\trans{s}{\xi}(q),\eta}. 
 \end{equation*}
 Now take any $q'=\cc{x}{\xi'}(s') \in \cone[x]{m}$. 
 We know that $\trans{s'}{\xi'}$ is an isometry, and we have just seen
 that it maps $\ray{x,\eta'}$ to $\ray{q',\eta'}$ for any 
 $\eta' \in \partial \cone[x]{m}$.  
 Therefore, for any $\xi,\eta \in \partial \cone[x]{m}$, 
 $\trans{s'}{\xi'}$ maps a flat sector $\sect[x]{\xi}{\eta}$ to a flat
 sector bounded by $\ray{q',\xi}\cup \ray{q',\eta}$, that is, 
 $\sect[q']{\xi}{\eta}$. 
 Then, by applying (\ref{eq:tau_commutative}) on
 $\sect[q']{\xi}{\eta}$,  we see that
  \begin{equation*}
   \trans{s}{\xi}(\cc{q'}{\eta}(t)) =
    \trans{s}{\xi}\circ\trans{t}{\eta}(q')
    =\trans{t}{\eta}\circ \trans{s}{\xi}(q') 
    = \cc{\trans{s}{\xi}(q')}{\eta}(t), 
 \end{equation*}
 which means that an isometry $\trans{s}{\xi}$ maps a geodesic ray
 starting from $q' \in \cone[x]{m}$ and terminating at 
 $\eta \in \partial \cone[x]{m}$ to a geodesic ray terminating at 
 the same $\eta$, that is,  $\trans{s}{\xi}$ is an isometric
 translation.
 This completes the proof. 
\end{proof}

%
%
\begin{Proposition}
\label{prop:flat_triangle}
$(1)$ For any $x\in f(\Gamma)$, $q \in C_{x}^{(m)}$, and 
 $\eta_1,\eta_2 \in \partial\cone[x]{m}$ with $\eta_1\not=\eta_2$, 
 the convex hull of $[q,\eta_1)\cup [q,\eta_2)$ is a flat sector
 $($possibly degenerates to a geodesic line$)$. 

$(2)$ Let $\xi, \eta_1, \eta_2 \in \partial \cone[x]{m}$. 
 Suppose that $0<\angle_{x}(\eta_i,\xi)<\pi$, $i=1,2$. 
 Then, for any $p \in \ray{x,\xi}$, 
 $q_i \in \ray{x,\eta_i}$, $\tri{p}{q_1}{q_2}$ is a flat
 triangle. 
\end{Proposition}

\begin{proof}
$(1)$ Take any $q \in \cone[x]{m}$.  
 Expressing $q$ as $q=\cc{x}{\xi}(s)$, we see that 
 $\trans{s}{\xi}\colon \cone[x]{m}\rightarrow \cone[x]{m}$
 is an isometric translation with $\trans{s}{\xi}(x)=q$ 
 by Lemma~\ref{lem:tau_extends} above. 
 Since 
 $\ray{q,\eta_1}\cup \ray{q,\eta_2}
 \subset\trans{s}{\xi}(\sect[x]{\eta_1}{\eta_2})$, 
 and $\trans{s}{\xi}(\sect[x]{\eta_1}{\eta_2})$ is a flat sector, 
 it is clear that  
 $\ch{\ray{q,\eta_1}\cup\ray{q,\eta_2}}\subset 
 \trans{s}{\xi}(\sect[x]{\eta_1}{\eta_2})$. 
 On the other hand, for any $q_1\in \ray{q,\eta_1}$ and 
 $q_2\in \ray{q,\eta_2}$,  $[q_1,q_2]$ must be contained in  
 $\ch{\ray{q,\eta}\cup \ray{q,\eta'}}$. 
 Therefore
 $\trans{s}{\xi}(\sect[x]{\eta_1}{\eta_2})\subset
 \ch{\ray{q,\eta}\cup\ray{q,\eta'}}$. 
 Thus we conclude that $\ch{\ray{q,\eta}\cup\ray{q,\eta'}}$ is a flat
 sector. 

$(2)$ If $\eta_1=\eta_2$, then $\tri{p}{q_1}{q_2}$ lies in
 $\sect[x]{\eta_1}{\xi}$, and hence it is flat. 
 Suppose that $\eta_1\not= \eta_2$. 
 We can write $q_i=\cc{x}{\eta_i}(s_i)$, and we may assume that 
 $s_1\geq s_2$. 
 By our assumption $\angle_{x}(\eta_i,\xi)<\pi$, 
 we have a non-trivial flat sector
 $\sect[x]{\xi}{\eta_2}$. 
 Since $[p,q_2]$ is a geodesic segment in 
 $\sect[x]{\xi}{\eta_2}$ joining points in
 $\ray{x,\xi}$ and $\ray{x,\eta_2}$, by sliding it by
 $\trans{t}{\xi}$ with $t>0$, we see that 
 $\trans{t}{\xi}([p,q_2])$ can be extended to 
 a geodesic segment $[\trans{t}{\xi}(p), q_2']$, where 
 $q_2' \in \ray{x,\eta_2}$. 
 Now we can take $t$ so that $q_2'=\cc{x}{\eta_2}(s_1)$, and 
 take $p' \in \ray{x,\xi}$ so that 
 $p, \trans{t}{\xi}(p) \in [x,p']$. 
 By (1), we have a flat sector
 $\sect[q]{\eta_1}{\eta_2}$ for any $q \in \ray{x,\xi}$, and we can
 consider $\bigcup_{q \in [x,p']} \sect[q]{\eta_1}{\eta_2}$. 

 Suppose that $d_{\angle}(\eta_1,\eta_2)<\pi$. 
 Since $\sect[x]{\xi}{\eta}$ is a flat sector for any
 $\eta \in [\eta_1, \eta_2] \subset \partial \cone[x]{m}$, 
 $\bigcup_{q \in [x,p']}\ray{q,\eta}\subset \sect[x]{\xi}{\eta}$ is a 
 (possibly degenerated) half-infinite parallelogram, and hence 
 we can apply Lemma~\ref{lem:parallelogram_and_sliding} 
 to $\bigcup_{q \in [x,p']} \sect[q]{\eta_1}{\eta_2}$ by setting
 $p_1=x$, $p_2=p'$. 
 By our choice of $t$, we see that 
 $\trans{t}{\xi}(q_1), q_2'\in R_{s_1s_1}$, 
  where $R_{s_1s_1}$ is the convex hull of
 $[\cc{x}{\eta_1}(s_1),\cc{x}{\eta_2}(s_1)] \cup 
 [\cc{p'}{\eta_1}(s_1),\cc{p'}{\eta_2}(s_1)]$,
 and that 
 $\tri{\trans{t}{\xi}(p)}{\trans{t}{\xi}(q_1)}{q_2'}$ is
 flat by Lemma~\ref{lem:parallelogram_and_sliding} $(4)$. 
 Note that $\trans{t}{\xi}(\tri{p}{q_1}{q_2})$ is a subset of the
 convex hull of 
 $\tri{\trans{t}{\xi}(p)}{\trans{t}{\xi}(q_1)}{q_2'}$,
 and hence it is flat. 
 By Lemma~\ref{lem:tau_extends}, $\trans{t}{\xi}$ is an isometric
 translation, 
 we conclude that $\tri{p}{q_1}{q_2}$ is also flat. 

 While if $d_{\angle}(\eta_1,\eta_2)=\pi$, then, according to
 Remark~\ref{remark:when_angle_is_pi}, $\tri{p}{q_1}{q_2}$ is lying in
 a flat strip, and hence it is flat. 
 This completes the proof. 
\end{proof}

\subsection{Existence of a Euclidean simplex}
\label{sec:existence-simplex}

%
%
 Let $l$ be a natural number and consider $\cone[x]{l}$.
 Take any $q \in \cone[x]{l}$. 
 Then there exists $\eta \in \partial \cone[x]{l-1}$ such that
 $q \in \sect[x]{\eta}{\xxi{l}}$. 
 Then $q$ must lie on $\ray{\hat q,\eta}$ for some 
 $\hat q \in \ray{x,\xxi{l}}$; 
 we can write $q = \cc{\hat q}{\eta}(t)$. 
 If $q \in \cone[x]{l-1}$, we can take $\hat q=x$.
 If $q \not\in \cone[x]{l-1}$, then 
 take $p_j \in [x,x_j]$ so that $\uulim[l] p_j=\hat q$. 
 Then $\uulim[l] \cc{p_j}{\eta}(t)= \cc{\hat q}{\eta}(t)=q$. 
 Keep this notation and define a map on $\cone[x]{l}$ by 
\begin{equation*}
 \Phi_1 (q) = 
 \begin{cases}
  q & q \in \cone[x]{l-1}, \\
  \uulim[l+1] c_{p_j}^{\eta}(t) & q= \cone[x]{l}\setminus 
  \cone[x]{l-1}. 
 \end{cases}
\end{equation*}
 Then, since $\uulim[l+1]\cc{p_j}{\eta}(t)=\cc{q'}{\eta}(t)$ for 
 $q' = \uulim[l+1] p_j \in \ray{x,\xxi{l+1}}$, 
 we see that 
 $\Phi_1(\sect[x]{\eta}{\xxi{l}}) = \sect[x]{\eta}{\xxi{l+1}}$, 
 and obtain
 $\Phi_1 \colon \cone[x]{l}\rightarrow C_{x}(1,\dots,l-1,l+1)$. 
 Take $\q{1}, \q{2} \in \cone[x]{l}$ and write them as
 $\q{i} = \cc{\hat q_i}{\eta^i}(t_i)$, where
 $\hat q^{(i)} \in \ray{x,\xxi{l}}$, 
 $\eta^i \in \partial \cone[x]{l-1}$ and 
 $t_i \in \ray{0,\infty}$ as above. 
 Take $\p{i}_j \in [x,x_j]$ so that 
 $\uulim[l] p_j^{(i)}=\hat q_i$. 
 Then we have 
 \begin{equation*}
   d(\q{1},\q{2})=
  \ulim_j d(\cc{\p{1}_j}{\eta^1}(t_1), \cc{\p{2}_j}{\eta^2}(t_2)), 
 \end{equation*}
 and we also have
\begin{equation*}
   d(\Phi_1(\q{1}), \Phi_1(\q{2}))
  =\ulim_j d(\cc{\p{1}_j}{\eta^1}(t_1), \cc{\p{2}_j}{\eta^2}(t_2)), 
\end{equation*}
 since 
 $\Phi_1(\q{i})=\uulim[l+1]_j \cc{\p{i}_j}{\eta^i}(t_i)$. 
 Therefore $\Phi_1$ is an isometry. 
 We extend this map to an isometry defined on whole $\cone[x]{l+1}$ as
 follows. 

 Note that, for any $\eta \in \cone[x]{l}$, 
 $\sect[x]{\eta}{\xxi{l+1}}$ can be identified with a flat
 sector in $\R^2$. 
 Take any point $q \in \sect[x]{\xxi{l+1}}{\eta}$. 
 Through this identification, we can write 
\begin{equation*}
 q=a p + b \p{l+1}= \cc{x}{\eta}(a) + \cc{x}{(l+1)}(b), 
\end{equation*}
 where $p=\cc{x}{\eta}(1)$, $\p{l+1}=\cc{x}{(l+1)}(1)$. 
 Note that this description of $q$ is unique,
 since $0<\angle(\eta,\xxi{l+1})<\pi$ for any 
 $\eta \in \partial \cone[x]{l}$ by
 Proposition~\ref{prop:less_than_pi} $(1)$.
 Note also that, by Lemma~\ref{lem:sector_exists} with $m=l+1$, the
 convex hull of $\ray{x,\xi^{(l)}} \cup \Phi_1(\ray{x,\eta})$ is a
 flat sector, which we denote by $\sect[x]{\eta'}{\xi^{(l)}}$. 
 By Proposition~\ref{prop:less_than_pi} $(2)$,  for each 
 $q \in \cone[x]{l+1}\setminus \left(\cone[x]{l}\cup
 \ray{x,\xxi{l+1}}\right)$, 
 we have unique $\eta \in \cone[x]{l}$ such that 
 $q \in \sect[x]{\eta}{\xxi{l+1}}$. 
 Therefore setting 
\begin{equation*}
 \Phi(q)= \Phi(ap + b \p{l+1})
  = a \Phi_1(p) + b \p{l} 
\end{equation*}
 gives us a well-defined map on $\cone[x]{l+1}$, 
 and we see that $\Phi$ maps 
 $\sect[x]{\eta}{\xxi{l+1}}$ onto $\sect[x]{\eta'}{\xxi{l}}$. 
 In this manner, we get a map defined on whole 
 $\cone[x]{l+1}=C_{x}(1,\dots,l+1)=
 \bigcup_{\eta \in \partial \cone[x]{l}}
 \sect[x]{\eta}{\xxi{l+1}}$. 

 We will show that $\Phi$ is actually an isometry. 
 Since 
 $\angle(\xxi{i},\xxi{j})=\zure$ and 
 $\Phi(\xxi{i})=\xxi{j}$ for any $i$
 ($j=i$ if $i \in \{1,\dots, l-1\}$, and $(i,j)=(l,l+1), (l+1,l)$
 otherwise), 
 $\Phi$ is an isometry on each 
 $\sect[x]{\xxi{i}}{\xxi{l+1}}$. 
 By applying Lemma~\ref{lem:parallelogram_and_sliding} $(4)$ to
 $\bigcup_{q \in [x,\p{l+1}]} \sect[q]{\xxi{i}}{\xxi{j}}$, 
 $i,j \in \{1,\dots, l\}$, 
 we see that 
 $\tri{\p{l+1}}{\p{i}}{\p{j}}$ is a flat triangle, where
 $\p{k}=\cc{x}{(k)}(1)$.  
 The same is true for $\tri{\p{l}}{\Phi_1(\p{i})}{\Phi_1(\p{j})}$, and
 these two triangles are regular triangles whose side length is equal to
 $2\sin (\zure/2)$, hence they are isometric to each other. 
 Let $p_t=(1-t)\p{i}+ t\p{j}\in [\p{i},\p{j}]$ for $t \in [0,1]$. 
 Since $p_t \in \cone[x]{l}$, $\Phi(p_t)=\Phi_1(p_t)$. 
 Note that $\Phi_1$ is an isometry with 
 $\Phi_1(\ray{x,\xxi{i}})=\ray{x,\xxi{\tilde i}}$, where
 $\tilde i=i$ if $1\leq i \leq l-1$, and $\tilde i=l+1$ if $i=l$. 
 Then we see that
 $\Phi(p_t)=(1-t)\p{\tilde i}+t \p{\tilde j}$ holds for any
 $i,j \in \{1,\dots, l\}$. 
 Since $\tri{\p{l+1}}{\p{i}}{\p{j}}$ and 
 $\tri{\p{l}}{\p{\tilde i}}{\p{\tilde j}}$ are isometric to each
 other,  we see that $d(p_t, \p{l+1})=d(\Phi(p_t),\p{l})$.  
 Also since we know that $\tri{x}{\p{i}}{\p{j}}$ and 
 $\tri{x}{\p{\tilde i}}{\p{\tilde j}}$ are flat and isometric
 to each other, $d(x, p_t)=d(x,\Phi(p_t))$. 
 Since
 $\tri{\p{l+1}}{x}{p_t}$ is lying in $\sect[x]{\xi(t)}{\xxi{l+1}}$
 for $\xi(t) \in \partial \cone[x]{l}$ satisfying 
 $p_t \in \ray{x,\xi(t)}$, it is obviously flat. 
 The same is true for $\tri{\p{l}}{x}{\Phi(p_t)}$ lying in 
 $\sect[x]{\tilde\xi(t)}{\xi^{l}}$, where
 $\tilde\xi(t) \in \partial \cone[x]{l+1}$ satisfies
 $\Phi(p_t) \in \ray{x,\tilde\xi(t)}$. 

\noindent
\begin{minipage}[c]{8cm}
 Since $d(p_t, \p{l+1})=d(\Phi(p_t), \p{l})$ and 
 $d(x, p_t)=d(x,\Phi(p_t))$ hold as we have seen above, 
 together with $d(x,\p{l+1})=d(x,\p{l})=1$, we see
 that $\tri{\p{l+1}}{x}{p_t}$ and $\tri{\p{l}}{x}{\Phi(p_t)}$
 are isometric to each other, and in particular, 
 $\angle_{x}(p_t,\p{l+1})=\angle_{x}(\Phi(p_t),\p{l})$.
 Recalling the defintion of $\Phi$, we conclude that 
 $\Phi \colon \sect[x]{\xi(t)}{\xxi{l+1}} \rightarrow
 \sect[x]{\tilde\xi(t)}{\xi^{(l)}}$
is an isometry. 
\end{minipage}
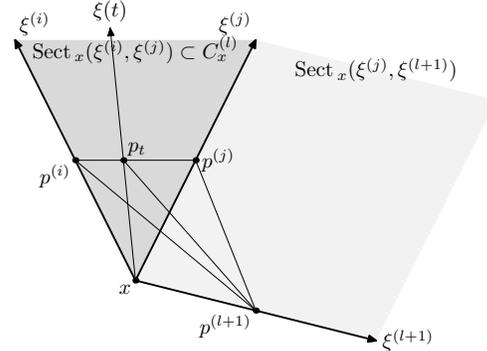
\begin{minipage}[c]{6cm}
 \vspace{-2cm}
 \hspace{1.7cm}
 \includegraphics[scale=.8]{fig.10}
\end{minipage}

 Now fix $t \in (0,1)$ and let $\p{t}=\cc{x}{\xi(t)}(1)$. 
 By applying Lemma~\ref{lem:parallelogram_and_sliding} $(4)$ to
 $\bigcup_{q \in [x,\cc{x}{\xi(t)}(T)]}\sect[q]{\xxi{k}}{\xxi{l+1}}$
 for $k \in \{1,\dots, l\}$ and suitable $T>0$, 
 we see that $\tri{\p{l+1}}{\p{t}}{\p{k}}$ is a flat triangle, and the
 same is true for $\tri{\p{l}}{\Phi(\p{t})}{\Phi(\p{k})}$. 
 Since we already know that
\begin{equation*}
\begin{split}
 & d(\p{t},\p{l+1})=d(\Phi(\p{t}),\p{l}), \quad
 d(\p{l+1},\p{k})=d(\p{l},\Phi(\p{k})), \\
 & d(\p{t},\p{k})=d(\Phi_1(\p{t}),\Phi_1(\p{k}))
 =d(\Phi(\p{t}),\Phi(\p{k}))
\end{split}
\end{equation*} 
hold,  by the same argument as above, we see that $\Phi$ is an isometry on 
 $\sect[x]{\xi'}{\xxi{l+1}}$ for any 
 $\xi' \in [\xi(t),\xxi{k}] \subset \partial\cone[x]{l}$ with 
 $k \not=i,j$ and $t \in [0,1]$. 
 Repeating this shows that
 $\Phi$ is an isometry on $\sect[x]{\eta}{\xxi{l+1}}$ for
 any $\eta \in \cone[x]{l}$.  

 Now take any $q_1,q_2 \in \cone[x]{l+1}$. 
 If there exists $\eta \in \partial \cone[x]{l}$ such that 
 $q_1,q_2 \in \sect[x]{\eta}{\xxi{l+1}}$, then we already know
 $d(q_1,q_2)=d(\Phi(q_1),\Phi(q_2))$. 
 Suppose there is no such $\eta$.
 Then we have $\eta_i \in \partial \cone[x]{l}$ satisfying 
 $q_i \in \sect[x]{\eta_i}{\xxi{l+1}}$.  
  By Proposition~\ref{prop:less_than_pi} $(1)$,
 we can find $p \in \ray{x,\xxi{l+1}}$ and
 $q_i' \in \ray{x,\eta_i}$ so that $q_i \in [p,q_i']$, 
 and we see that 
 Proposition~\ref{prop:flat_triangle} $(2)$ implies that 
 $\tri{p}{q_1'}{q_2'}$ is flat. 
 Note that 
 $[p,q_i'] \subset \sect[x]{\eta_i}{\xxi{l+1}}$ and  
 $[q_1',q_2']\subset \cone[x]{l}$. 
 Therefore we have $d(p,q_i')=d(\Phi(p),\Phi(q_i'))$, $i=1,2$, and 
 $d(q_1',q_2')=d(\Phi(q_1'),\Phi(q_2'))$. 
 Since $\Phi$ maps $\sect[x]{\eta_i}{\xxi{l+1}}$ isometrically onto its
 image,  we know that 
 $0<\angle_{x}(\Phi(q_i'),\Phi(p))=\angle_{x}(\eta_i,\xxi{l+1})<\pi$
 by Proposition~\ref{prop:less_than_pi} $(1)$.
 Thus, by applying Proposition~\ref{prop:flat_triangle} $(2)$ to 
 $\Phi(p) \in \ray{x,\xxi{l}}$ and  
 $\Phi(q_i') \in \Phi(\ray{x,\eta_i})$, $i=1,2$, 
 we see that 
 $\tri{\Phi(p)}{\Phi(q_1')}{\Phi(q_2')}$ is also flat.
 Therefore we conclude that $\Phi$ maps $\tri{p}{q_1'}{q_2'}$
 isometrically onto $\tri{\Phi(p)}{\Phi(q_1')}{\Phi(q_2')}$, and
 in particular, $d(\Phi(q_1),\Phi(q_2))=d(q_1,q_2)$. 

 Repeating this procedure allows us to extend $\Phi$ to an isometry
 defined on $\cone[x]{m}$ for $m>l+1$; we get an isometry on
 $\cone[x]{m}$ permuting $\xi^{(l)}$ and $\xxi{l+1}$.  
 Composing such maps suitably,  we obtain the following proposition.

\begin{Proposition}
\label{prop:phi_is_a_isometry}
 Let $s$ be a permutation of $\{1, \dots, m\}$. 
 Then there exists an isometry
 $\Phi_{s} \colon \cone[x]{m}\rightarrow \cone[x]{m}$ which
 satisfies $\Phi_{s}(\ray{x,\xxi{i}})=\ray{x,\xxi{s(i)}}$. 
\end{Proposition}

%
%
\begin{Proposition}
 \label{prop:simplex}
  Fix $r>0$, and 
  define $\sigma^{(j)}$ inductively by
\begin{equation*}
 \sigma(i)=\{\p{i}\}, \quad
 \sigma(i,\dots, j,k)=\bigcup_{p \in \sigma(i,\dots,j)}[p,p^{(k)}], 
 \quad   \sigma^{(j)}= \sigma(1, \dots, j), 
\end{equation*}
 where $\p{j}=\cc{x}{(j)}(r)$.
 Then, under our setting, for any $n \in \N$, $\sigma^{(n-1)}$ is
 isometric to a regular Euclidean $(n-1)$-simplex. 
\end{Proposition}

\begin{proof}
 First note that $\sigma^{(n-1)}$ is a slice of $C_{x}^{(n)}$
 containing $\{p^{(1)},\dots, p^{(n)}\}$ and left invariant by an
 isometry $\Phi_{s}$ in Proposition~\ref{prop:phi_is_a_isometry} for any
 permutation $s$. In particular, $[\p{i},\p{j}]$ has the same length
 $2r\sin (\zure/2)$ for $i,j=1, \dots, n$ with $i\not= j$. 
 Let $S(1, \dots, n)$ be a Euclidean regular
 $(n-1)$-simplex with vertices $\bar p_1, \dots, \bar p_n$ and
 $d(p_i,p_j)=2r\sin (\zure/2)$ for $i\not= j$. 
 We also denote, for $\{i_1, \dots, i_m\}\subset \{1,\dots, n\}$, 
 $(m-1)$-face with vertices $\{\bar p_{i_1}, \dots, \bar p_{i_m}\}$ by
 $S(i_1, \dots, i_m)$. 
 
 We show that there is an isometry from 
 $S(1, \dots, n)$ onto
 $\sigma^{(n-1)}$ by induction on $n$. 
 Note that, for any point $\bar q$ in $S(1,\dots, m-1)$, $m=1,\dots, n$, 
 there exists a unique segment 
 $[\bar q,\bar p_m] \subset S(1,\dots, m)$, 
 and we know that
 \begin{equation*}
  S(1, \dots, m)=\bigcup_{\bar q \in S(1, \dots, m-1)} [\bar q, \bar p_m]. 
 \end{equation*}
 Furthermore,  if 
 $\bar q\not= \bar q' \in S(1, \dots, m-1)$, then 
 $[\bar q,\bar p_m]\cup [\bar q',\bar p_m]=\{\bar p_m\}$. 
 Keeping this in mind,  we can define a surjection 
 $\Psi\colon S(1,\dots,n)\rightarrow \sigma^{(n)}$ inductively by
\begin{equation*}
\begin{cases}
  \Psi(\bar p_1)=p^{(1)}, & \ \\
 \Psi(\bar c_{\bar q}^{\bar p_j}(t))=c_{\Psi(\bar q)}^{\p{j}}(dt), 
  & \bar q \in S(1,\dots, j-1), \ j=2,\dots, n, 
\end{cases}
\end{equation*}
 where $\bar c_{\bar q}^{\bar p_j}$ 
 is a geodesic joining $\bar q$ and $\bar p_j$ in $S(1,\dots,j)$, 
 and $d=d(\Psi(\bar q),\p{j})/d(\bar q, \bar p_j)$.  

 It is obvious that $\Psi$ is an isometry on $S(1)$. 
 Suppose that $\Psi\colon S(1,\dots,n-1) \rightarrow \sigma^{(n-2)}$ is
 an isometry. 
 Take $\bar q_1, \bar q_2 \in S(1,\dots, n)$.  
 We will show that $\Psi$ maps 
 $\tri{\bar q_1}{\bar q_2}{\bar p_n}$ isometrically
 into $\sigma^{(n-1)}$. 
 Note that, for $\bar q_i \in S(1, \dots, n)$, there exists 
 $\bar q_i' \in S(1, \dots, n-1)$ such that 
 $\bar q_i \in [\bar q_i', \bar p_n]$. (If $\bar q_i=\bar p_n$, then
 take any $\bar q_i' \in S(1,\dots, n-1)$.)
 Furthermore, $[\bar q_1', \bar q_2']$ can be extended to 
 $[\bar q_1'', \bar q_2'']$ so that 
 $\bar q_i'' \in S(1, \dots, \hat l(i), \dots, n-1)$ for some 
 $l(i) \in \{1,\dots,n-1\}$. 
 Since $S(1, \dots, n)$ is a Euclidean regular
 $(n-1)$-simplex and $\bar q_i'' \in S(1,\dots, \hat l(i),\dots, n-1)$,
 we know that $d(\bar q_i'', \bar p_{l(i)})=d(\bar q_i'', \bar p_n)$. 
 By definition, we see that 
 $\Psi(\bar q_i'')\in \sigma(1,\dots, \hat l(i), \dots, p^{(n-1)})$. 
 By Proposition~\ref{prop:phi_is_a_isometry}, we have an isometry
 $\Phi_s$ which fixes 
 $\sigma(1,\dots, \hat l(i), \dots, n-1)$ and switches 
 $p^{(l(i))}$ and $\p{n}$. 
 Thus we see that 
 $d(\Psi(\bar q_i''), \p{n})=d(\Psi(\bar q_i''),\p{l(i)})$.  
 Since $\Psi$ restricted to $S(1, \dots, n-1)$ is an
 isometry, noting that $\Psi(\bar p_{l(i)})=p^{(l(i))}$ and 
 $\Psi(\bar p_n)=p^{(n)}$,  we see that 
\begin{equation*}
  d(\bar q_i'', \bar p_n)= d(\bar q_i'',  \bar p_{l(i)})
  = d(\Psi(\bar q_i''),\Psi(\bar p_{l(i)}))=
  d(\Psi(\bar q_i''),\Psi(\bar p_n)), 
\end{equation*}  
 as well as
 $d(\bar q_1'', \bar q_2'')=d(\Psi(\bar q_1''),\Psi(\bar q_2''))$.
 This shows that $\Psi$ maps $[\bar q_1'', \bar q_2'']$ and 
 $[\bar q_i'', \bar p_n]$, $i=1,2$, isometrically into
 $\sigma^{(n-1)}$. 
 Furthermore, by Proposition~\ref{prop:flat_triangle} $(2)$, we know
 that $\tri{\Psi(\bar q_1'')}{\Psi(\bar q_2'')}{p^{(n)}}$ is flat. 
 By the definition of $\Psi$,  we conlude that $\Psi$ maps the convex
 hull of $\tri{\bar q_1''}{\bar q_2''}{\bar p_n}$ isometrically onto
 that of $\tri{\Psi(\bar q_1'')}{\Psi(\bar q_2'')}{p^{(n)}}$. 
 Since $\tri{\bar q_1}{\bar q_2}{\bar p_n}$ is
 contained in the convex hull of 
 $\tri{\bar q_1''}{\bar q_2''}{\bar p_n}$, 
 we see that
 $\tri{\bar q_1}{\bar q_2}{\bar p_n}$ is also isometrically mapped onto
 $\tri{\Psi(\bar q_1)}{\Psi(\bar q_2)}{p^{(n)}}$.
 In particular, 
 $d(\bar q_1, \bar q_2)=d(\Psi(\bar q_1), \Psi(\bar q_2))$, that is, 
 $\Psi$ is an isometry. This completes the proof. 
\end{proof}

\subsection{Proof of Theorem~\ref{thm:compactness}}
\label{sec:the_proof}

Now we give the proof of Theorem~\ref{thm:compactness}. 

\noindent
{\it Proof of Theorem~\ref{thm:compactness}.}
 We first show that $\ulim_j \angle_{x_0}(x_j,\xxi{1})=\zure=0$, where
 $x_0=f(e)$, by deriving a contradiction under the assumption that
 $\zure>0$.  
 
 Suppose first that $Y$ has locally finite dimension. 
 Then, for given $R>0$, there exists $n \in \N$ such that
 any $B \subset B(x_0,R)$ satisfies (\ref{eq:geom_dim}). 
 On the other hand, by taking ultralimits $m$-times, we see that
 $B(x_0, R)\cap \cone[x_0]{m}\subset Y^{(m)}$ contains an arbtrarily small
 Euclidean regular 
 $(m-1)$-simplex $\sigma^{(m-1)}$ by Proposition~\ref{prop:simplex}. 
 Note that $\sigma^{(m-1)}$ (more precisely the set of vertices
 $\{p^{(1)}, \dots, p^{(m)}\}$) satisfies
\begin{equation*}
 \rad{Y^{(m)}}{\sigma^{(m-1)}} = 
 \sqrt{\frac{m-1}{2m}}\diam{\sigma^{(m-1)}}. 
\end{equation*}
 If we take $m>n+1$, then we have
\begin{equation*}
 \sqrt{\frac{n}{2(n+1)}} < \sqrt{\frac{m-1}{2m}}, 
\end{equation*}
 and this contradicts Proposition~\ref{prop:dimension_of_ulim}. 

 Now suppose $Y$ has telescopic dimension at most $n$. 
 Let $m>n+1$ and take $\delta$ so that
\begin{equation*}
 \delta + \sqrt{\frac{n}{2(n+1)}} < \sqrt{\frac{m}{2(m+1)}}.
\end{equation*}
 By definition, there exists $D>0$ such that any $B \subset Y$ with
 $\diam{B}>D$ satisfies (\ref{eq:tele_dim}). 
 However, taking $\sigma^{(m-1)} \subset \cone[x_0]{m}$ so that
 $\diam{\sigma^{(m-1)}}>D$, we see that $\sigma^{(m-1)}$
 does not satisfy (\ref{eq:tele_dim}) as we have seen above. 
 This contradicts Proposition~\ref{prop:dimension_of_ulim}. 
 Therefore, in either case, we conclude that $\zure=0$. 

 Let $\xi=\xxi{1} = \ulim x_j \in \partial Y^{(1)}$. 
 We show that $\{x_j\}$ actually converges to $\xi$ in the cone
 topology. 
 By our assumption, by setting 
\begin{equation*}
 A'_M=\{j \in \N \mid d(x_0,x_j)\geq M\}
\end{equation*}
we have $\lambda(A'_M)=1$ for any $M>0$. 
Also by setting 
\begin{equation*}
 B'_{\varepsilon}=\{j \in \N \mid
\angle_{x_0}(x_j,\xi)<\varepsilon\}, 
\end{equation*}
we see that 
$\lambda(B'_{\varepsilon})=1$ for any $\varepsilon>0$, 
since we know $\ulim_j\angle_{x_0}(x_j,\xi)= 0$. 
Therefore, for any $M>0$ and $\varepsilon >0$, 
we have $\lambda (A'_M \cap B'_{\varepsilon})=1$ and hence 
we can find $j \in \N$ such that 
$d(x_0,x_j)\geq M$ and
$\angle_{x_0}(x_j,\xi)< \varepsilon$. 
Note that, by Lemma~\ref{lem:ultralimit_of_buseman_fct} and our assumption, 
we have $\int_{\Gamma}b_{\xi}(f(\gamma),f(e)) d\mu(\gamma)=0$. 
Then by noting that $\{x_j\}\subset f(\Gamma)$, 
we can apply Lemma~\ref{lem:core} $(2)$,
and see that the convex hull of 
$\ray{x_0,\xi} \cup [x_0,x_j]\cup \ray{x_j,\xi}$
is a half-infinite parallelogram.  
Hence, for any $t \in [0,d(x_0,x_j)]$ and a geodesic $c_j$ with
$c_j(0)=x_0$ and $c_j(d(x_0,x_j))=x_j$, we have
$d(c_j(t), \ray{x_0,\xi}) = t\sin \angle_{x_0}(x_j,\xi)$ 
as long as $\angle_{x_0}(x_j,\xi) \leq \pi/2$. 
Therefore, by letting $M\to \infty$ and $\varepsilon \to 0$ suitably, 
we can choose a sequence $\{j'\}\subset \N$ so that $\{c_{j'}\}$
 converges to a geodesic ray joining $x_0$ and $\xi$ on each
 bounded interval in $\ray{0,\infty}$. 
In particular, $\{x_{j'}\}$ converges to $\xi$ in the 
cone topology of $Y_{\infty} \cup \partial Y_{\infty}$. 
Since $\{x_{j'}\}\subset Y$, we see that $\xi$ must lie in
$\partial Y$. This completes the proof. 
\qed

\section{An upper bound of 
$l_{\rho}(\Gamma, \sum_{n=1}^{\infty}a_n\mu^n)$}
\label{sec:appendix}

Let $\Gamma$ be a countable group equipped with a
symmetric probability measure $\mu$ whose support generates $\Gamma$. 
Suppose that $\Gamma$ acts on a $\cat{0}$ space $Y$ via a homomorphism
$\rho\colon \Gamma \rightarrow Y$. 
Suppose further that there exists a $\rho$-equivariant $\mu$-harmonic
map $f\colon \Gamma \rightarrow Y$. 
Recall that if $f$ is a $\mu$-harmonic map, then 
$\gamma \mapsto d(f(e),f(\gamma))$ is a $\mu$-subharmonic function
on $\Gamma$ (Proposition~\ref{prop:pullback-is-subharmonic}). 

As in \S~\ref{sec:rate_of_escape_and_hmap}, we denote by $\mu^n$ the
$n$-fold convolution of $\mu$; $\mu^n$ describes the distribution of the
position of a random walk generated by $\mu$ at time $n$.  
Consider a convex combination of $\mu^n$, that is, 
$\sum_{n=1}^{\infty} a_n\mu^n$, where $a_n\geq 0$ and
$\sum_{n=1}^{\infty}a_n=1$. In what follows,  we prove that
\begin{equation}
\label{eq:conv_comb}
 l_{\rho}\left(\Gamma, \sum_{n=1}^{\infty}a_n\mu^n \right)
 \leq \left(1+\sum_{n=1}^{\infty} n a_n\right) l_{\rho}(\Gamma,\mu). 
\end{equation}
Although it is weaker than the claim of \cite[Theorem 4.5]{forghani}, 
under the assumption $l_{\rho}(\Gamma,\mu)=0$,
this implies $l_{\rho}(\Gamma,\sum_{n=1}^{\infty}a_n\mu^n)=0$ 
which is sufficient for our purpose. 
The proof follows the line suggested in \cite{forghani}, except 
that we use the existence of a $\mu$-harmonic map $f$ at some point. 

Let $k \in \Z_{>0}$, and note that 
$\left(\sum_{n=1}^{\infty}a_n\mu^n\right)^k$ is again a convex
combination of $\mu^j$'s; we can express as 
$\left(\sum_{i=n}^{\infty}a_n\mu^n\right)^k=\sum_{j=1}^{\infty}a_j'\mu^j$. 
Note that, since $\sum_{n=1}^{\infty} a_n =1$, 
\begin{equation*}
\begin{split}
  \sum_{j=1}^{\infty} ja_j' 
 & =\sum_{n_1, \dots, n_k} (n_1 + \dots + n_k) a_{n_1}\dots a_{n_k} \\
 & = \sum_{n_2,\dots, n_k}\left(
     \sum_{n_1}  (n_1 + \dots + n_k) a_{n_1}\right)a_{n_2} \dots a_{nk} \\
 & = \sum_{n_2,\dots, n_k} \left(\left(\sum_{n_1}n_1 a_{n_1}\right)+ 
     n_2 +  \dots + n_k\right) a_{n_2} \dots  a_{n_k} \\
 & = \sum_{i_1} n_1 a_{n_1} + \dots + \sum_{n_k}n_k a_{n_k} = k \sum_n n a_n
\end{split}
\end{equation*}
holds. 
Now set $b_i = \sum_{j=k(i-1)+1}^{ki} a_j'$ ($i\in \Z_{>0}$) and 
\begin{equation*}
 (\tilde \Omega^k,\mathbb{P}) = 
  (\Gamma \times \Z_{>0}, \mu^k \times \sum b_i
  \delta_i) \times 
 (\Gamma \times \Z_{>0}, \mu^k \times \sum b_i  \delta_i) 
 \times \cdots , 
\end{equation*}
where $\delta_i$ is a Dirac measure supported on $i \in \Z_{>0}$. 
An element $\tilde \omega$ in $\tilde \Omega^k$ can be written as 
$\tilde \omega = \left((\omega_1, i_1), (\omega_2,i_2), \dots,
(\omega_n,i_n), \dots \right)$. 
Let $\mathcal{F}_n$ be a $\sigma$-algebra of subsets of
$\tilde\Omega^k$ generated by subsets of the form
\begin{equation*}
 \{(\omega_1,i_1)\} \times \dots \times \{(\omega_n,i_n)\} \times 
  (\Gamma \times \Z_{>0}) \times \dots \times (\Gamma \times \Z_{>0})
  \times \dots .
\end{equation*}
 Then $\mathcal{F}_n\subset \mathcal{F}_{n+1}$ holds and
 $\{\mathcal{F}_n\}_{n \in \Z_{>0}}$ forms a {\it filtration}. 
 We note that a function 
 $L\colon \tilde \Omega^k\rightarrow \R$ is
 $\mathcal{F}_n$-measurable means that $L$ is determined by the
 information up to the $n$th component of 
 $\tilde\omega=((\omega_1,i_1), (\omega_2,i_2),\dots)$.
 Define a function $\tau \colon \tilde \Omega^k \rightarrow \Z_{>0}$ by
 $\tau(\tilde \omega)=i_1$.  
 Then $\tau$ is a {\it stopping time} (or {\it Markov time}); that is, 
 $\tau$ satisfies, for any $m \in \Z_{>0}$, 
\begin{equation*}
  \{\tilde\omega\in \tilde\Omega^k \mid 
    \tau(\tilde\omega) \leq m\} \in \mathcal{F}_m, 
\end{equation*}
 since the left-hand side belongs to $\mathcal{F}_1$ by the definition
 of $\tau$. 

 Let $L_n \colon \tilde\Omega^k \rightarrow \R$ be a 
 $\mathcal{F}_n$-measurable function and, for $n\geq m$, denote 
 by $\mathbb{E}(L_n \mid \mathcal{F}_{m})$ the conditional
 expectation of $L_n$ with respect to $\mathcal{F}_{m}$, namely 
 $\mathbb{E}(L_n\mid \mathcal{F}_{m})$ is 
 a $\mathcal{F}_{m}$-measurable function written as 
\begin{equation*}
 \mathbb{E}(L_n \mid \mathcal{F}_{m})\colon \tilde \Omega^k  
  \rightarrow \R ;
  \ \mathbb{E}(L_n \mid \mathcal{F}_{m})(\tilde\omega)
   = \int_{A_{m}^n(\tilde\omega)} L_n (\tilde\omega') 
      d\mathbb{P}(\tilde\omega'), 
\end{equation*}
where
\begin{equation*}
 A_{m}^n(\tilde\omega)=\{\tilde\omega' \in \tilde \Omega^k \mid 
 \tilde \pi_i(\tilde\omega')=\tilde \pi_i(\tilde\omega),\ i=1,\dots, m\}, 
\end{equation*}
 and $\tilde\pi_i\colon\tilde\Omega^k\rightarrow\Gamma\times\Z_{>0}$
 is the projection onto the $i$th factor.  
 Note that $A_{m}^n(\tilde\omega)\in \mathcal{F}_n$. 
 We set $L_n$ to be a function defined by
\begin{equation*}
 L_n(\tilde\omega) = n \mathbb{E}(d(f(e),f(\gamma_{1}^k(\tilde\omega))))
  - d(f(e),f(\gamma_n^k(\tilde\omega))), 
\end{equation*}
 where $\gamma_n^k(\tilde\omega)=\omega_1 \dots \omega_n \in \Gamma$. 
 Then we see that, by the triangle inequality and the
 $\rho$-equivariance of $f$, 
\begin{equation*}
\begin{split}
 & \mathbb{E}(L_{n+1}\mid \mathcal{F}_n)(\tilde\omega) \\
 =& (n+1)\mathbb{E}(d(f(e),f(\gamma_1^k(\tilde\omega)))) 
 - \int_{A_n^{n+1}(\tilde\omega)} 
    d(f(e),f(\gamma_{n+1}^k(\tilde\omega')))
   d\mathbb{P}(\tilde\omega') \\
 =& (n+1)\int_{\Gamma} d(f(e),f(\gamma)) d\mu^k(\gamma)
 - \int_{\Gamma} d(f(e),f(\gamma_n^k(\tilde\omega)\gamma))
 d\mu^k(\gamma) \\
 \geq &  (n+1)\int_{\Gamma}d(f(e),f(\gamma)) d\mu^k(\gamma) 
   -  d(f(e),f(\gamma_n^k(\tilde\omega)))  \\
& \phantom{==}
  -  \int_{\Gamma}d(f(\gamma_n^k(\tilde\omega)), 
          f(\gamma_n^k(\tilde\omega) \gamma)) d\mu^k(\gamma) \\
 =& n \mathbb{E}(d(f(e),f(\gamma_1^k(\tilde\omega)))) 
     -d(f(e),f(\gamma_n^k(\tilde\omega))) = L_n(\tilde\omega)
\end{split}
\end{equation*} 
holds; that is, $L_n$ is a {\it submartingale}. 
 Note that, by our definition of $(\tilde \Omega^k,\mathbb{P})$, 
 the expectation of 
 $\tilde\omega\mapsto d(f(e),\gamma_n^k(\tilde\omega))$ can be expressed
 as 
\begin{equation*}
\begin{split}
  \mathbb{E}(d(f(e),f(\gamma_n^k(\omega)))) & = 
 \int_{\tilde \Omega^k} d(f(e),f(\gamma_n^k(\omega))) d\mathbb{P}(\omega)
  = \int_{\Gamma} d(f(e),f(\gamma)) d\mu^{nk}(\gamma) \\
  & = L^{nk}(f,\mu)=L^{nk}(f), 
\end{split}
\end{equation*}
 where $L^{nk}(f)$ is defined in \S~\ref{sec:rate_of_escape_and_hmap}. 

 Noting that $\tau\wedge n(\tilde\omega)=\min \{\tau(\tilde\omega),n\}$
 also gives a stopping time and that $\tau\wedge n$ is bounded, 
 by Doob's optional stopping theorem, we get 
 $\mathbb{E}(L_{\tau\wedge n}\mid \mathcal{F}_1)(\tilde\omega)\geq
 L_1(\tilde\omega)$, 
 where
 \begin{equation*}
  L_{\tau\wedge n}(\tilde\omega)=(\tau\wedge n)(\tilde\omega) 
 \mathbb{E}(d(f(e),f(\gamma_1^k(\tilde\omega))))
  - d(f(e),f(\gamma_{(\tau\wedge n)(\tilde\omega)}^k(\tilde\omega))). 
 \end{equation*}
 Since $\mathbb{E}(L_1)=0$, by integrating the both sides over 
 $\tilde \Omega^k$, we obtain
\begin{equation*}
 \mathbb{E}(\tau\wedge n) 
  \mathbb{E}(d(f(e),f(\gamma_1^k(\cdot)))) 
 \geq \mathbb{E}(d(f(e),
    f(\gamma_{(\tau\wedge n)(\cdot)}^k(\cdot)))).
\end{equation*}
 Since $\tau \geq \tau \wedge n$, we get
\begin{equation*}
 \mathbb{E}(\tau)
  \mathbb{E}(d(f(e),f(\gamma_1^k(\cdot))))  \geq 
 \mathbb{E}(d(f(e),
   f(\gamma_{(\tau\wedge n)(\cdot)}^k(\cdot)))). 
\end{equation*} 
 Since 
 $d(f(e),f(\gamma_{(\tau\wedge n)(\tilde\omega)}(\tilde\omega)))\geq 0$
 and 
 $\{d(f(e),f(\gamma_{(\tau\wedge n)(\tilde\omega)}^k
 (\tilde\omega)))\}_{n\in \Z_{>0}}$
 converges to $d(f(e),f(\gamma_{\tau(\tilde\omega)}^k(\tilde\omega)))$  
 as $n \to \infty$, Fatou's Lemma implies that 
 \begin{equation}
 \label{eq:app-1}
 \begin{split}
   \mathbb{E}(\tau) L^k(f,\mu)
   & \geq \int_{\tilde \Omega^k} 
    d(f(e),f(\gamma_{\tau(\tilde\omega)}^k(\tilde\omega)))
   d\mathbb{P}(\tilde\omega) \\
   & = \sum_{i=1}^{\infty}b_i \left(
     \int_{\Gamma} d(f(e),f(\gamma)) d\mu^{ik}(\gamma)\right)
 \end{split} 
 \end{equation}
 by our definition of $(\tilde \Omega^k, \mathbb{P})$. 
 Since $\gamma \mapsto d(f(e),f(\gamma))$ is $\mu$-subharmonic, 
 we have 
 $\int_{\Gamma} d(f(e),f(\gamma'\gamma)) d\mu(\gamma) 
 \geq d(f(e),f(\gamma'))$, 
 and hence we get, for $m\geq l$, 
\begin{equation}
\label{eq:app-2}
 \int_{\Gamma} d(f(e),f(\gamma)) d\mu^m(\gamma)
  \geq \int_{\Gamma} d(f(e),f(\gamma)) d\mu^l(\gamma). 
\end{equation}
 On the other hand, recalling the definition of $b_i$, we get 
\begin{equation}
\label{eq:app-3}
 \mathbb{E}(\tau) = \sum_{n=1}^{\infty}i b_i 
  = \sum_{j=1}^{\infty} \left\lceil{\frac{j}{k}}\right\rceil a_j'
   \leq \frac{1}{k} \left(k+ \sum_{j=1}^{\infty}j  a_j' \right)
   = 1+  \sum_{n=1}^{\infty} n a_n, 
\end{equation}
 where $\lceil s \rceil$ denotes the smallest integer greater than or
 equal to $s \in \R$. 
 Combining inequalities (\ref{eq:app-1}), (\ref{eq:app-2}) and
 (\ref{eq:app-3}), and recalling that 
 $\sum_{j=1}^{\infty}a_j'\mu^j=(\sum_{n=1}^{\infty}a_n\mu^n)^k$, 
 we conclude that 
\begin{equation*}
\begin{split}
     \left(1+ \sum_{n=1}^{\infty} n a_n\right) L^k(f,\mu) 
  & \geq \sum_{i=1}^{\infty}b_i \left(
     \int_{\Gamma} d(f(e),f(\gamma)) d\mu^{ik}(\gamma)\right) \\
  & \geq \sum_{j=1}^{\infty} a_j' \left(
      \int_{\Gamma} d(f(e),f(\gamma)) d\mu^j(\gamma)\right) \\
  & = \int_{\Gamma} d(f(e),f(\gamma))
     d(\sum_{n=1}^{\infty}a_n\mu^n)^k. 
\end{split}
\end{equation*}
 Therefore we get
\begin{equation*}
 \left(1 + \sum_{n=1}^{\infty} n a_n\right)
 \frac{L^k(f,\mu)}{k} \geq 
 \frac{L^k(f,\sum_{n=1}^{\infty}a_n\mu^n)}{k}. 
\end{equation*}
 Letting $k\to \infty$ completes the proof of (\ref{eq:conv_comb}).


\end{document}